\documentclass[]{article}
\usepackage{graphicx}
\usepackage{amsfonts}
\usepackage{amsmath}
\usepackage{amssymb}
\usepackage{fancyhdr}
\usepackage{titlesec}
\usepackage{indentfirst}
\usepackage{booktabs}
\usepackage{verbatim}
\usepackage{color}
\usepackage{amsthm}
\usepackage{subfigure}

\usepackage[page,header]{appendix}
\usepackage{titletoc}

\newcommand{\rd}{\,\mathrm{d}}
\numberwithin{equation}{section}
\newtheorem{theorem}{Theorem}[section]
\newtheorem{lemma}[theorem]{Lemma}
\newtheorem{corollary}[theorem]{Corollary}
\newtheorem{proposition}[theorem]{Proposition}

\newtheorem{remark}[theorem]{Remark}

\newcommand{\rev}[1]{{#1}}

\def\thetheorem {{\arabic{section}.\arabic{theorem}}}

\def\bu{{\bf u}}
\def\bx{{\bf x}}
\def\by{{\bf y}}

\def\cC{\mathcal{C}}
\def\cT{\mathcal{T}}
\def\cI{\mathcal{I}}
\def\cS{\mathcal{S}}

\def\km{\mathfrak{m}}

\def\sgn{\textnormal{sgn}}
\def\supp{\textnormal{supp}}
\def\ini{\textnormal{in}}

\begin{document}
\title{Equilibration of aggregation-diffusion equations with weak interaction forces}
\author{Ruiwen Shu\footnote{Department of Mathematics, University of Maryland, College Park, MD 20742, USA (rshu@cscamm.umd.edu). The author's research was supported in part by NSF grants DMS-1613911, RNMS-1107444 (KI-Net) and ONR grant N00014-1812465.}}
\date{}
\maketitle

\abstract{This paper studies the large time behavior of aggregation-diffusion equations. For one spatial dimension with certain assumptions on the interaction potential, the diffusion index $m$, and the initial data, we prove the convergence to the unique steady state as time goes to infinity (equilibration), with an explicit algebraic rate. The proof is based on a uniform-in-time bound on the first moment of the density distribution, combined with an energy dissipation rate estimate. This is the first result on the equilibration of aggregation-diffusion equations for a general class of weakly confining potentials $W(r)$: those satisfying $\lim_{r\rightarrow\infty}W(r)<\infty$.
}

\section{Introduction}
In this paper we are concerned with the large time behavior of the aggregation-diffusion equation
\begin{equation}\label{eq00}
\partial_t \rho + \nabla\cdot(\rho \bu) = \Delta (\rho^m),\quad \bu(t,\bx) = -\int \nabla W(\bx-\by)\rho(t,\by)\rd{\by},
\end{equation}
where $t\in\mathbb{R}_{\ge0}$, $\bx\in\mathbb{R}^d$ are the temporal and spatial variables respectively, and $\rho(t,\bx)$ is the density distribution function of a large crowd of moving agents. The drift term $\nabla\cdot(\rho \bu)$ describes the pairwise attraction forces among agents, where a radial interaction potential $W=W(r),\,r=|\bx|$ gives rise to the drift velocity $\bu$. The attraction nature is guaranteed by the condition $W'(r)>0,\,\forall r> 0$. This term leads to the aggregation effect of the agents, i.e., they tend to get close to each other. The nonlinear diffusion term $\Delta (\rho^m),\,m\ge 1$ represents the localized repulsion among agents~\cite{Oel} or the effect of Brownian motions~\cite{Pat,KS}, and the agents tend to avoid being too crowded due to this term.

The aggregation-diffusion equation \eqref{eq00} arises naturally in the study of the collective behavior of large groups of swarming agents~\cite{MEK1,MEK2,BurCM,BCM,MCO,TBL} and the chemotaxis phenomena of bacteria~\cite{Pat,KS,JL,Hor,BDP,BCL,CC}. In the modeling of chemotaxis, one usually assumes $d=2$, $m=1$ with $W$ being the Newtonian attraction potential, and \eqref{eq00} is called the Keller-Segel model~\cite{Pat,KS}. The most interesting phenomenon of the Keller-Segel model is the critical threshold on the initial total mass for the existence of global smooth solution/finite time blow-up. 

However, when $m$ gets larger, the diffusion effect is stronger at locations with larger density, which has a stronger tendency of suppressing blow-up. In fact, it is shown in~\cite{BCC,Sug} that $m=2-2/d$ is the critical index: global wellposedness of \eqref{eq00} holds when $m>2-2/d$, and in particular, this is the case for $m\ge 2$ for any spatial dimensions, which will be assumed in the rest of this paper. Also, the H\"older regularity of the solution has been studied \rev{by~\cite{CHKK,KZ,HZ}}.

Therefore, it is natural to study \emph{the large time behavior} of \eqref{eq00}. To do this, the starting point is the formal \emph{2-Wasserstein gradient flow} structure~\cite{AGS,CMV1,CMV2} of \eqref{eq00}. Define the total energy
\begin{equation}\label{E}
E[\rho] := \frac{1}{m-1}\int \rho(\bx)^m \rd{\bx} + \frac{1}{2}\int \int W(\bx-\by)\rho(\by)\rd{\by}\rho(\bx)\rd{\bx},\quad E(t) := E[\rho(t,\cdot)],
\end{equation}
where $\rho(t,\cdot)$ is the solution to \eqref{eq00}, then $E(t)$ formally satisfies
\begin{equation}
\frac{\rd E}{\rd{t}} = -\Big\|\frac{\delta E}{\delta \rho}\Big\|^2 = -\int 
\Big|\bu(t,\bx) - \frac{m}{m-1}\nabla(\rho(t,\bx)^{m-1})\Big|^2\rho(t,\bx)\rd{\bx}.
\end{equation}
Therefore $E(t)$ is non-increasing, and a steady state $\rho_\infty$ is reached (i.e., \emph{equilibration}) if and only the energy dissipation rate is zero, i.e.,
\begin{equation}
\bu_\infty(\bx) - \frac{m}{m-1}\nabla(\rho_\infty(\bx)^{m-1}) = 0,\quad \forall \bx \in \supp \rho_\infty.
\end{equation}

There have been several works towards the large time behavior of \eqref{eq00}. \cite{Bed} shows the existence of steady state(s), by using scaling arguments. Regarding the uniqueness of steady state and equilibration,
\begin{itemize}
\item In the special case of Newtonian attraction, \cite{LY,Str,CCV} prove the uniqueness and radial symmetry\footnote{A density distribution $\rho(\bx)$ is radially-symmetric if $\rho$ is a function of $r=|\bx|$, and radially-decreasing if $\rho(r)$, as a function of $r$, is decreasing on $[0,\infty)$.} of the steady state for $m>2-2/d$. \cite{KY} generalizes these results to cases where the interaction potential is the convolution of the Newtonian potential with a radially-decreasing function, and prove the exponential equilibration of \eqref{eq00} for radially-decreasing initial data. This work is based on a comparison-principle argument: it relies on the fact that the radially-decreasing property of solution propagates in time, which is not true for general interaction potentials.
\item  For a general class of interaction potentials, \cite{CHVY} shows that every steady state has to be radially-decreasing, for interaction potentials which are no more singular than the Newtonian attraction near the origin.  This work uses the continuous Steiner symmetrization (CSS), and proves that the energy dissipation rate has to be strictly positive if $\rho$ is not radially-decreasing. It also improves~\cite{KY} to prove the equilibration for general solutions, in the case of Newtonian attraction, but without an explicit convergence rate. This idea, combined with Hardy-Littlewood-Sobolev type inequalities, has been applied in~\cite{CCH,CHMV,CCH2} to study the radial symmetry and uniqueness of steady states for Riesz potentials $W(r) = r^k/k,\, -d<k<0$ which could be more singular than Newtonian if $k<-d+2$. 
\item $ $\cite{DYY} shows the uniqueness of steady state for a general class of interaction potentials in the case $m\ge 2$, and non-uniqueness in the case $1<m<2$ for some potentials. The main approach to prove the uniqueness for $m\ge 2$ (which will be relevant to the current work) is the design of a curve connecting two radially-decreasing states, having convexity property for the energy functional.
\end{itemize}

Although a great amount of effort has been spent on the large time behavior of \eqref{eq00} in the past decade, the equilibration of \eqref{eq00} for general interaction potentials remains open. Once we know the existence and uniqueness of steady state (which is the case of $m\ge 2$ with general interaction potential), the biggest difficulty towards equilibration is \emph{tightness}, i.e., guaranteeing that no mass can escape to infinity. To be precise, one typical way (as done in \cite{CHVY}) is to pick a sequence of time spots $\{t_n\}$, and try to take a strongly convergent subsequence of $\{\rho(t_n,\cdot)\}$ (in certain norm), and prove that this subsequence has to converge to the unique steady state. This cannot work if a positive amount of mass escapes to infinity: take $d=1$ for example, if there exists a sequence $x_n$ with $\lim_{n\rightarrow\infty}x_n = \infty$ and $\int_{x_n}^{x_n+1}\rho(t_n,x)\rd{x} \ge c > 0$, then there cannot be a strongly convergent subsequence of $\{\rho(t_n,\cdot)\}$.

To this end, one can distinguish the interaction potentials into two classes:
\begin{itemize}
\item Strongly confining potentials: if $\lim_{r\rightarrow\infty} W(r) = \infty$. In this case, from the uniform bound  $E(t)\le E(0),\,\forall t\ge 0$ one can easily rule out the `escaping to infinity' situation mentioned above and gain the tightness, and therefore obtain the equilibration. \rev{For example, in \cite{CHVY} the authors prove the tightness of the solution for the case of 2D Newtonian attraction, in the sense of its uniformly bounded second moment, and then obtain the equilibration of the solution.}
\item Weakly confining potentials: if $\lim_{r\rightarrow\infty} W(r) < \infty$. In this case, a direct use of $E(t)\le E(0)$ does not give the tightness, and it is hard to obtain tightness.
\end{itemize}

The current work gives a first result on the equilibration of \eqref{eq00} for a large class of weakly confining potentials. In the case $d=1$ and under certain assumptions on the asymptotic behavior of $W'(r)$ as $r\rightarrow\infty$, we prove the equilibration of \eqref{eq00} with an explicit algebraic decay rate for $m$ large enough and certain sub-critical initial data. Our result covers the potentials with $W'(r)$ behaving like $r^{-\alpha}$ as $r\rightarrow\infty$ for any $\alpha>0$, including weakly confining potentials. Also, all our assumptions on $W'(r)$ are on its size: no structural assumptions like those in \cite{KY} are required.

The main difficulty, of course, is to obtain the tightness for such weakly confining potentials. We overcome this difficulty by proving a uniform-in-time bound of the first moment of $\rho$, via proving the finiteness $\int_0^\infty \int_{5R_1}^{6R_1}\rho(t,x)\rd{x}\rd{t}$ for $R_1$ large. This is obtained by delicate estimates of the energy dissipation rate via several newly-designed curves for density distributions: two of which are CSS and its variant, to control the non-radially-decreasing part of $\rho$ in the above integral; another one is a local compression map which reflects the tendency of local clustering, to control the radially-decreasing part. In fact, the formation of local clusters of \eqref{eq00} with $m>2$ has been numerically observed in~\cite{BCH,CCWW}. The reason behind is that, when $m$ is large, due to the difference in homogeneity, the aggregation term $\nabla\cdot(\rho\bu)$ (which is quadratic in $\rho$) is much stronger than the diffusion term $\Delta (\rho^m)$ in regions where $\rho$ is small, and the aggregation term, like that in the consensus model~\cite{MT}, leads to local clustering.

Based on the tightness, we also manage to obtain an explicit equilibration rate, by refining the energy dissipation rate estimates in \cite{CHVY} and \cite{DYY}, and connecting them via a perturbative argument. Although the current rate we obtain is algebraic, these estimates could give an exponential rate, provided a stronger result on tightness (for example, a uniform-in-time bound on the size of support of $\rho$). We believe that such estimates on the explicit convergence rate is essential if one wants to study the second-order analog of \eqref{eq00}, i.e., the isentropic Euler equations with a power-law pressure with pairwise attraction force and certain velocity damping mechanism (linear damping or Cucker-Smale alignment~\cite{CS1,CS2,HT}, for example).

This paper is organized as follows: in section 2 we introduce some assumptions and state the main result, namely, Theorem \ref{thm_main}. In section 3 we give the definition of some basic notations and give an outline of the proof, including the important intermediate results. In section 4 we introduce a few variants of the continuous Steiner symmetrization, which are crucial to many parts of the proof of the main theorem. In section 5 we prove Theorem \ref{thm_main1} concerning the tightness of the solution. In section 6 we prove Theorem \ref{thm_main2} which gives the quantitative estimate of the energy decay rate. In section 7 we finish the proof of the main result Theorem \ref{thm_main}. The paper is concluded in section 8.

\section{The main result}

From now on, we will focus on the case of one spatial dimension ($d=1$) with $m>2$, for which \eqref{eq00} can be written as
\begin{equation}\label{eq0}
\partial_t \rho + \partial_x (\rho u) = \partial_{xx} (\rho^m),\quad u(t,x) = -\int W'(x-y)\rho(t,y)\rd{y}.
\end{equation}
The initial data of \eqref{eq0} is denoted as $\rho(0,\cdot) = \rho_{\ini}(\cdot)$.

We propose the following assumptions:
\begin{itemize}
\item {\bf (A1)} $W=W(x)$ is an attractive potential: 
\begin{equation}
W(0) = 0,\quad W(x)=W(-x),\quad W'(x)>0,\,\forall x>0 .
\end{equation}
\item {\bf (A2)} $W'$ satisfies upper and lower bounds:
\begin{equation}
\|W'\|_{L^\infty(0,\infty)} < \infty,\quad \lambda(x) := c_\alpha(1+|x|)^{-\alpha} \le W'(x) \le   \Lambda(x) :=C_\beta(1+|x|)^{-\beta} ,\quad \forall x>0,
\end{equation}
for some $\alpha\ge \beta>0$, $c_\alpha,C_\beta>0$.
\item {\bf (A3)} $\|W'''\|_{L^\infty(0,\infty)} < \infty$.
\item {\bf (A4)} $m>\max\{2,\alpha\}$.
\item {\bf (A5)} The initial data is radially-symmetric: $\rho_{\ini}(x) = \rho_{\ini}(-x)$, non-negative, compactly supported, with total mass 1: $\int\rho_{\ini}(x)\rd{x}=1$.
\item {\bf (A6)} The initial data is \emph{sub-critical}, in the sense that
\begin{equation}\label{A6}
E(0) < \frac{1}{4}\lim_{x\rightarrow\infty}W(x) + 2E[\rho_{\infty,1/2}],
\end{equation}
where $\rho_{\infty,1/2}$ denotes the unique steady state determined by the same $W$ and $m$ with total mass $1/2$ (whose existence and uniqueness are guaranteed by the results in~\cite{Bed,CHVY,DYY} with the previous assumptions).
\end{itemize}

Then we state the main result:
\begin{theorem}\label{thm_main}
Under the assumptions {\bf (A1)}-{\bf (A6)} on $W$, $m$ and $\rho_{\ini}$, the solution to \eqref{eq0}, denoted as $\rho(t,\cdot)$, satisfies
\begin{equation}\label{thm_main_1}
E(t) - E_\infty \le C(1+t)^{-1/\gamma},\quad \gamma=2+\alpha+4\frac{\alpha^2}{\beta}+\frac{\alpha^3}{\beta^2},
\end{equation}
for some $C>0$, where $E(t)$ is as defined in \eqref{E}, and $E_\infty=E[\rho_\infty]$ with $\rho_\infty$ being the unique steady state of \eqref{eq0} with total mass 1. 
\end{theorem}

\rev{\begin{remark}
This result shows the equilibration of \eqref{eq0} with an explicit algebraic decay rate, in terms of the energy. For the convergence of the solution $\rho(t,\cdot)$ to $\rho_\infty$, we are able to show that a subsequence $\rho(t_n,\cdot)$ converges to $\rho_\infty$ in $H^{-1}$ as $n\rightarrow\infty$, see Appendix for a precise and quantitative version. However, it is still open to show such a convergence without taking subsequences. 
\end{remark}}

Before discussing the proof, we first give a few remarks on the assumptions:
\begin{enumerate}
\item The lower bound with $\lambda(x)$ in {\bf (A2)} dictates the decay rate of $W'(x)$ as $x$ gets large. In particular, if $\alpha> 2$, this lower bound is consistent with weakly confining potentials, i.e., those $W$ satisfying $\lim_{x\rightarrow\infty} W(x) < \infty$. For example, $W(x)=-\frac{1}{(1+|x|)^{\alpha-1}}+1$ is a weakly confining potential satisfying {\bf (A1)} and {\bf (A2)} with parameter $\alpha$, if $\alpha>2$.

\item {\bf (A2)} requires $W'(x)$ bounded below when $x>0$ is close to 0, and thus $W(x)\sim c|x|$ for some $c>0$, i.e., behaves like the 1D Newtonian attraction near 0. In particular, it does not allow smooth $W$ (for which $\lim_{x\rightarrow 0}W'(x)=W'(0)=0$). It may require more work to allow such degeneracy of $W'$ near 0 (which could make the confining effect weaker), and this is out of the scope of this paper.

\item The upper bound with $\Lambda(x)$ in {\bf (A2)} can be viewed as a condition number on $W'$: besides knowing that it cannot be too small for large $x$, we also need to know that it cannot be too large. This is critical in the proof of the tightness (Theorem \ref{thm_main1}) where we need to estimate the influence from mass which are very far away from the center. We remark that the proof also works under a weaker assumption $\lim_{x\rightarrow\infty}\Lambda(x)=0$, but in this case one could only obtain the equilibration with no rate.

\item The assumption {\bf (A3)} is only used in the energy dissipation rate estimate (Lemma \ref{lem_RDper}) and not required in the proof of the tightness.

\item In {\bf (A4)}, the condition $m>2$ is to guarantee the uniqueness of steady state, see~\cite{DYY}. The condition $m>\alpha$ is critical in the tightness of the radially-decreasing part (see \eqref{rhost} for the definition), c.f. Proposition \ref{prop_rd}. Intuitively, the radially-decreasing part, which is small for large $x$ (to be precise, at most $O(1/x)$), has weaker dissipation effect for larger $m$, and this makes the solution having less tendency of escaping to infinity. This competes with the $\lambda(x)$ lower bound in {\bf (A2)}, which says that the attraction force which drives mass back to the center, cannot be weaker than $\lambda(x)$ for large $x$. To the best knowledge of the author, the competition between these two mechanisms has not been studied before.

\item The total-mass-one condition in {\bf (A5)} is not restrictive: one can always reduce to this case by scaling arguments.

\item The purpose of  {\bf (A6)} and the symmetry assumption in {\bf (A5)} is to keep at least a \emph{positive} amount of mass not escaping to infinity (Lemma \ref{lem_ct}). In particular, {\bf (A6)} avoids the possibility that two bulks of mass with the same shape escape to $\infty$ and $-\infty$ respectively. In the proof of tightness, this part of the mass near the center will serve as a source to attract other mass towards the center. It is still open whether the radial-symmetry and sub-critical condition of the initial data are necessary to conclude the equilibration. Also notice that {\bf (A6)} is automatically satisfied for strongly confining potentials.

\item For initial data satisfying {\bf (A5)}, Theorem \ref{thm_main} also applies if $\rho(t_0,\cdot)$ satisfies \eqref{A6} with $E(0)$ replaced by $E(t_0)$, for any $t_0>0$. In particular, if one observes $E(t)$ falling below the critical threshold in a numerical simulation with a certified error bound, then one can conclude the equilibration.

\rev{\item {\bf (A3)} implies that $W\in W^{3,\infty}(0,\infty)$. However, when viewed as a function defined on $\mathbb{R}$, the best one can conclude is that $W\in W^{1,\infty}(\mathbb{R})$, since $W(x)$ typically behaves like $C|x|$ near $x=0$. {\bf (A6)} clearly implies that $\rho_{\ini}\in L^1(\mathbb{R})$. Since \eqref{eq0} preserves the nonnegative property of $\rho$ as well as the total mass, the solution $\rho$ lies in the space $L^\infty([0,\infty),L^1(\mathbb{R}))$.
}

\end{enumerate}

\section{Notations and outline of the proof}

In this section we state some basic notations which will be used throughout this paper, and give an outline of the proof.

\subsection{Notations}

\begin{itemize}
\item Throughout this paper, the assumptions {\bf (A1)}-{\bf (A6)} are always assumed in the statement of all intermediate results.

\item $\rho=\rho(x)$, with possible $t$-dependence stated as below, always denotes a density distribution function which is radially-symmetric, non-negative, compactly supported, with total mass 1. $\rho(t,\cdot)$ always denotes the solution to \eqref{eq0} with {\bf (A1)}-{\bf (A6)} satisfied. 

\item $\rho_t(\cdot)$ always denotes a family (also called \emph{curve}) of density distributions, depending on the parameter $t$, for $0\le t\le t_1$ with $t_1>0$. Usually $\rho_0$ is taken as the initial data $\rho_{\ini}$ and the curve is constructed from $\rho_0$, and generally it does not coincide with $\rho(t,\cdot)$. $\rho_t$ may refer to different curves in different context.

\item We write the total energy $E[\rho]$ into the internal part $\cS[\rho]$ and interaction part $\cI[\rho]$:
\begin{equation}
E[\rho] = \cS[\rho] + \cI[\rho],\quad \cS[\rho]= \frac{1}{m-1}\int \rho^m \rd{x} ,\quad \cI[\rho] = \frac{1}{2}\int \int W(x-y)\rho(y)\rd{y}\rho(x)\rd{x} .
\end{equation}
We further define a bilinear form for the interaction energy:
\begin{equation}
\cI[\rho_1,\rho_2] =  \int \int W(x-y)\rho_2(y)\rd{y}\rho_1(x)\rd{x},\quad \cI[\rho] = \frac{1}{2}\cI[\rho,\rho].
\end{equation}

\item We will denote $\rho^\#$ as the Steiner symmetrization of $\rho$, defined \rev{as
\begin{equation}
\rho^\sharp(x) := \inf\{s\in[0,\infty]: |\{x:\rho(x) > s\}|\le 2|x|\}.
\end{equation}}
It is clear that $\rho^\#$ is radially-decreasing, compactly supported, with total mass 1, and $\cS[\rho] = \cS[\rho^\#]$.

\item The radially-decreasing part $\rho^*$ of $\rho$ is defined as the radially-decreasing distribution with largest total mass such that $\rho^*(x)\le \rho(x),\,\forall x$. It is explicitly given by
\begin{equation}\label{rhost}
\rho^*(x) = \min_{0\le y\le x}\rho(y).
\end{equation}
Then we decompose $\rho(x)$ into
\begin{equation}\label{decomp}
\rho(x) = \rho^*(x) + \mu(x),\quad \mu(x) = \mu(x)\chi_{(0,\infty)}(x) + \mu(x)\chi_{(-\infty,0)}(x) =: \mu_+(x) + \mu_-(x),
\end{equation}
where $\mu(x)=\mu[\rho](x)$ is the non-radially-decreasing part of $\rho(x)$ (its dependence on $\rho$ will be omitted when it is clear from the context). $\mu$ is radially-symmetric, non-negative and compactly supported.

\item $C$ and $c$ always refer to positive constants which may depend on \\
$\|W'\|_{L^\infty},\|W'''\|_{L^\infty}, \alpha, c_\alpha, \beta, C_\beta, m, \rho_{\ini}$, as appeared in the assumptions {\bf (A1)}-{\bf (A6)}, and they may differ from line to line. Usually $C$ refers to a large constant and $c$ refers to a small constant.

\end{itemize}

\subsection{Outline of the proof}

We give an outline of the proof of Theorem \ref{thm_main}. The starting point is the following lemma, which reduces the estimate for energy dissipation rate into finding a suitable curve $\rho_t$ with large energy decay rate and small cost. It is indeed a consequence of the 2-Wasserstein gradient flow structure of \eqref{eq0}, \rev{which has been justified rigorously in \cite{AGS}, c.f. Theorems 11.2.8 and 11.3.2 therein.}
\begin{lemma}\label{lem_basic}
Let $\rho_t(x),\,0\le t \le t_1,\,t_1>0$ satisfy $\rho_0 = \rho_{\ini}$ and
\begin{equation}
\partial_t \rho_t + \partial_x(\rho_t v_t) = 0,
\end{equation}
for some velocity field $v_t(x)$ with $\int v_t^2\rho_t\rd{x} < \infty$. Assume 
\begin{equation}
\frac{\rd}{\rd{t}}\Big|_{t=0} E[\rho_t] < 0.
\end{equation}
Then the solution $\rho(t,x)$ to \eqref{eq0} satisfies
\begin{equation}
\frac{\rd}{\rd{t}}\Big|_{t=0} E[\rho(t,\cdot)] \le -\frac{(\frac{\rd}{\rd{t}} E[\rho_t(\cdot)])^2}{\int v_t^2\rho_t\rd{x}}\Big|_{t=0}.
\end{equation}
\end{lemma}
Notice that the denominator $\int v_t^2\rho_t\rd{x}$ represents the infinitesimal cost of the curve $\rho_t$ in the sense of the 2-Wasserstein metric, in align with the Benamou-Brenier formulation~\cite{BB}.

\begin{proof}
Write $v(t,x) = u(t,x) - \frac{m}{m-1}\partial_x(\rho(t,x)^{m-1})$ as the velocity field of \eqref{eq0}. Then the 2-Wasserstein gradient flow structure reads
\begin{equation}
\partial_t \rho(t,x) + \partial_x (\rho(t,x) v(t,x)) = 0,\quad \frac{\rd}{\rd{t}}\Big|_{t=0}E[\rho(t,\cdot)] = -\int v(0,x)^2\rho_{\ini}(x)\rd{x}.
\end{equation}
We also have
\begin{equation}
\frac{\rd}{\rd{t}}\Big|_{t=0}E[\rho_t] = -\int v_0(x)v(0,x)\rho_{\ini}(x)\rd{x} \ge -\left( \int v_0(x)^2\rho_{\ini}(x)\rd{x} \cdot\int v(0,x)^2\rho_{\ini}(x)\rd{x}\right)^{1/2},
\end{equation}
in case the LHS is negative, and the conclusion follows.

\end{proof}

This lemma can also be applied at any time $t$ rather than the case at the initial time ($t=0$) as stated above. Combined with various constructions of the curve $\rho_t$, it allows us to bound the energy dissipation rate from below for certain types of density distributions $\rho(t,\cdot)$. 

The proof of Theorem \ref{thm_main} has two main steps: 
\begin{enumerate}
\item Give a uniform-in-time estimate of the first moment
\begin{equation}
\km_1(t) := \int |x|\rho(t,x)\rd{x},
\end{equation}
see Theorem \ref{thm_main1}. This gives the necessary tightness for the equilibration.
\item Give a quantitative estimate of the energy dissipation rate, in terms of the size of support $R(t) = \max\{|x|:x\in\supp\rho(t,\cdot)\}$:
\begin{equation}
\frac{\rd}{\rd{t}} E(t) \le -c(R) (E(t)-E[\rho_\infty]),
\end{equation}
see Theorem \ref{thm_main2}. The uniform moment estimate basically says that the size of support is uniformly bounded, up to a small tail, and a perturbed version of Theorem \ref{thm_main2} gives the algebraic convergence rate (c.f. section 7).
\end{enumerate}

In the rest of this section, we outline the important intermediate steps for the uniform moment estimate and the quantitative dissipation rate estimate. 

\subsection{Uniform estimate of the first moment}

We formulate the uniform estimate of the first moment as follows:
\begin{theorem}\label{thm_main1}
We have
\begin{equation}
\km_1(t) := \int |x|\rho(t,x)\rd{x} \le C,
\end{equation}
for all $t\ge 0$.
\end{theorem}

We prove this theorem as follows, with the aid of intermediate results stated in this section:

\begin{proof}
We start by defining an alternative of $\km_1$:
\begin{equation}\label{mphi}
\km_\phi(t):= \int \phi(x)\rho(t,x)\rd{x},\quad \phi(x) = \left\{\begin{split}
& 0,\quad |x|\le 5R_1 \\
& \frac{1}{2}(|x|-5R_1)^2,\quad 5R_1\le |x|\le 6R_1 \\
& R_1|x|-\frac{11}{2}R_1^2,\quad |x|\ge 6R_1 \\
\end{split}\right.,
\end{equation}
where $\phi$ is a convex function, designed to satisfy $\phi''=\chi_{[-6R_1,-5R_1]\cup [5R_1,6R_1]}$, and $R_1>0$ is a large constant to be chosen. It is clear that the uniform-in-time bound of $\km_1$ is equivalent to that of $\km_\phi$, for any choice of $R_1$. Then it is important to observe that
\begin{equation}\begin{split}
& \frac{\rd}{\rd{t}}\int \phi(x)\rho(t,x) \rd{x} \\
= & \int \phi(x)(-\partial_x(\rho u) + \partial_{xx}(\rho^m)) \rd{x} 
=  \int \phi'(x)\rho u \rd{x} + \int \phi''(x)\rho^m\rd{x}\\
= & -\int \phi'(x)\rho(t,x)\int W'(x-y)\rho(t,y)\rd{y}  \rd{x} + \int_{5R_1\le |x|\le 6R_1} \rho(t,x)^m\rd{x}\\
= & -\frac{1}{2}\int\int  W'(x-y)(\phi'(x)-\phi'(y))\rho(t,y)\rd{y}\rho(t,x)  \rd{x} + \int_{5R_1\le |x|\le 6R_1} \rho(t,x)^m\rd{x}\\
\le & \int_{5R_1\le |x|\le 6R_1} \rho(t,x)^m\rd{x}, \\
\end{split}\end{equation}
where the last equality symmetrizes $x$ and $y$ by using $W'(-x)=-W'(x)$, and the last inequality uses the facts that both $W'(x-y)$ and $\phi'(x)-\phi'(y)$ have the same sign as $x-y$ (a consequence of the attractive nature of $W$ and the convexity of $\phi$).
Then Theorem \ref{thm_main1} follows from the finiteness of $\int_0^\infty\int_{5R_1\le |x|\le 6R_1} \rho(t,x)^m\rd{x}\rd{t}$, which is guaranteed by the uniform $L^\infty$ bound of $\rho(t,x)$  (Lemma \ref{lem_reg1}) and the following two propositions, for the non-radially-decreasing part and radially-decreasing part respectively.
\end{proof}

\begin{proposition}\label{prop_nrd}
For $R_1$ sufficiently large and $R_2>2R_1$, we have
\begin{equation}
\int_0^\infty \int_{[2R_1,R_2]} \mu(t,x)\rd{x} \rd{t} \le CR_2^{4\alpha+\frac{\alpha^2}{\beta}}.
\end{equation} 
\end{proposition}

\begin{proposition}\label{prop_rd}
For $R_1$ sufficiently large, we have
\begin{equation}
\int_0^\infty \int_{[5R_1,6R_1]} \rho^*(t,x)\rd{x} \rd{t} \le CR_1^{\alpha+4\frac{\alpha^2}{\beta}+\frac{\alpha^3}{\beta^2}}.
\end{equation} 
\end{proposition}

The proof of these two propositions will be given in section \ref{sec_tight}, with the aid of the variants of CSS curves discussed in section \ref{sec_css}.

We finally remark that the critical threshold assumption {\bf (A6)} is used in the proof of these two propositions, by the following lemma:
\begin{lemma}\label{lem_ct}
For $R_1$ sufficiently large, there exists $c_\rho>0$ such that
\begin{equation}\label{R1}
\int_{[0,R_1]}\rho(t,x)\rd{x} \ge c_\rho,\quad \forall t\ge 0.
\end{equation}
\end{lemma}
\begin{proof}
We first show that the function $s\mapsto E[\rho_{\infty,s}]$ (where $\rho_{\infty,s}$ is the unique steady state with total mass $s$) defined on $\mathbb{R}_{>0}$ is continuous and increasing in $s$. In fact, for any $s>0$ and $|\delta|<1$,
\begin{equation}
E[\rho_{\infty,(1+\delta)s}] \le E[(1+\delta)\rho_{\infty,s}] \le \max\{(1+\delta)^m,(1+\delta)^2\}E[\rho_{\infty,s}],
\end{equation}
since the density distribution $(1+\delta)\rho_{\infty,s}$ has total mass $(1+\delta)s$. Then the claimed continuity and monotonicity follows.

By {\bf (A6)}, if $R_1$ is sufficiently large, there holds
\begin{equation}\label{A6_1}
E(0) < (\frac{1}{2}-c_\rho)^2W(2R_1) + 2E[\rho_{\infty,1/2-c_\rho}],
\end{equation}
for some $c_\rho>0$. 

Assume on the contrary that $\int_{[0,R_1]}\rho(t,x)\rd{x} < c_\rho$ for some $t$. We decompose $\rho=\rho(t,\cdot)$ into
\begin{equation}
\rho = \rho \chi_{[-R_1,R_1]} + \rho \chi_{(-\infty,-R_1]} + \rho \chi_{[R_1,\infty)},
\end{equation}
and then by the symmetry of $\rho$, we have 
\begin{equation}
\int \rho \chi_{[R_1,\infty)}\rd{x} = \int \rho \chi_{(-\infty,-R_1]}\rd{x} > \frac{1}{2}-c_\rho.
\end{equation}
By the bi-linearity of the interaction energy,
\begin{equation}\begin{split}
E[\rho] = & E[\rho \chi_{[-R_1,R_1]}] + E[\rho \chi_{(-\infty,-R_1]}] + E[\rho \chi_{[R_1,\infty)}] \\
& + \cI[\rho \chi_{(-\infty,-R_1]},\rho \chi_{[R_1,\infty)}] + \cI[\rho \chi_{[-R_1,R_1]},\rho \chi_{(-\infty,-R_1]}] + \cI[\rho \chi_{[-R_1,R_1]},\rho \chi_{[R_1,\infty)}] \\
\ge & 0 + E[\rho_{\infty,1/2-c_\rho}]+ E[\rho_{\infty,1/2-c_\rho}] + \int_{R_1}^\infty \int_{-\infty}^{-R_1}W(x-y)\rho(y)\rd{y}\rho(x)\rd{x} + 0 + 0 \\
\ge & 2E[\rho_{\infty,1/2-c_\rho}] + (\frac{1}{2}-c_\rho)^2 W(2R_1) ,\\
\end{split}\end{equation}
where the last inequality uses the fact that $|x-y|\ge 2R_1$ in the integrand and thus $W(x-y)\ge W(2R_1)$. This and \eqref{A6_1} contradict the fact that the total energy is decreasing: $E[\rho(t,\cdot)] \le E(0)$.

\end{proof}

\subsection{Quantitative estimate of the energy dissipation rate}

We formulate the quantitative estimate of the energy dissipation rate as follows:
\begin{theorem}\label{thm_main2}
If $\supp\rho(t_0,\cdot) \subset [-R,R]$ for some $t_0\ge 0$ and $R>0$, then we have
\rev{\begin{equation}
\frac{\rd}{\rd{t}}\Big|_{t=t_0} E(t) \le - c\frac{\lambda(R)^3}{R^{8/3}}(E(t_0) - E_\infty).
\end{equation}}
\end{theorem}

This result is locally optimal, in the following sense: suppose one knows $\supp\rho(t,\cdot)\subset[-R,R]$ for some $R>0$ and all $t\ge 0$, then Theorem \ref{thm_main2} implies the exponential decay of $E(t)$ to $E_\infty$. In other words, apart from tightness issues, this is the best estimate one can obtain for the energy dissipation rate.

To prove Theorem \ref{thm_main2}, we start by the following proposition, which describes the energy dissipation rate generated from the non-radially-decreasing part: this is a quantitative version of Proposition 2.15 of \cite{CHVY} (in its 1D version).
\begin{proposition}\label{prop_RCSS}
If $\supp\rho(t_0,\cdot) \subset [-R,R]$, then we have 
\begin{equation}\label{prop_RCSS_1}
\frac{\rd}{\rd{t}}\Big|_{t=t_0}E(t) \le -c \frac{\lambda(R)^2}{R^2} (E(t_0)-E[\rho^\#(t_0,\cdot)]).
\end{equation}
\end{proposition}

Proposition \ref{prop_RCSS} gives a lower bound on the energy dissipation rate for $\rho$ which is \emph{not} radially-decreasing, i.e., $\rho(t,\cdot)\ne \rho^\#(t,\cdot)$, but it is useless for radially-decreasing distributions. To deal with this difficulty, we first give a quantitative version of Theorem 2.6 of \cite{DYY} (in its 1D version):
\begin{proposition}\label{prop_RD}
If $\rho(t_0,x)$ is radially-decreasing at some $t_0$ and $\supp\rho(t_0,\cdot) \subset [-R,R]$, then the solution to \eqref{eq0} satisfies
\rev{\begin{equation}
\frac{\rd}{\rd{t}}\Big|_{t=t_0}E(t) \le - c\frac{\lambda(R)}{R^2}(E(t_0)-E_\infty).
\end{equation}}
\end{proposition}

Proposition \ref{prop_RD} gives the decay rate for radially-decreasing distributions. To obtain an effective energy decay rate estimate, we still need to deal with one case: when $\rho(t,\cdot)$ is close to being radially-decreasing but not radially-decreasing. In this case, Proposition \ref{prop_RCSS} only gives tiny amount of energy decay rate, but Proposition \ref{prop_RD} does not apply. To handle this difficulty, we will use a perturbed version of the proof of Proposition \ref{prop_RD} (c.f. Lemma \ref{lem_RDper}) and finish the proof of Theorem \ref{thm_main2}.


\subsection{Some regularity lemmas}

Before we go to the details of the proof, we state a few lemmas on the regularity of the solution $\rho(t,\cdot)$:

The following Lemma is a direct consequence of Theorem 3.1 of \cite{KZ}:
\begin{lemma}\label{lem_reg1}
The solution $\rho(t,\cdot)$ to \eqref{eq0} satisfies
\begin{equation}
\|\rho\|_{L^\infty((0,\infty)\times \mathbb{R})} \le C.
\end{equation}
\end{lemma}

Next we prove the following lemma:
\begin{lemma}\label{lem_reg2}
The unique steady state $\rho_\infty$ satisfies
\begin{equation}
\|\rho_\infty\|_{L^\infty} + |\partial_{xx}\rho_\infty(0) | < C.
\end{equation}
\end{lemma}
\begin{proof}
It is straightforward to check that $\|\rho_\infty\|_{L^\infty} \le C$. Then, $\rho_\infty$ satisfies the steady state equation
\begin{equation}
\partial_{xx}(\rho_\infty^m) = \partial_x (\rho_\infty (W'*\rho_\infty)),
\end{equation}
which implies
\begin{equation}
m\rho_\infty^{m-1}\partial_{xx}\rho_\infty = \partial_x\rho_\infty (W'*\rho_\infty) + \rho_\infty (W'*\partial_x\rho_\infty) - m(m-1)\rho_\infty^{m-2}(\partial_x\rho)^2.
\end{equation}
When evaluating at $x=0$, we first notice that $\partial_x\rho_\infty(0)=0$ by symmetry. Furthermore, 
\begin{equation}
\|W'*\partial_x\rho_\infty\|_{L^\infty} \le \|W'\|_{L^\infty}\|\partial_x\rho_\infty\|_{L^1}\le C,
\end{equation}
since $\|\partial_x\rho_\infty\|_{L^1} = 2\int_0^\infty (-\partial_x\rho_\infty)\rd{x} = 2\rho_\infty(0)$ by the radially-decreasing property of $\rho_\infty$. Also notice that $\rho_\infty(0)>0$ since $\rho_\infty$ is radially-decreasing and has total mass 1. These facts  imply $|\partial_{xx}\rho_\infty(0)| \le C$.
\end{proof}

\section{Variants of continuous Steiner symmetrization}\label{sec_css}

Continuous Steiner symmetrization (CSS) was introduced in~\cite{CHVY} to prove that the energy dissipation rate is positive if $\rho(t,\cdot)$ is not radially-decreasing. In its 1D version, It is a curve $\rho_t$ of density distributions which moves the non-radially-decreasing part of $\rho_0$ towards the origin, see section \ref{sec_CSS1} for the precise definition. In this section we discuss the original CSS and two variants, and estimate the energy change associated to them.

\subsection{$h$-representation of density distributions} 

We first introduce the \emph{layer cake decomposition} to represent a density distribution $\rho(x)$:
\begin{equation}
\rho(x) = \int_0^\infty \chi_{\cC[\rho](h)}(x)\rd{h},\quad \cC[\rho](h) = \{x:\rho(x) > h\},
\end{equation}
where $\cC[\rho](h)$ is the $h$-super-level set of $\rho$. On the other hand, given a given set-valued function $\cC(h): [0,\infty)\rightarrow \{\text{symmetric measurable subsets of }\mathbb{R}\}$, which is \emph{admissible} in the sense that
\begin{equation}
\cC(h_1)\subset \cC(h_2),\quad \forall h_1\ge h_2,
\end{equation}
we can reconstruct a density distribution by
\begin{equation}
\rho[\cC](x) = \int_0^\infty \chi_{\cC(h)}(x)\rd{h}.
\end{equation}
The following proposition shows the consistency of the above decomposition and construction, up to measure-zero sets.
\begin{proposition}
For admissible $\cC(h)$, one has $\cC[\rho[\cC]](h) \subset \cC(h)$ for every $h\in [0,\infty)$. Furthermore, $|\cC(h)\backslash \cC[\rho[\cC]](h)| = 0$ for almost every $h\in [0,\infty)$.
\end{proposition}
\begin{proof}
By definition, $\rho[\cC](x) = \sup_h \{x\in\cC(h)\}$. Therefore, for any $h_1\in[0,\infty)$, if $x\in \cC[\rho[\cC]](h_1)$ then $\sup_h \{x\in\cC(h)\} > h_1$ which implies that there exists $h_2>h_1$ such that $x\in \cC(h_2)$. Therefore $x\in \cC(h_1)$, which implies the first claim.

Next, if $x\notin \cC[\rho[\cC]](h_1)$ then $\sup_h \{x\in\cC(h)\} \le h_1$, which implies $x\not\in \cC(h_1+\epsilon)$ for any $\epsilon>0$. Therefore, for any $x$, there exists at most one value of $h$ such that $x\in  \cC(h)\backslash \cC[\rho[\cC]](h)$. Therefore the second claim follows from Fubini theorem.
\end{proof}

For an open set $\cC(h)\subset\mathbb{R}$ (for a fixed $h$), we define $I_0(h)=(-r_0(h),r_0(h))$ as the unique maximal interval in $\cC(h)$ containing 0 (if there exists such an interval, otherwise $I_0(h)=\emptyset$). We always write $\cC(h) = \bigcup_{j}I_j(h)$ as a (finite or countable) union of disjoint open intervals\footnote{In the rest of this paper, we may not be precise about whether such intervals are open or closed: this does not affect $\rho[\cC]$ in the a.e. sense.}, with $I_j\subset (0,\infty)$ for $j>0$, and $I_{-j} = -I_j$. We always use the notation $I_j=(c_j-r_j,c_j+r_j)$ when the underlying $I_j$ is clear. See Figure \ref{fig1} as illustration.

Without further explanations, we always assume to write $\cC(h) = \bigcup_{j}I_j(h)$ with $I_j$ disjoint. However, sometimes an interval $I_j\in \cC(h),\,j\ne 0$ may be assumed to be cut into smaller pieces when necessary, for example $I_j=(c_j-r_j,c_j+r_j)$ is replaced by two intervals $(c_j-r_j,X)$ and $(X,c_j+r_j)$ for some $X\in I_j$. A finite number of such operations for each $h$ only modify $\cC(h)$ at a finite number of points, and thus do not change $\rho[\cC]$.

Once every interval $I_j$ is cut at some $X$, then the point $X$ will not appear as an interior point of any $I_j$.

\begin{figure}
\begin{center}
  \includegraphics[width=.8\linewidth]{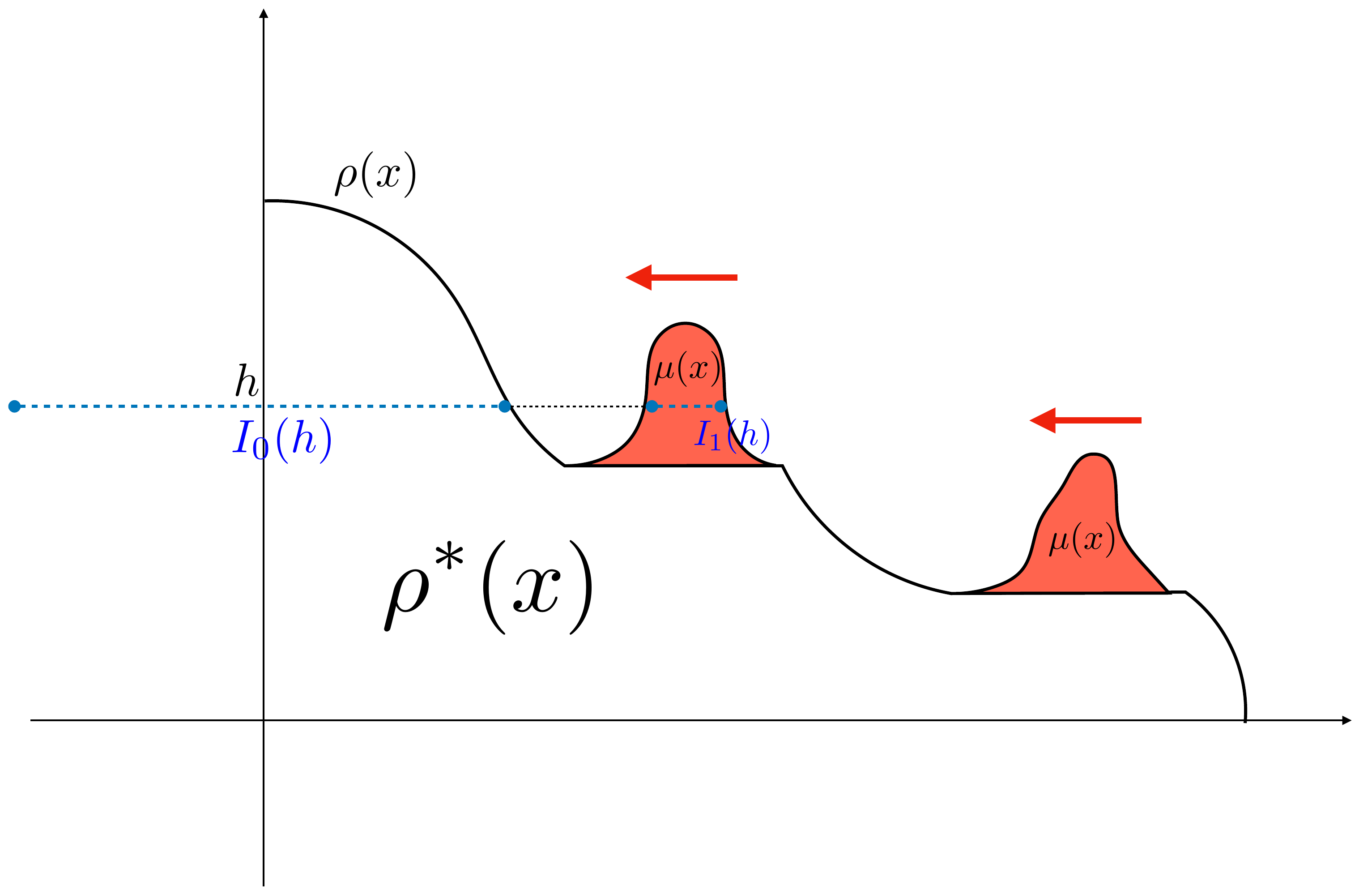}  
  \caption{Illustration of $\rho^*(x)$ (the white region under the curve), $\mu(x)$ (the red regions), $\cC(h)=\bigcup_j I_j(h)$ (the blue dashed segments), and the CSS1 curve (red arrows).}
\label{fig1}
\end{center}
\end{figure}

It is clear that
\begin{equation}
\rho^*(x) = \int_0^\infty \chi_{I_0(h)}(x)\rd{h},\quad \mu(x) = \int_0^\infty \sum_{j\ne 0}\chi_{I_j(h)}(x)\rd{h}.
\end{equation}
\rev{if $I_0(h),I_j(h)$ are the intervals coming from $\cC[\rho](h)$.}

The following lemma describes the internal energy for a density distribution given by $\cC(h)$:
\begin{lemma}\label{lem_Phi}
For any $\cC(h)$ and smooth convex function $\Phi(x)$ defined on $[0,\infty)$ with $\Phi(0)=0$, there holds
\begin{equation}
\int \Phi(\rho[\cC](x)) \rd{x} \le \int_0^\infty \Phi'(h)|\cC(h)|\rd{h}.
\end{equation}
If $\Phi$ is strictly convex, then equality holds if and only if $\cC$ is admissible.
\end{lemma}
\begin{proof}
\begin{equation}\begin{split}
\int \Phi(\rho[\cC](x)) \rd{x}  = & \int \int_0^{\rho[\cC](x)}\Phi'(h)\rd{h} \rd{x} \le \int \int_{x\in \cC(h)} \Phi'(h)\rd{h} \rd{x} \\
= & \int_0^\infty \int_{x\in \cC(h)} \Phi'(h)\rd{x} \rd{h} = \int_0^\infty \Phi'(h)|\cC(h)|\rd{h},
\end{split}\end{equation}
where the inequality uses the fact that $|\{h:x\in \cC(h)\}| = \rho[\cC](x)$ and $\Phi'$ is an increasing function. If $\Phi'$ is strictly increasing, then equality holds only when $\{h:x\in \cC(h)\} = [0,\rho[\cC](x)]$ for every $x$, which implies $\cC(h)$ is admissible.
\end{proof}

This implies that for the CSS we will define in the following subsections, the internal energy is always decreasing, by the following:
\begin{corollary}\label{cor_S}
If $\cC(h)$ is admissible, and $\tilde{\cC}(h)$ satisfies
\begin{equation}\label{cor_S_1}
|\tilde{\cC}(h)| = |\cC(h)|,
\end{equation}
then
\begin{equation}\label{cor_S_2}
\int \Phi(\rho[\tilde{\cC}](x)) \rd{x} \le \int \Phi(\rho[\cC](x)) \rd{x} .
\end{equation}
\end{corollary}

\begin{remark}
We will consider several curves $\rho_t$ for a given $\rho_0$, defined by $\rho_t = \rho[\cC_t]$ for some well-designed $\cC_t$ with $\cC_0=\cC[\rho_0]$. To apply Corollary \ref{cor_S}, we take $\cC=\cC[\rho_0]$ which is clearly admissible, and $\tilde{\cC} = \cC_t,\,t>0$. In fact, in all the applications in this paper, $\tilde{\cC}=\cC_t$ is also admissible and the equality in \eqref{cor_S_2} holds, but we still want to formulate Corollary \ref{cor_S} in the general form as stated. We believe Corollary \ref{cor_S} is comparable to the entropy condition in the study of hyperbolic equations, and this comparison deserves further investigation.
\end{remark}

\subsection{Basic lemmas for CSS}

Before we introduce the variants of CSS, we first give two lemmas which will be useful for the energy decay estimates for CSS.

The following lemma gives the interaction energy decay rate when a point mass is moving towards the center of the characteristic function of a symmetric interval:
\begin{lemma}\label{lem_CSSb}
Let $r,x>0$. Then
\begin{equation}
\frac{\rd}{\rd{t}}\Big|_{t=0}\cI[\chi_{[-r,r]},\delta_{x-t}] \le -2\lambda(r+x)\min\{r,x\},
\end{equation}
where $\delta_x$ denotes the Dirac delta centered at $x$.
\end{lemma}
\begin{proof}
\begin{equation}\begin{split}
\frac{\rd}{\rd{t}}\Big|_{t=0}\cI[\chi_{[-r,r]},\delta_{x-t}] = & \frac{\rd}{\rd{t}}\Big|_{t=0}\int_{-r}^r W(x-t-y)\rd{y}
= -\int_{-r}^r W'(x-y)\rd{y}.\\
\end{split}\end{equation}
If $x>r$ then $x-y>0$ always holds, and it follows that
\begin{equation}
\frac{\rd}{\rd{t}}\Big|_{t=0}\cI[\chi_{[-r,r]},\delta_{x-t}] \le -2\lambda(r+x)r.
\end{equation}
Otherwise
\begin{equation}\begin{split}
\frac{\rd}{\rd{t}}\Big|_{t=0}\cI[\chi_{[-r,r]},\delta_{x-t}] = & -\int_{-r}^{2x-r} W'(x-y)\rd{y}\le -2\lambda(r+x)x,\\
\end{split}\end{equation}
by using $W'(-x)=-W'(x)$.
\end{proof}

\begin{figure}
\begin{center}
  \includegraphics[width=.45\linewidth]{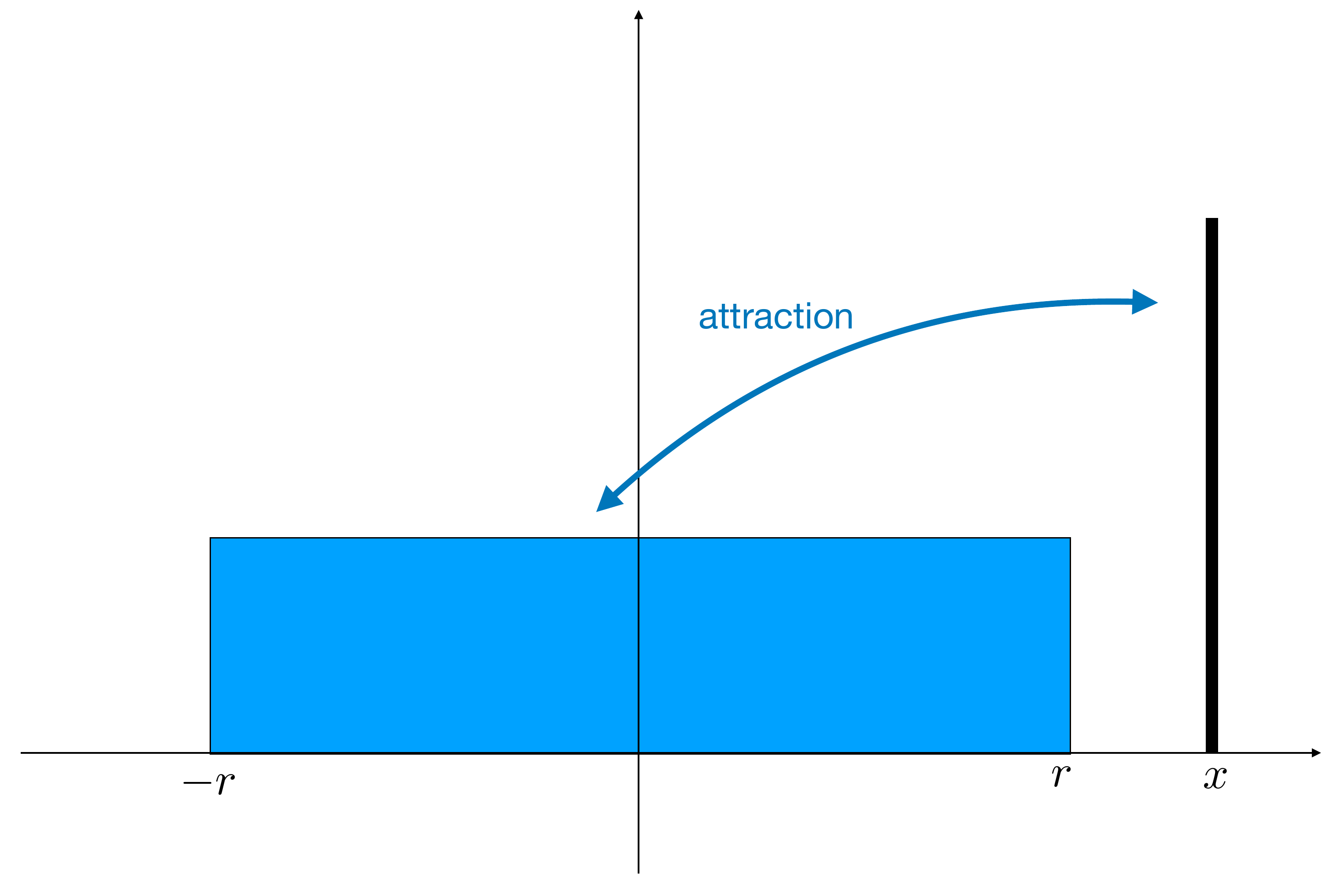}  
  \includegraphics[width=.45\linewidth]{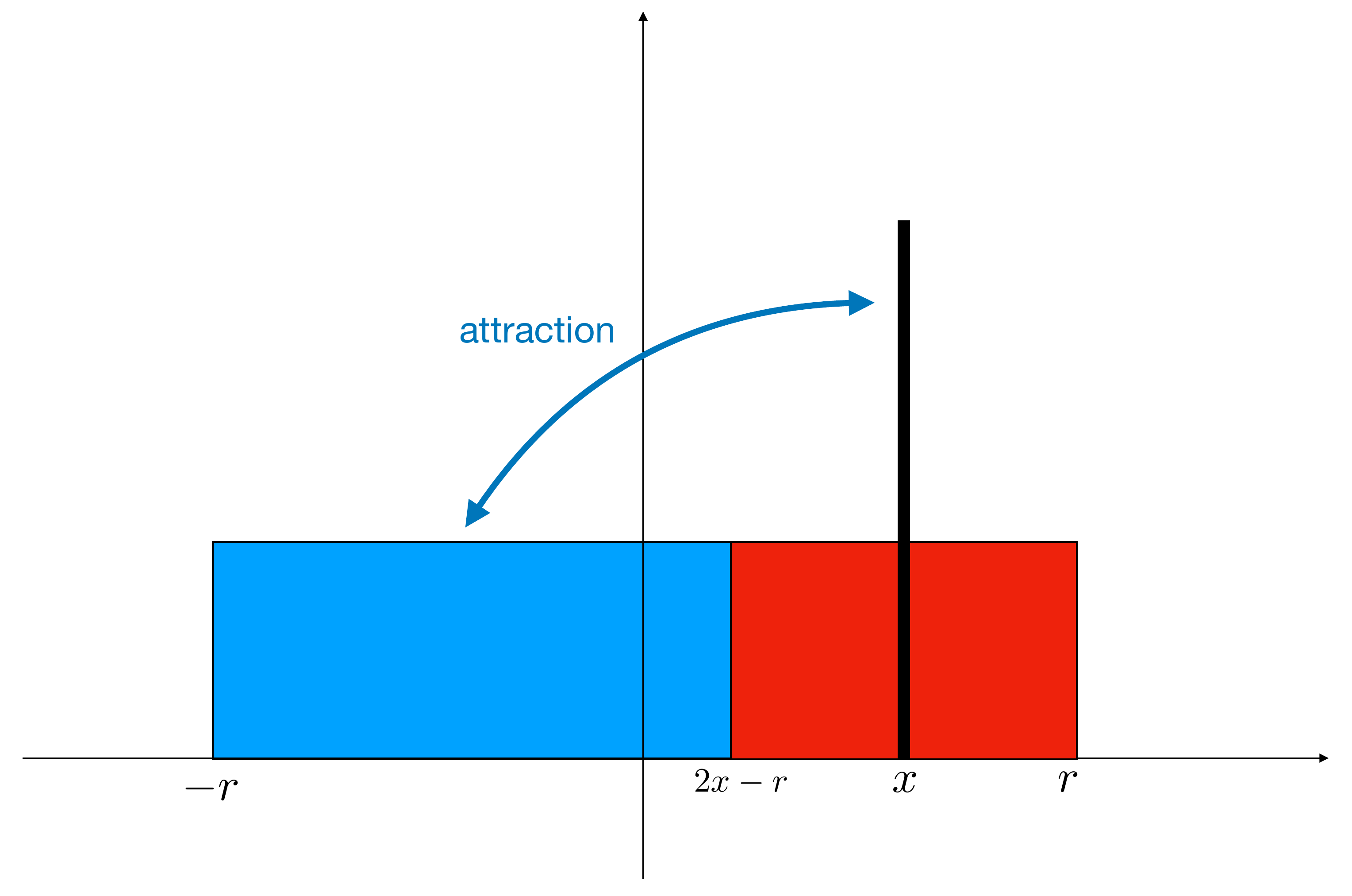}  
  \caption{Proof of Lemma \ref{lem_CSSb}: case $x>r$ (left) and case $x\le r$ (right).}
\label{fig2}
\end{center}
\end{figure}

The following lemma is essentially Lemma 2.16 of \cite{CHVY}, which claims that the interaction energy is decaying when characteristic functions of two intervals have centers getting closer. For the sake of completeness, we provide its proof here.
\begin{lemma}\label{lem_CSSb2}
If $c_1(t),c_2(t)$ satisfies $\frac{\rd}{\rd{t}}\Big|_{t=0}|c_1(t)-c_2(t)| \le 0$, then
\begin{equation}
\frac{\rd}{\rd{t}}\Big|_{t=0}\cI[\chi_{[c_1(t)-r_1,c_1(t)+r_1]},\chi_{[c_2(t)-r_2,c_2(t)+r_2]}] \le 0.
\end{equation}
\end{lemma}
\begin{proof}
By translation and exchanging $c_1,c_2$, we may assume $c_1(t)=0$ and $c_2(0)>0,\,c_2'(0)\le 0$. Then
\begin{equation}\label{Ic1c2}\begin{split}
& \frac{\rd}{\rd{t}}\Big|_{t=0}\cI[\chi_{[c_1(t)-r_1,c_1(t)+r_1]},\chi_{[c_2(t)-r_2,c_2(t)+r_2]}] \\
= & \frac{\rd}{\rd{t}}\Big|_{t=0}\int_{-r_1}^{r_1}\int_{-r_2}^{r_2} W(x-(c_2(t)+y))\rd{y}\rd{x} \\
= & -c_2'(0)\int_{-r_1}^{r_1}\int_{-r_2}^{r_2} W'(x-(c_2(0)+y))\rd{y}\rd{x} \\
= & -c_2'(0)\int_{-r_1}^{r_1}(W(x-(c_2(0)-r_2))-W(x-(c_2(0)+r_2)))\rd{x}. \\
\end{split}\end{equation}
Notice that for $\phi(z):=\int_{-r_1}^{r_1} W(x-z) \rd{x}$ which sastisfies $\phi(z)=\phi(-z)$,
\begin{equation}
\frac{\rd}{\rd{z}}\phi(z) = -\int_{-r_1}^{r_1} W'(x-z) \rd{x} = W(-r_1-z)-W(r_1-z) \ge 0,
\end{equation}
for $z>0$, since $|-r_1-z| \ge |r_1-z|$ in this case. Therefore the conclusion follows from \eqref{Ic1c2} and the fact that $|c_2(0)-r_2|\le |c_2(0)+r_2|$.
\end{proof}

The following lemma estimates the cost of CSS in the sense of 2-Wasserstein metric (c.f. Lemma \ref{lem_basic}):
\begin{lemma}\label{lem_CSScost}
Let $\rho_t$ be defined by 
\begin{equation}
\cC_t\rev{(h)} = \bigcup_j I_{j,t},\quad I_{j,t} = [c_j(t)-r_j,c_j(t)+r_j],
\end{equation}
for each $h$ and $I_j\subset \cC(h)$. Then
\begin{equation}\label{lem_CSScost_1}
\partial_t \rho_t + \partial_x (\rho_t v_t) = 0,\quad v_t(x) = \frac{1}{\rho_t(x)} \int_0^{\rho_t(x)} c_{j(x,h,t)}'(t) \rd{h},
\end{equation}
where $j(x,h,t)$ in the last expression means the unique interval $I_{j,t}\subset \cC(h)$ containing $x$. There holds the estimate for the cost function
\begin{equation}\label{lem_CSScost_2}
\int v_t^2\rho_t\rd{x}\Big|_{t=0} \le \int \sum_j |I_j|\cdot |c_j'(t)|^2 \rd{h}.
\end{equation}
\end{lemma}

\begin{proof}

It suffices to verify \eqref{lem_CSScost_1} at $t=0$. In fact, 
\begin{equation}
\rho_t(x) = \int_0^\infty \sum_j \chi_{I_{j,t}(h)}(x)\rd{h}.
\end{equation}
Taking the primitive function,
\begin{equation}
\int_{-\infty}^x\rho_t(y)\rd{y} = \int_0^\infty \sum_j \int_{-\infty}^x\chi_{I_{j,t}(h)}(y)\rd{y}\rd{h}= \int_0^\infty \sum_j |(-\infty,x]\cap I_{j,t}(h)|\rd{h}.
\end{equation}
Now fix $x$ and take $\frac{\rd}{\rd{t}}\Big|_{t=0}$:
\begin{equation}
\frac{\rd}{\rd{t}}\Big|_{t=0}\int_{-\infty}^x\rho_t(y)\rd{y} = \int_0^\infty \sum_{j: x\in I_{j,t}(h)} (-c_j'(0)) \rd{h},
\end{equation}
since $|(-\infty,x]\cap I_{j,t}(h)|$ is constant in $t$ unless $x\in I_{j,t}(h)$. Then taking $x$ derivative we obtain \eqref{lem_CSScost_1}.

To see \eqref{lem_CSScost_2}, (omitting the indices $t$)
\begin{equation}\begin{split}
\int v^2\rho\rd{x} = & \int \frac{1}{\rho(x)} \Big(\int_0^{\rho(x)} c_{j(x,h)}' \rd{h}\Big)^2\rd{x}\\
\le & \int \frac{1}{\rho(x)} \Big(\int_0^{\rho(x)} (c_{j(x,h)}')^2 \rd{h}\Big)\cdot \Big(\int_0^{\rho(x)}  \rd{h}\Big)\rd{x}\\
\le & \int \int_0^{\rho(x)} (c_{j(x,h)}')^2 \rd{h}\rd{x}\\
= & \int \sum_j |I_j|\cdot |c_j'|^2 \rd{h}.
\end{split}\end{equation}

\end{proof}

\subsection{CSS1 (the original CSS): moving all particles with unit speed}\label{sec_CSS1}

In this and the following subsections, we will define a curve $\rho_t=\rho[\cC_t]$ in each subsection, for a given density distribution $\rho=\rho_0$ and $\cC = \cC[\rho]$. We will always assume that $\cC_t(h)$ remains symmetric for any $t,h$, and then we only need to specify the movement of $I_j,\,j>0$: there always holds $I_{0,t}=I_0$ and $I_{-j,t} = - I_{j,t},\,j>0$.

The curve we will define in this subsection is the same as the CSS defined in \cite{CHVY}. We define the CSS1 curve $\rho_t$ for small $t>0$ by\footnote{We ignore the issue that for some $\rho(x)$ this may not be well-defined for any $t>0$ small. This issue can be easily handled by approximation arguments, since our main focus is the behavior of $\rho_t$ for arbitrarily small $t>0$. The same applies to all other curves $\rho_t$ defined in this paper.}
\begin{equation}
I_{j,t} = [(c_j-t)-r_j,(c_j-t)+r_j],\,j>0,\quad I_{0,t} = I_0,
\end{equation}
where $\cC(h) = \bigcup_{j}I_j$ and $\cC_t(h) = \bigcup_{j}I_{j,t}$. See Figure \ref{fig1} as illustration.

We have the following energy decay estimate:
\begin{lemma}\label{lem_CSS1}
For any fixed $R_3>0$, we have the estimate
\begin{equation}
\frac{\rd}{\rd{t}}\Big|_{t=0}\cI[\rho_t] \le -2\lambda(2R_3)\Big(\int_{[0,R_3]}\mu(x)\rd{x}\Big)^2.
\end{equation}
\end{lemma}
\begin{proof}
We will cut every interval $I_j,\,j>0$ at $R_3$, and this does not change the CSS1 curve.  By using the bi-linearity of $\cI$,
\begin{equation}\begin{split}
\frac{\rd}{\rd{t}}\Big|_{t=0}\cI[\rho_t] = & \frac{1}{2}\int\int \sum_{I_j\subset \cC(h_1)}\sum_{I_k\subset \cC(h_2)} \frac{\rd}{\rd{t}}\Big|_{t=0}\cI[\chi_{I_{j,t}},\chi_{I_{k,t}}] \rd{h_1}\rd{h_2}. \\
\end{split}\end{equation}
It is  clear that $\frac{\rd}{\rd{t}}\Big|_{t=0}\cI[\chi_{I_{j,t}},\chi_{I_{k,t}}] \le 0$ always holds, by Lemma \ref{lem_CSSb2}. Furthermore, if $j<0,\,k>0$ and $I_j,I_k\subset[-R_3,R_3]$, then
\begin{equation}\begin{split}
& \frac{\rd}{\rd{t}}\Big|_{t=0}\cI[\chi_{I_{j,t}},\chi_{I_{k,t}}] \\
= & \frac{\rd}{\rd{t}}\Big|_{t=0} \int_{-r_j}^{r_j}\int_{-r_k}^{r_k} W((c_j+t+x)-(c_k-t+y))\rd{x}\rd{y} \\
= & 2\int_{-r_j}^{r_j}\int_{-r_k}^{r_k} W'((c_j+x)-(c_k+y))\rd{x}\rd{y} \\
\le & -2\lambda(2R_3) |I_j|\cdot|I_k|,
\end{split}\end{equation}
and similarly for $j>0,\,k<0$. Summing over $j,k$ and integrating in $h_1,h_2$ gives the conclusion (where we use the fact that every $I_j$ is cut at $R_3$).

\end{proof}

\begin{corollary}\label{cor_CSS1}
For the solution to \eqref{eq0} we have the energy decay estimate
\begin{equation}
\frac{\rd}{\rd{t}}E[\rho(t,\cdot)] \le -4\lambda(2R_3)^2\Big(\int_{[0,R_3]}\mu(x)\rd{x}\Big)^4.
\end{equation}
\end{corollary}
\begin{proof}
It suffices to prove for $t=0$. We consider the CSS1 curve with $\rho_0=\rho_{\ini}$. Lemma \ref{lem_CSScost} shows that the cost $\int v^2\rho\rd{x} \le 1$ for CSS1. Therefore the conclusion follows from Lemma \ref{lem_CSS1} and Lemma \ref{lem_basic}, by noticing that the internal energy is decaying along this CSS curve: $\frac{\rd}{\rd{t}}\Big|_{t=0}\cS[\rho_t] \le 0$ by Corollary \ref{cor_S}.
\end{proof}

\subsection{CSS2: moving all $I_j$ with center in $[R_1,R_2]$ with unit speed}

Fix $R_2>2R_1>0$. For a given distribution $\rho(x)$ with $h$-representation $\cC(h)$, assuming every interval $I_j,\,$ has been cut at $R_2$,  we define the CSS2 curve $\rho_t = S_2[\rho]$ for small $t>0$ by
\begin{equation}
I_{j,t} = \left\{\begin{split}& [(c_j-t)-r_j,(c_j-t)+r_j],\quad\text{ if }c_j\in[R_1,R_2] \\
& I_j,\quad \text{ otherwise }
\end{split}\right.,
\end{equation}
for every $j>0$, and $I_{0,t} = I_0$. See Figure \ref{fig3} as illustration.

\begin{figure}
\begin{center}
  \includegraphics[width=.8\linewidth]{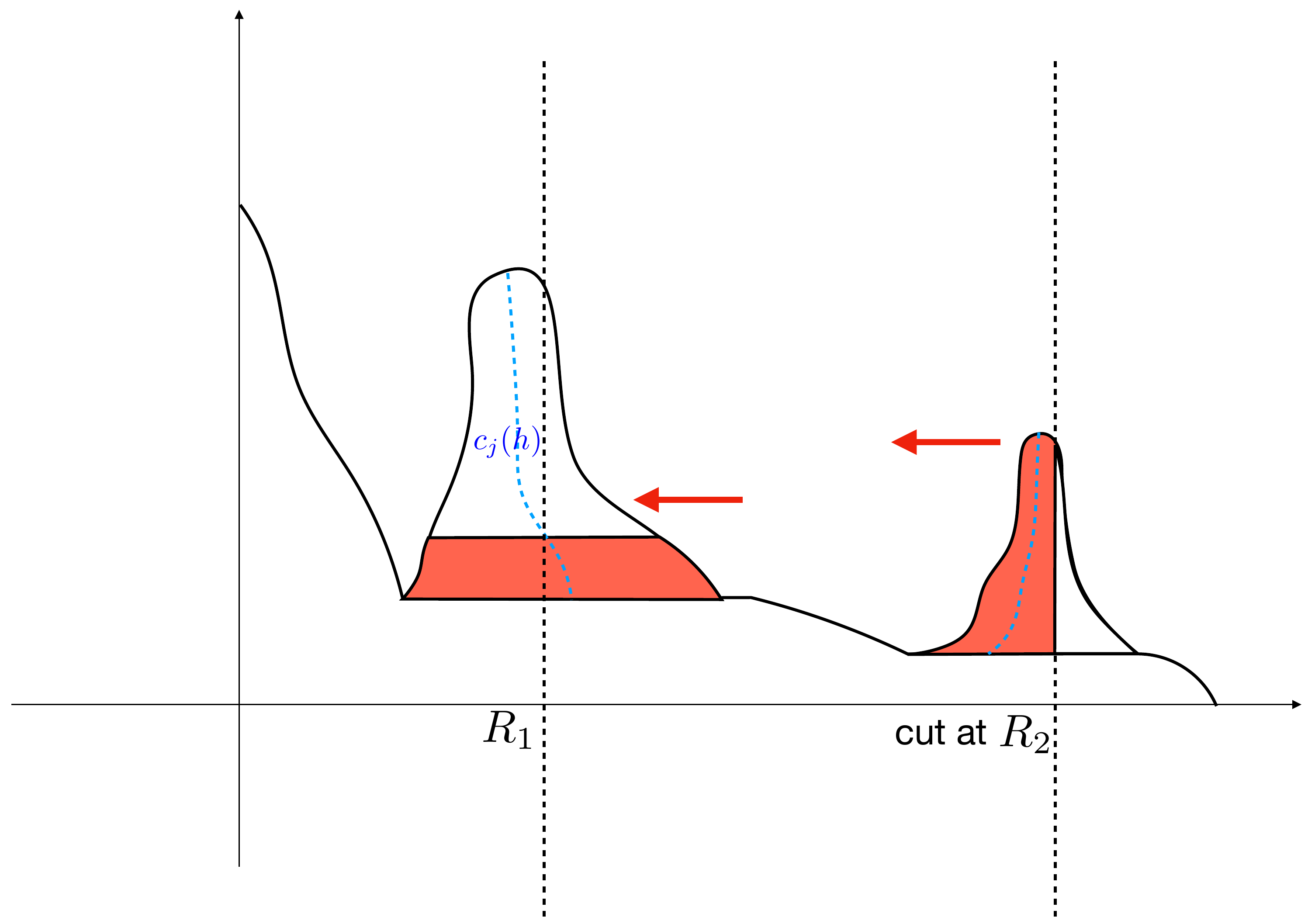}  
  \caption{Illustration of the CSS2 curve: only the red regions are moving left.}
\label{fig3}
\end{center}
\end{figure}

We have the following energy decay estimate:
\begin{lemma}\label{lem_CSS2}
Fix $0<2R_1<R_2<R_3$, with $R_1$ large enough such that \eqref{R1} holds. Then we have the estimate
\begin{equation}\begin{split}
\frac{\rd}{\rd{t}}\Big|_{t=0}\cI[\rho_t]
\le & -\int\sum_{j>0,\,c_j(h)\in [R_1,R_2]}|I_j|\rd{h}\cdot \Big[ \frac{c_\rho}{2}\lambda\Big(2R_2+\frac{4R_1}{c_\rho}\Big) \\
& - \Big(\|W'\|_{L^\infty}\int_{R_2}^{R_3}\mu(x)\rd{x} + \Lambda(R_3-R_2)\int_{R_3}^{\infty}\mu(x)\rd{x}\Big) \Big].
\end{split}\end{equation}
\end{lemma}
\begin{proof}
We will further cut every interval $I_j$ at $R_3$, and this does not affect the definition of CSS2 since all $I_j,\,j>0$ which are moving are inside $[0,R_2]$ (because every $I_j$ has been cut at $R_2$).  By using the bi-linearity of $\cI$,
\begin{equation}\label{dtCSS2}\begin{split}
\frac{\rd}{\rd{t}}\Big|_{t=0}\cI[\rho_t] = & \frac{1}{2}\int\int \sum_{I_j\subset \cC(h_1)}\sum_{I_k\subset \cC(h_2)} \frac{\rd}{\rd{t}}\Big|_{t=0}\cI[\chi_{I_{j,t}},\chi_{I_{k,t}}] \rd{h_1}\rd{h_2} \\.
\end{split}\end{equation}
For any $\cC(h)$, we write
\begin{equation}
\cC(h) = I_0 \cup \cC_1 \cup \cC_2 \cup \cC_3  \cup \cC_{-1}  \cup \cC_{-2}  \cup \cC_{-3}  ,
\end{equation}
where $\cC_{1,2,3}$ contains those $I_j,\,j>0$ whose $c_j$ are inside $[0,R_1)$, $[R_1,R_2]$,  $(R_2,\infty)$ respectively (similarly $\cC_{-1,-2,-3}$ are for $j<0$). By definition of CSS2, only those $I_j\subset \cC_{\pm2}$ are moving, and thus the terms on RHS of \eqref{dtCSS2} are nonzero only when $I_j\subset \cC_{\pm2}$ or $I_k\subset \cC_{\pm2}$, and we only need to analyze the case $I_j\subset \cC_{2}$ by symmetry. 

Take $I_j\subset \cC_{2}(h_1)$ (whose $c_j$ is in $[R_1,R_2]$ and moving to the left). Then\footnote{Below, items 3,4 are the quantitative good contribution, and item 5 is the bad contribution.}
\begin{itemize}
\item If $I_k\subset \cC_{1,0,-1,-2,-3,}(h_2)$ then clearly $\frac{\rd}{\rd{t}}\Big|_{t=0}\cI[\chi_{I_{j,t}},\chi_{I_{k,t}}]\le 0$ by Lemma \ref{lem_CSSb2} since $I_k$ is either staying or moving to the right (i.e., $\frac{\rd}{\rd{t}}\Big|_{t=0}c_{k,t}\ge 0$) and $c_k<c_j$.
\item If $I_k\subset \cC_{2}(h_2)$ then clearly $\frac{\rd}{\rd{t}}\Big|_{t=0}\cI[\chi_{I_{j,t}},\chi_{I_{k,t}}] = 0$ since $I_j$ and $I_k$ are both moving to the left with speed 1 (i.e., $\frac{\rd}{\rd{t}}\Big|_{t=0}(c_{k,t}-c_{j,t}) = 0)$.
\item If $I_k\subset \cC_{-1,-2}(h_2)$ then by Lemma \ref{lem_CSSb},
\begin{equation}
\frac{\rd}{\rd{t}}\Big|_{t=0}\cI[\chi_{I_{j,t}},\chi_{I_{k,t}}] \le -\lambda(2R_2)|I_j|\cdot |I_k|,
\end{equation}
since $I_k\subset(-\infty,0]$ is either staying or moving to the right, $I_j\cap I_k=\emptyset$, and $I_j,I_k\subset [-R_2,R_2]$.
\item If $k=0$ then $I_k=I_0(h_2)$ is staying, and by Lemma \ref{lem_CSSb},
\begin{equation}\begin{split}
& \frac{\rd}{\rd{t}}\Big|_{t=0}\cI[\chi_{I_{j,t}},\chi_{I_0(h_2)}]\\
 \le & -2\int_{c_j-r_j}^{c_j+r_j} \lambda(x+r_0(h_2))\min\{x,r_0(h_2)\}\rd{x} \\
\le & -2\lambda(R_2+r_0(h_2))\int_{c_j-r_j}^{c_j+r_j} \min\{x,r_0(h_2)\}\rd{x} \\
\le & -2\lambda(R_2+r_0(h_2))\int_{c_j-r_j}^{c_j+r_j} \Big(\min\{c_j-r_j,r_0(h_2)\} \\
& + \frac{\min\{c_j+r_j,r_0(h_2)\}-\min\{c_j-r_j,r_0(h_2)\}}{2r_j}x\Big)\rd{x} \\
= & -2\lambda(R_2+r_0(h_2))r_j(\min\{c_j+r_j,r_0(h_2)\}+\min\{c_j-r_j,r_0(h_2)\}) \\
\le & -\lambda(R_2+r_0(h_2))\min\{R_1,r_0(h_2)\}\cdot |I_j|, \\
\end{split}\end{equation}
where the second inequality uses the concavity of $x\mapsto \min\{x,r_0(h_2)\}$, and the last inequality uses the fact that $c_j \ge R_1$.
\item If $I_k\subset \cC_{3}(h_2)$ then $I_k\subset[R_2,\infty)$ is staying, and
\begin{equation}\begin{split}
& \frac{\rd}{\rd{t}}\Big|_{t=0}\cI[\chi_{I_{j,t}},\chi_{I_{k,t}}] \\
= & \int_{-r_j}^{r_j}\int_{I_k}\frac{\rd}{\rd{t}}\Big|_{t=0}W(c_j-t+x-y)\rd{y}\rd{x} \\
= & -\int_{-r_j}^{r_j}\int_{I_k}W'(c_j+x-y)\rd{y}\rd{x} \\
= & -\int_{c_j-r_j}^{c_j+r_j}\int_{I_k\cap [R_2,R_3]}W'(x-y)\rd{y}\rd{x} -\int_{c_j-r_j}^{c_j+r_j}\int_{I_k\cap[R_3,\infty)}W'(x-y)\rd{y}\rd{x} \\
\le & |I_j|\cdot (\|W'\|_{L^\infty}|I_k\cap [R_2,R_3]| + \Lambda(R_3-R_2)|I_k\cap [R_3,\infty)|), \\
\end{split}\end{equation}
where the last inequality uses the uniform upper bound for $W'$ for the first integral, and the decaying upper bound $\Lambda$ for $W'$ for the second integral, due to $I_j\in [0,R_2]$.

\end{itemize}

Therefore, by summing over $k$ and integrating in $h_2$, we get
\begin{equation}\label{dtCSS22}\begin{split}
& \int \sum_k \frac{\rd}{\rd{t}}\Big|_{t=0}\cI[\chi_{I_{j,t}},\chi_{I_{k,t}}]\rd{h_2} \\
\le & -|I_j|\cdot \Big[ \lambda(2R_2)\int_0^{R_2}\mu(x)\rd{x} + \int \lambda(R_2+r_0(h_2))\min\{R_1,r_0(h_2)\}\rd{h_2} \\
& - \Big(\|W'\|_{L^\infty}\int_{R_2}^{R_3}\mu(x)\rd{x} + \Lambda(R_3-R_2)\int_{R_3}^{\infty}\mu(x)\rd{x}\Big) \Big].
\end{split}\end{equation}

Then we estimate (where $\epsilon>0$ to be determined)
\begin{equation}\label{r0eps}\begin{split}
& \int \lambda(R_2+r_0(h_2))\min\{R_1,r_0(h_2)\}\rd{h_2} \\
\ge & \lambda(R_2+r_0(\epsilon))\int_{h_2>\epsilon} \min\{R_1,r_0(h_2)\}\rd{h_2} \\
= & 2\lambda(R_2+r_0(\epsilon))\int_0^{R_1} \max\{\rho^*(x)-\epsilon,0\}\rd{x} \\
\ge & 2\lambda(R_2+r_0(\epsilon))\Big(\int_0^{R_1} \rho^*(x)\rd{x}-\epsilon R_1\Big). \\
\end{split}\end{equation}
Notice that $r_0(\epsilon) \le 1/\epsilon$ since the total mass is 1. Then, by choosing 
\begin{equation}
\epsilon = \frac{c_\rho}{4R_1},
\end{equation}
we get
\begin{equation}\begin{split}
& \int \lambda(R_2+r_0(h_2))\min\{R_1,r_0(h_2)\}\rd{h_2} \ge \frac{c_\rho}{2}\lambda(R_2+\frac{4R_1}{c_\rho}),\quad \text{ if }\int_0^{R_1} \rho^*(x)\rd{x} \ge \frac{c_\rho}{2}. \\
\end{split}\end{equation}
On the other hand, if $\int_0^{R_1} \rho^*(x)\rd{x} < \frac{c_\rho}{2}$ then $\int_0^{R_1}\mu(x)\rd{x}\ge \frac{c_\rho}{2}$. Using these in \eqref{dtCSS22} gives
\begin{equation}\begin{split}
& \int \sum_k \frac{\rd}{\rd{t}}\Big|_{t=0}\cI[\chi_{I_{j,t}},\chi_{I_{k,t}}]\rd{h_2} \\
\le & -|I_j|\cdot \Big[ \frac{c_\rho}{2}\lambda\Big(\max\{2R_2,R_2+\frac{4R_1}{c_\rho}\}\Big) \\
& - \Big(\|W'\|_{L^\infty}\int_{R_2}^{R_3}\mu(x)\rd{x} + \Lambda(R_3-R_2)\int_{R_3}^{\infty}\mu(x)\rd{x}\Big) \Big].
\end{split}\end{equation}
Summing over $I_j\subset \cC_2(h_1)$ and integrating in $h_1$, we get the conclusion.

\end{proof}

\begin{corollary}\label{cor_CSS2}
Under the same assumption as Lemma \ref{lem_CSS2}, for the solution to \eqref{eq0} we have the energy decay estimate
\begin{equation}\label{cor_CSS2_1}\begin{split}
\frac{\rd}{\rd{t}}E[\rho(t,\cdot)]
\le & -\frac{1}{2}\int_{2R_1}^{R_2} \mu(x)\rd{x}\cdot \Big[ \frac{c_\rho}{2}\lambda\Big(2R_2+\frac{4R_1}{c_\rho}\Big) \\
& - \Big(\|W'\|_{L^\infty}\int_{R_2}^{R_3}\mu(x)\rd{x} + \Lambda(R_3-R_2)\int_{R_3}^{\infty}\mu(x)\rd{x}\Big) \Big]^2,
\end{split}\end{equation}
provided the quantity in the last bracket is nonnegative.
\end{corollary}
\begin{proof}
It suffices to prove for $t=0$. We consider the CSS2 curve with $\rho_0=\rho_{\ini}$. Lemma \ref{lem_CSScost} shows that the cost 
\begin{equation}
\int v^2\rho\rd{x} \le 2\int\sum_{j>0,\,c_j(h)\in [R_1,R_2]}|I_j|\rd{h},
\end{equation}
 for CSS2. Also notice that
\begin{equation}
\int\sum_{j>0,\,c_j(h)\in [R_1,R_2]}|I_j|\rd{h} \ge \int_{2R_1}^{R_2} \mu(x)\rd{x},
\end{equation}
since every interval $I_j(h),\,j>0$ with $I_j\cap [2R_1,R_2) \ne \emptyset$ necessarily have $I_j\subset [0,R_2]$ (since we cut at $R_2$ for CSS2) and thus having $c_j\in [R_1,R_2]$.  Therefore the conclusion follows from Lemma \ref{lem_CSS2} and Lemma \ref{lem_basic}, by noticing that the internal energy is decaying along this CSS curve: $\frac{\rd}{\rd{t}}\Big|_{t=0}\cS[\rho_t] \le 0$ by Corollary \ref{cor_S}.
\end{proof}

\subsection{Rescaled continuous Steiner symmetrization (RCSS)}

In this subsection we introduce the \emph{rescaled continuous Steiner symmetrization} (RCSS), which will be used for the quantitative energy decay rate estimate. 

For a given distribution $\rho(x)$ with $h$-representation $\cC(h)$,  we define the RCSS curve $\rho_t$ for small $t>0$ by
\begin{equation}
I_{j,t} =  [c_je^{-t}-r_j,c_je^{-t}+r_j],
\end{equation}
for every $j>0$, and $I_{0,t} = I_0$. This means that those $I_j,\,j>0$ is moving towards the center at a faster speed if $c_j$ is further away from the center.

Now we estimate the energy decay from the RCSS curve. 

\begin{lemma}\label{lem_RCSS}
Assume that $\rho$ is supported on $[-R,R]$. Then the RCSS curve $\rho_t$ satisfies
\begin{equation}
\frac{\rd}{\rd{t}}\Big|_{t=0}\cI[\rho_t] \le -\frac{\lambda(2R)}{R}\int_0^\infty x^2\mu(x)\rd{x}.
\end{equation}
\end{lemma}

\begin{proof}

By the bi-linearity of the interaction energy, we first write
\begin{equation}\begin{split}
\frac{\rd}{\rd{t}}\Big|_{t=0}\cI[\rho_t] = & \frac{1}{2}\int\int \sum_{I_j\subset \cC(h_1)}\sum_{I_k\subset \cC(h_2)} \frac{\rd}{\rd{t}}\Big|_{t=0}\cI[\chi_{I_{j,t}},\chi_{I_{k,t}}] \rd{h_1}\rd{h_2} \\
= & 2\int\int \sum_{I_k\subset \cC(h_2),\,k>0} \frac{\rd}{\rd{t}}\Big|_{t=0}\cI[\chi_{I_0(h_1)},\chi_{I_{k,t}}] \rd{h_1}\rd{h_2} \\ 
& + \int\int \sum_{I_j\subset \cC(h_1),\,j<0}\sum_{I_k\subset \cC(h_2),\,k>0} \frac{\rd}{\rd{t}}\Big|_{t=0}\cI[\chi_{I_{j,t}},\chi_{I_{k,t}}] \rd{h_1}\rd{h_2} \\
& + \int\int \sum_{I_j\subset \cC(h_1),\,j>0}\sum_{I_k\subset \cC(h_2),\,k>0} \frac{\rd}{\rd{t}}\Big|_{t=0}\cI[\chi_{I_{j,t}},\chi_{I_{k,t}}] \rd{h_1}\rd{h_2}, \\
\end{split}\end{equation}
where we used the fact that $I_0(h)$ is not moving and the symmetry of $\rho_t$. 

By Lemma \ref{lem_CSSb}, 
\begin{equation}\begin{split}
\frac{\rd}{\rd{t}}\Big|_{t=0}\cI[\chi_{I_0(h_1)},\chi_{I_{k,t}(h_2)}] = & \int_{-r_k}^{r_k}\frac{\rd}{\rd{t}}\Big|_{t=0}\cI[\chi_{[-r_0(h_1),r_0(h_1)]},\delta_{c_ke^{-t}+x}] \rd{x} \\
\le & -2\lambda(2R)c_k \int_{-r_k}^{r_k} \min\{r_0(h_1),c_k+x\}\rd{x} \\
\le & -\lambda(2R)(c_k+r_k) \int_{-r_k}^{r_k} \min\{r_0(h_1),c_k+x\}\rd{x} \\
\le & -\lambda(2R) \int_{c_k-r_k}^{c_k+r_k} \min\{r_0(h_1),x\}x\rd{x}. \\
\end{split}\end{equation}
Integrating in $h_1$ gives
\begin{equation}\begin{split}
\int\frac{\rd}{\rd{t}}\Big|_{t=0}\cI[\chi_{I_0(h_1)},\chi_{I_{k,t}(h_2)}]\rd{h_1} \le & -\lambda(2R) \int_{c_k-r_k}^{c_k+r_k} \int\min\{r_0(h_1),x\} \rd{h_1}x\rd{x} \\
= & -\lambda(2R) \int_{c_k-r_k}^{c_k+r_k} \int_0^x \rho^*(y)\rd{y} x\rd{x} \\
\le & -\lambda(2R) \int_{c_k-r_k}^{c_k+r_k} \frac{m_*}{R}x\cdot x\rd{x} \\
= & -\lambda(2R)\frac{m_*}{R} \int_{c_k-r_k}^{c_k+r_k} x^2\rd{x}, \\
\end{split}\end{equation}
where $m_*:=\int_0^R\rho^*(y)\rd{y}$, and the second inequality uses the fact that $x\mapsto  \int_0^x \rho^*(y)\rd{y}$ is a concave function. Then summing in $k$ and integrating in $h_2$ gives
\begin{equation}\begin{split}
\int\int \sum_{I_k\subset \cC(h_2),\,k>0}\frac{\rd}{\rd{t}}\Big|_{t=0}\cI[\chi_{I_0(h_1)},\chi_{I_{k,t}(h_2)}]\rd{h_1}\rd{h_2} \le  -\lambda(2R)\frac{m_*}{R} \int_0^\infty x^2\mu(x)\rd{x}. \\
\end{split}\end{equation}

Next, since $I_j$ and $I_k$ are disjoint in the terms below,
\begin{equation}\begin{split}
& \int\int \sum_{I_j\subset \cC(h_1),\,j<0}\sum_{I_k\subset \cC(h_2),\,k>0} \frac{\rd}{\rd{t}}\Big|_{t=0}\cI[\chi_{I_{j,t}},\chi_{I_{k,t}}] \rd{h_1}\rd{h_2} \\
\le &-\lambda(2R) \int\int \sum_{I_j\subset \cC(h_1),\,j<0}\sum_{I_k\subset \cC(h_2),\,k>0} (c_j+c_k)|I_j|\cdot|I_k| \rd{h_1}\rd{h_2} \\
= &-2\lambda(2R) \int\int \sum_{I_j\subset \cC(h_1),\,j>0}\sum_{I_k\subset \cC(h_2),\,k>0} c_k|I_j|\cdot|I_k| \rd{h_1}\rd{h_2} \\
= &-2\lambda(2R) \int_0^\infty \mu(x)\rd{x} \cdot \int_0^\infty x\mu(x)\rd{x} \\
\le &-\frac{2\lambda(2R)}{R} \cdot\frac{1-m_*}{2} \cdot \int_0^\infty x^2\mu(x)\rd{x} .\\
\end{split}\end{equation}

Finally, we have
\begin{equation}
\int\int \sum_{I_j\subset \cC(h_1),\,j>0}\sum_{I_k\subset \cC(h_2),\,k>0} \frac{\rd}{\rd{t}}\Big|_{t=0}\cI[\chi_{I_{j,t}},\chi_{I_{k,t}}] \rd{h_1}\rd{h_2} \le 0 ,
\end{equation}
by Lemma \ref{lem_CSSb2}, since $c_j(t)-c_k(t) = (c_j-c_k)e^{-t}$ is contracting. Therefore we obtain the conclusion. 
\end{proof}

\begin{corollary}\label{cor_RCSS}
For the solution to \eqref{eq0} we have the energy decay estimate
\begin{equation}\label{cor_RCSS_1}
\frac{\rd}{\rd{t}}E(t) \le -\frac{\lambda(2R)^2}{R^2}\int x^2\mu(x)\rd{x}.
\end{equation}
\end{corollary}
\begin{proof}
It suffices to prove for $t=0$. We consider the RCSS curve with $\rho_0=\rho_{\ini}$. Lemma \ref{lem_CSScost} shows that the cost 
\begin{equation}\label{RCSScost}
\int v^2\rho\rd{x} \le \int \sum_j |I_j|c_j^2 \rd{h} \le  \int \sum_j \int_{c_j-r_j}^{c_j+r_j}x^2\rd{x} \rd{h} = \int x^2\mu(x)\rd{x}.
\end{equation}
 Therefore the conclusion follows from Lemma \ref{lem_RCSS} and Lemma \ref{lem_basic}, by noticing that the internal energy is decaying along this CSS curve: $\frac{\rd}{\rd{t}}\Big|_{t=0}\cS[\rho_t] \le 0$ by Corollary \ref{cor_S}.
\end{proof}

\section{Uniform estimate of the first moment}\label{sec_tight}

In this section we prove Proposition \ref{prop_nrd} and Proposition \ref{prop_rd} which lead to Theorem \ref{thm_main1}.

\subsection{Controlling the non-radially-decreasing part}

\begin{proof}[Proof of Proposition \ref{prop_nrd}]
\rev{We start by noticing that $E(0)-E(t) \le E(0)-E_\infty$ for any $t\ge 0$, since $E_\infty=E[\rho_\infty]$ is the global minimum of the energy.} For any $c_1>0$ and $R_3>R_2$, Corollary \ref{cor_CSS1} gives
\begin{equation}
\Big|\{t:\int_{[R_2,R_3]}\mu(t,x)\rd{x}> c_1\}\Big| \le \frac{E(0)-E_\infty}{4\lambda(2R_3)^2c_1^4}.
\end{equation}
Since $R_1$ is assumed to be large enough, we are able to apply Corollary \ref{cor_CSS2}. By the property $\lim_{x\rightarrow\infty}\Lambda(x)=0$ in {\bf (A2)}, we can choose $R_3$ large enough (depending on $R_1,R_2$) such that
\begin{equation}
\Lambda(R_3-R_2) \le \frac{c_\rho}{8}\lambda\Big(2R_2+\frac{4R_1}{c_\rho}\Big),
\end{equation}
and then choose 
\begin{equation}
c_1= \frac{1}{\|W'\|_{L^\infty}}\cdot \frac{c_\rho}{8}\lambda\Big(2R_2+\frac{4R_1}{c_\rho}\Big).
\end{equation}
\rev{Then, for every $t$ with $\int_{[R_2,R_3]}\mu(t,x)\rd{x}\le c_1$, the bracket on the RHS of \eqref{cor_CSS2_1} is bounded below by
\begin{equation}
\frac{c_\rho}{2}\lambda\Big(2R_2+\frac{4R_1}{c_\rho}\Big) - \Big(\|W'\|_{L^\infty}\cdot c_1 + \frac{c_\rho}{8}\lambda\Big(2R_2+\frac{4R_1}{c_\rho}\Big)\cdot 1\Big) = \frac{c_\rho}{4}\lambda\Big(2R_2+\frac{4R_1}{c_\rho}\Big).
\end{equation}
}By Corollary \ref{cor_CSS2}, these guarantee that, for all time except a time length of $\frac{E(0)-E_\infty}{4\lambda(2R_3)^2c_1^4}$, we have
\rev{\begin{equation}\begin{split}
\frac{\rd}{\rd{t}}E[\rho(t,\cdot)]
\le & -\int_{2R_1}^{R_2} \mu(t,x)\rd{x}\cdot  \Big[\frac{c_\rho}{4}\lambda\Big(2R_2+\frac{4R_1}{c_\rho}\Big)\Big]^2. \\
\end{split}\end{equation}}
Therefore
\rev{\begin{equation}
\int_0^\infty \int_{[2R_1,R_2]} \mu(t,x)\rd{x} \rd{t} \le \frac{E(0)-E_\infty}{\Big[\frac{c_\rho}{4}\lambda\Big(2R_2+\frac{4R_1}{c_\rho}\Big)\Big]^2} + \frac{E(0)-E_\infty}{4\lambda(2R_3)^2c_1^4} \le CR_2^{4\alpha+\frac{\alpha^2}{\beta}},
\end{equation}}
where we used $R_1\le R_2$ and $R_3 \sim R_2^{\alpha/\beta}$.

\end{proof}

\subsection{Controlling the radially-decreasing part}

To prove Proposition \ref{prop_rd} we first state the following lemma, which is the key to enforce the formation of local clusters for radially-decreasing distributions:
\begin{lemma}[Lemma of local clustering]\label{lem_r}
Let $R>0$, $\rho(x)$ be a decreasing non-negative function defined on $[R,3R]$, and $m>1$. Then there exists $r\in [R,2R]$ such that
\begin{equation}\label{lem_r_1}
\int_r^{3R} \rho(x)^m\rd{x} \le a\int_r^{3R}(x-r)\rho(x)\rd{x},\quad a = \frac{C_m}{R}\rho(R)^{m-1}.
\end{equation}
\end{lemma}

\begin{proof}
We choose $C_m$ by
\begin{equation}\label{Cm1}
C_m = \max\{10000,200\cdot 2^{2/(m-1)})\}.
\end{equation}
The key property we will use is that, for all $0\le q \le (\frac{200}{C_m})^{1/2} \le \frac{1}{2}$,
\begin{equation}\label{Cm2}
\sum_{k=0}^\infty q^{m^k} \le \sum_{k=0}^\infty \Big((\frac{200}{C_m})^{(m-1)/2}\Big)^k = \frac{1}{1-(\frac{200}{C_m})^{m-1}} \le 2.
\end{equation}

We may assume $\rho(x)>0$ for $x\in [R,1.5R]$ since otherwise $\rho(x)|_{[1.5R,3R]}=0$ by its decreasing property, and the conclusion follows trivially by setting $r=1.5R$. Then we assume the contrary of the conclusion:
\begin{equation}\label{rcont}
\int_r^{3R} \rho(x)^m\rd{x} > a\int_r^{3R}(x-r)\rho(x)\rd{x},\quad \forall r\in [R,2R].
\end{equation}
For a fixed $r\in [R,1.5R]$, let $s=s(r)> r\in [R,1.6R]$ be the unique number such that
\begin{equation}\label{phirs}
\phi_r(s) := 2\rho(s)^{m-1} - a(s-r) =0,
\end{equation}
whose existence and uniqueness follow from the decreasing property of $\phi_r$, and $\phi_r(r)> 0$, \\ $\phi_r(1.6R)\le 2\rho(R)^{m-1} - \frac{C_m}{R}\rho(R)^{m-1}\cdot 0.1R \le 0$ (by \eqref{Cm1}). Furthermore,  
\begin{equation}\label{phirs1}
s+0.1\sqrt{R(s-r)}\le 1.6R + 0.1\sqrt{0.6R^2} \le 2R.
\end{equation}
Then
\begin{equation}\begin{split}
0& < \int_r^{3R} \rho(x)^m\rd{x} - a\int_r^{3R}(x-r)\rho(x)\rd{x} \\
= & \int_r^{s} (\rho(x)^{m-1}-a(x-r))\rho(x)\rd{x} + \int_s^{3R} (\rho(x)^{m-1}-a(x-r))\rho(x)\rd{x} \\
\le & \int_r^{s} \rho(x)^m\rd{x} - \frac{a}{2}\int_s^{3R} (x-r)\rho(x)\rd{x} \\
\le & (s-r)\rho(r)^m - \frac{a}{2}\int_{s+0.1\sqrt{R(s-r)}}^{s+0.2\sqrt{R(s-r)}} (x-r)\rho(x)\rd{x} \\
\le & (s-r)\rho(r)^m - \frac{a}{2}0.1\sqrt{R(s-r)}\cdot 0.1\sqrt{R(s-r)}\cdot \rho(s+0.2\sqrt{R(s-r)}) \\
= & (s-r)(\rho(r)^m - 0.005aR\cdot \rho(s+0.2\sqrt{R(s-r)})), \\
\end{split}\end{equation}
where the first inequality uses the definition of $s=s(r)$  to estimate the second integral, the second inequality uses the decreasing property of $\rho$ and the property \eqref{phirs1}, and the third inequality uses the decreasing property of $\rho$.

Since $s>r$, this implies
\begin{equation}\label{rhoite}
\rho(s+0.2\sqrt{R(s-r)}) \le \frac{200}{aR}\rho(r)^m.
\end{equation}

Let $\{r_k\}$ be defined iteratively by
\begin{equation}
r_{k+1} = s(r_k) + 0.2\sqrt{R(s(r_k)-r_k)},\quad r_0 = R.
\end{equation}
We will show that $r_k\in [R,1.5R]$ which implies that every $r_k$ is well-defined. To see this, we first notice that \eqref{rhoite}, applied iteratively, implies
\begin{equation}\label{rhork}
\rho(r_k) \le \rho(R)^{m^k}\Big(\frac{200}{aR}\Big)^{\frac{m^k-1}{m-1}}.
\end{equation}
Therefore 
\begin{equation}
s(r_k)\le r_k + \frac{2}{a}\Big(\rho(R)^{m^k}\Big(\frac{200}{aR}\Big)^{\frac{m^k-1}{m-1}}\Big)^{m-1},
\end{equation}
since $\phi_{r_k}(s)$ (as defined in \eqref{phirs}), evaluated at the above RHS, is negative:
\begin{equation}\begin{split}
& \phi_{r_k}\Big(r_k + \frac{2}{a}\Big(\rho(R)^{m^k}\Big(\frac{200}{aR}\Big)^{\frac{m^k-1}{m-1}}\Big)^{m-1}\Big) \\
= & 2\rho\Big(r_k + \frac{2}{a}(\rho(R)^{m^k}a^{\frac{m^k-1}{m-1}})^{m-1}\Big)^{m-1} - a\cdot\frac{2}{a}\Big(\rho(R)^{m^k}\Big(\frac{200}{aR}\Big)^{\frac{m^k-1}{m-1}}\Big)^{m-1}\\
\le & 2\rho(r_k)^{m-1} - 2\rho(r_k)^{m-1} \le 0.
\end{split}\end{equation}
Therefore
\begin{equation}\begin{split}
r_{k+1}-r_k \le & \frac{2}{a}\rho(R)^{m^k(m-1)}\Big(\frac{200}{aR}\Big)^{m^k-1} + 0.2\sqrt{R\frac{2}{a}\rho(R)^{m^k(m-1)}\Big(\frac{200}{aR}\Big)^{(m^k-1)}} \\
= & \frac{R}{100}\Big(\frac{200\rho(R)^{m-1}}{aR}\Big)^{m^k} + \frac{R}{50}\Big(\frac{200\rho(R)^{m-1}}{aR}\Big)^{m^k/2}.
\end{split}\end{equation}
By noticing that
\begin{equation}
\frac{200\rho(R)^{m-1}}{aR} = \frac{200}{C_m},
\end{equation}
we conclude by \eqref{Cm2} that
\begin{equation}
r_K-R = \sum_{k=0}^{K-1}(r_{k+1}-r_k) \le  \frac{R}{10},\quad \forall K\ge 1,
\end{equation}
which justifies that all $r_k$ are well-defined, and $r_k\le 1.5R$. This implies $\lim_{k\rightarrow\infty}r_k \le 1.5R$. \eqref{rhork} clearly implies $\lim_{k\rightarrow\infty}\rho(r_k)=0$, and thus $\rho|_{[2R,3R]}=0$. This contradicts the assumption \eqref{rcont} with $r=2R$.

\end{proof}

\begin{remark}
Lemma \ref{lem_r} is false for $m=1$. In fact, fix any $R>0$, and then the function $\rho(x) = e^{-Ax}$ gives 
\begin{equation}
\int_r^{3R} \rho(x)\rd{x} = \frac{1}{A}(e^{Ar}-e^{A\cdot 3R}) ,
\end{equation}
and 
\begin{equation}
\int_r^{3R} (x-r)\rho(x)\rd{x} = -\frac{1}{A}(3R-r)e^{-A\cdot 3R} +  \frac{1}{A^2}(e^{Ar}-e^{A\cdot 3R}) \le \frac{1}{A}\cdot \frac{1}{A}(e^{Ar}-e^{A\cdot 3R}),
\end{equation}
which clearly fail \eqref{lem_r_1} with $m=1$ for any $R\le r \le 2R$, if $A>a$. 

The above failure for $m=1$ can be understood as follows: equating the LHS and RHS of \eqref{lem_r_1} and taking second order derivative with respect to $r$ gives the ODE $\rho'(r) = -\frac{a}{m}\rho(r)^{2-m}$. For $m>1$, the solutions to this ODE decay to zero within \emph{finite} distance, while for $m=1$, its solutions are $\rho(x)=Ce^{-ax}$ which may be positive everywhere. The proof of this lemma can be viewed as a discrete analog of this ODE.

\end{remark}

\begin{proof}[Proof of Proposition \ref{prop_rd}]

We take $R_1$ satisfying \eqref{R1} and $R_2\ge 7R_1$, being large and  to be chosen. Given a density distribution $\rho=\rho_0$, we define $\rho_t,\,t>0$ by the velocity field
\begin{equation}
\partial_t\rho_t + \partial_x(\rho_t v) = 0,\quad v(x) = -\chi_{[r,6R_1]}\frac{x-r}{4R_1},\,x>0,\quad v(-x) = -v(x),
\end{equation}
where $r$ is given by\footnote{Notice here that the choice of $r$, and therefore $v(x)$, may depend on $R_1$ and $\rho$.} Lemma \ref{lem_r} (with $\rho^*(x)$ and $R=2R_1$), satisfying $2R_1\le r \le 4R_1$. Clearly $\int v^2\rho\rd{x} \le 1$ since $|v(x)|\le 1,\,\forall x\in\mathbb{R}$. See Figure \ref{fig4} as illustration.

\begin{figure}
\begin{center}
  \includegraphics[width=.8\linewidth]{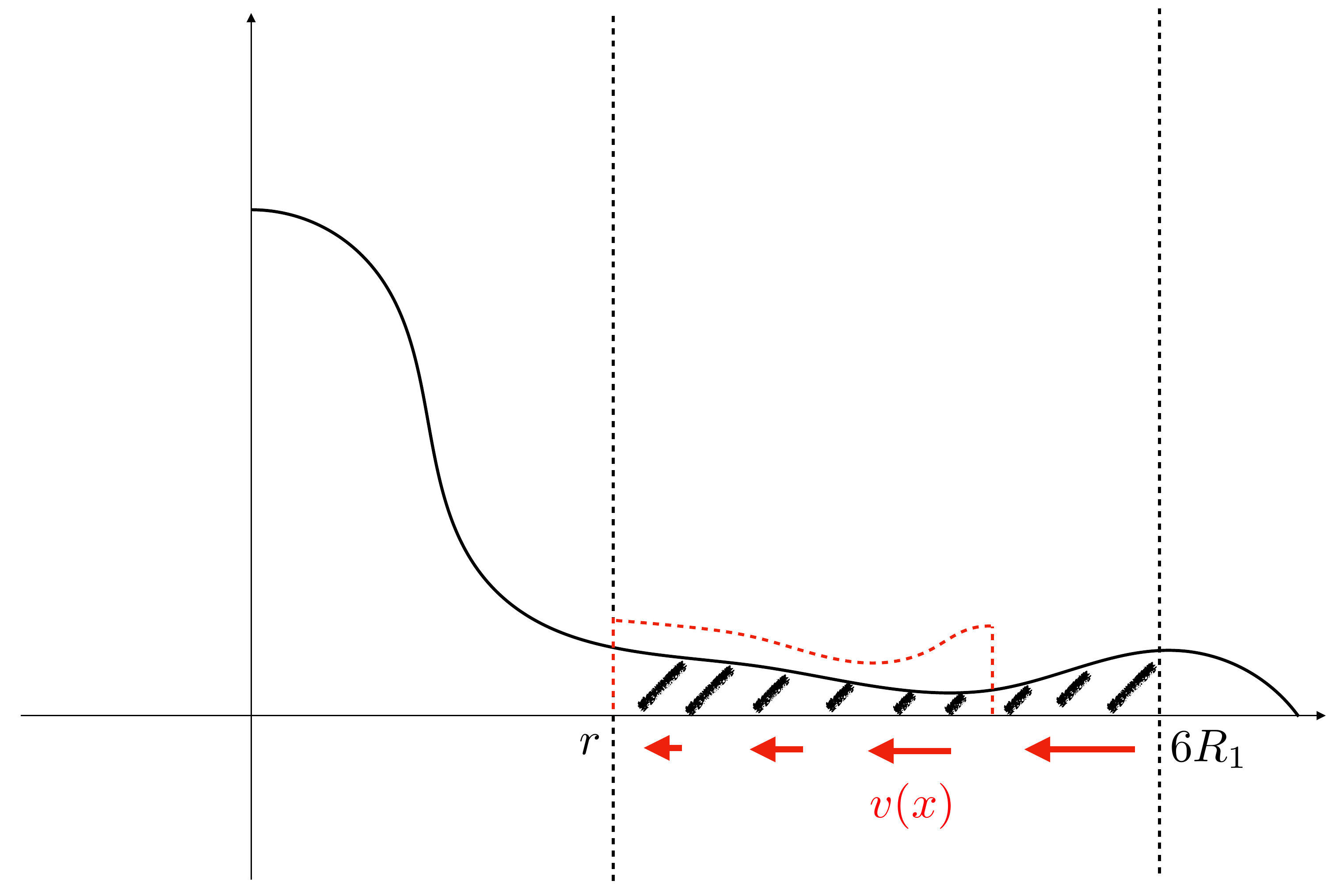}  
  \caption{Illustration of the local clustering curve: the shaded region is transported by $v(x)$ and becomes the region under the red dashed curve.}
\label{fig4}
\end{center}
\end{figure}

The interaction energy decay rate of this curve $\rho_t$ is given by
\begin{equation}\label{3R10}\begin{split}
\frac{\rd}{\rd{t}}\Big|_{t=0}\cI[\rho_t] = & -\int\int W(x-y)\rho(y)\rd{y}\partial_x(\rho(x)v(x))\rd{x} \\
= & \int_{r\le |x| \le 6R_1}\int W'(x-y)\rho(y)\rd{y}\rho(x)v(x)\rd{x}, \\
\end{split}\end{equation}
where the last equality uses the fact that $v(x)$ is supported on $[-6R_1,r]\cup [r,6R_1]$.

For fixed $x$ with $r\le x \le 6R_1$,
\begin{equation}\label{3R1}\begin{split}
\int W'(x-y)\rho(y)\rd{y} = & \int W'(x-y)\Big(\rho^*(y)+\mu(y)\chi_{[-r,r]}(y)\Big)\rd{y} \\
& + \int W'(x-y)\mu(y)\chi_{(-\infty,-r]}(y)\rd{y} \\
& + \int W'(x-y)\mu(y)\chi_{[r,\infty)}(y)\rd{y}. \\
\end{split}\end{equation}
Similar to \eqref{r0eps} and the argument below it, one can show the contribution from the first term
\begin{equation}\begin{split}
& \int_{r\le |x| \le 6R_1}\int W'(x-y)\Big(\rho^*(y)+\mu(y)\chi_{[-r,r]}(y)\Big)\rd{y}\rho(x)v(x)\rd{x} \\
\le & - \frac{c_\rho}{2}\lambda\Big(12R_1+\frac{4R_1}{c_\rho}\Big) \int_r^{6R_1} v(x)\rho(x)\rd{x}.
\end{split}\end{equation}
 The second term of \eqref{3R1} is positive, which means it gives negative contribution to \eqref{3R10}. The third term is possibly negative, and can be estimated by
\begin{equation}\begin{split}
& \left|\int W'(x-y)\mu(y)\chi_{[r,\infty)}(y)\rd{y}\right|\\
 \le & \left|\int W'(x-y)\mu(y)\chi_{[r,R_2]}(y)\rd{y}\right| + \left|\int W'(x-y)\mu(y)\chi_{[R_2,\infty)}(y)\rd{y}\right| \\
\le & \|W'\|_{L^\infty}\int_{2R_1}^{R_2} \mu(y)\rd{y} + \Lambda(R_2-6R_1),
\end{split}\end{equation}
using $r\ge 2R_1$.

Therefore we conclude
\begin{equation}\label{R1I}\begin{split}
 \frac{\rd}{\rd{t}}\Big|_{t=0}\cI[\rho_t] 
\le  & -\Big[\frac{c_\rho}{2}\lambda\Big(12R_1+\frac{4R_1}{c_\rho}\Big) -\|W'\|_{L^\infty}\int_{2R_1}^{R_2} \mu(y)\rd{y}\\
& - \Lambda(R_2-6R_1)\Big] \int_{r}^{6R_1}\frac{x-r}{4R_1}\rho(x)\rd{x}. \\
\end{split}\end{equation}
By taking $R_2=R_2(R_1)\sim R_1^{\alpha/\beta}$ large such that $\Lambda(R_2-6R_1) \le \frac{c_\rho}{2}\lambda\Big(12R_1+\frac{4R_1}{c_\rho}\Big)/2$ (c.f. assumption {\bf (A2)}), we get
\rev{
\begin{equation}\label{R1I1}\begin{split}
 \frac{\rd}{\rd{t}} &\Big|_{t=0}\cI[\rho_t] 
\le  -\left(\frac{c_\rho}{4}\lambda\Big(12R_1+\frac{4R_1}{c_\rho}\Big)-\|W'\|_{L^\infty}\int_{2R_1}^{R_2} \mu(y)\rd{y}\right) \int_{r}^{6R_1}\frac{x-r}{4R_1}\rho(x)\rd{x} \\
= & -\left(c\lambda(CR_1)-C\int_{2R_1}^{R_2} \mu(y)\rd{y}\right) \int_{r}^{6R_1}\frac{x-r}{4R_1}\rho(x)\rd{x} - c\lambda(CR_1) \int_{r}^{6R_1}\frac{x-r}{4R_1}\rho(x)\rd{x}.
\end{split}\end{equation}
}

Next we estimate the increment of the internal energy:
\begin{equation}\begin{split}
\frac{\rd}{\rd{t}}\Big|_{t=0}\cS[\rho_t] = & -\int m\rho(x)^{m-1}\partial_x(\rho(x)v(x))\rd{x} \\
= & m(m-1)\int \rho(x)^{m-1}\partial_x\rho(x)\cdot v(x)\rd{x} \\
= & (m-1)\int \partial_x(\rho(x)^{m})\cdot v(x)\rd{x} \\
= & -(m-1)\int \rho(x)^{m}\cdot \partial_xv(x)\rd{x}. \\
\end{split}\end{equation}
Notice that $\partial_x v(x)$ is an even function,
\begin{equation}
\partial_x v(x) = -\frac{1}{4R_1}\chi_{[r,6R_1]} + \frac{6R_1-r}{4R_1}\delta(x-6R_1),\quad x>0.
\end{equation}
Therefore (using $m>1$)
\begin{equation}\begin{split}
\frac{\rd}{\rd{t}}\Big|_{t=0}\cS[\rho_t] \le & \frac{2(m-1)}{4R_1}\int_{r}^{6R_1} \rho(x)^{m}\rd{x} \\
\le & \frac{C}{R_1}\left(\int_{r}^{6R_1} \mu(x)^{m}\rd{x} + \int_{r}^{6R_1}\rho^*(x)^{m}\rd{x} \right).\\
\end{split}\end{equation}
\rev{Adding to \eqref{R1I1}, we conclude
\begin{equation}\label{R1E}\begin{split}
\frac{\rd}{\rd{t}}\Big|_{t=0}E[\rho_t] 
\le & -\left(c\lambda(CR_1)-C\int_{2R_1}^{R_2} \mu(y)\rd{y}\right) \int_{r}^{6R_1}\frac{x-r}{4R_1}\rho(x)\rd{x} \\
& - c\lambda(CR_1) \int_{r}^{6R_1}\frac{x-r}{4R_1}\rho(x)\rd{x}+ \frac{C}{R_1}\left(\int_{r}^{6R_1} \mu(x)^{m}\rd{x} + \int_{r}^{6R_1}\rho^*(x)^{m}\rd{x} \right) \\
\le & -\left(c\lambda(CR_1)-C\int_{2R_1}^{R_2} \mu(y)\rd{y}\right) \int_{r}^{6R_1}\frac{x-r}{4R_1}\rho(x)\rd{x} \\
& -\frac{c}{R_1}\left(\lambda(CR_1) \int_{r}^{6R_1}(x-r)\rho(x)\rd{x}- C \int_{r}^{6R_1}\rho^*(x)^{m}\rd{x}\right) \\
& + \frac{C}{R_1}\int_{r}^{6R_1} \mu(x)^{m}\rd{x}.  \\
\end{split}\end{equation}
}

Next we estimate the second term in the last expression of \eqref{R1E}:
\begin{equation}\label{uselem_r}\begin{split}
& \lambda(CR_1) \int_{r}^{6R_1}(x-r)\rho^*(x)\rd{x} - C\int_{r}^{6R_1}\rho^*(x)^{m}\rd{x} .\\
\end{split}\end{equation}
By the choice of $r$ from Lemma \ref{lem_r}, 
\begin{equation}
\int_r^{6R_1}\rho^*(x)^m \rd{x} \le C\frac{1}{R_1}\rho^*(R_1)^{m-1}\int_r^{6R_1}(x-r)\rho^*(x) \rd{x}\le C\frac{1}{R_1^m}\int_r^{6R_1}(x-r)\rho^*(x) \rd{x},
\end{equation}
by noticing that $\rho^*(R_1) \le \frac{1}{R_1}$. Recall that $\lambda(R_1) = c_\alpha/R_1^\alpha$ by assumption {\bf (A2)}. Therefore, using $m>\alpha$ (assumption {\bf (A4)}), we can take $R_1$ large enough so that the second term in \eqref{uselem_r} can be absorbed by the first term, and lead to
\begin{equation}\label{R1E20}\begin{split}
\frac{\rd}{\rd{t}}\Big|_{t=0}E[\rho_t] 
\le & -\left(c\lambda(CR_1)-C\int_{2R_1}^{R_2} \mu(y)\rd{y}\right) \int_{r}^{6R_1}\frac{x-r}{4R_1}\rho(x)\rd{x}  + \frac{C}{R_1}\int_{r}^{6R_1} \mu(x)^{m}\rd{x}  \\
\le & -cR_1^{-\alpha}\left(1-CR_1^\alpha\int_{2R_1}^{R_2} \mu(y)\rd{y}\right) \int_{r}^{6R_1}(-v(x))\rho(x)\rd{x}  + \frac{C}{R_1}\int_{r}^{6R_1} \mu(x)\rd{x},  \\
\end{split}\end{equation}
by using $\|\rho\|_{L^\infty} \le C$. 

By applying \eqref{R1E20} to $\rho(t,\cdot)$ (the solution to \eqref{eq0}) for any $t\ge 0$ and using Lemma \ref{lem_basic}, we obtain, provided that the quantity inside the bracket below is negative, that
\begin{equation}\label{R1E2}\begin{split}
& \frac{\rd}{\rd{t}}E[\rho(t,\cdot)] \\
\le & -\frac{\left[-c_1R_1^{-\alpha}\left(1-C_1R_1^{\alpha}\int_{2R_1}^{R_2} \mu(t,y)\rd{y}\right) \int_{r(t)}^{6R_1}(-v(t,x))\rho(t,x)\rd{x}  + C_2R_1^{-1}\int_{r(t)}^{6R_1} \mu(t,x)\rd{x}\right]^2}{\int |v(t,x)|^2\rho(t,x)\rd{x}},  \\
\end{split}\end{equation}
where $r(t)$ and the velocity field $v(t,x)$ are determined by $\rho(t,\cdot)$, and $c_1,C_1,C_2$ are constants. Notice that $-v(t,x)$ is always non-negative for $x\in [r(t),6R_1]$. Then we define the sets $\cT_{1,2,3}\subset [0,\infty)$ as follows:
\begin{itemize}
\item $\cT_1$ contains those $t$ with $\int_{2R_1}^{R_2} \mu(t,y)\rd{y} \ge \frac{1}{2C_1R_1^\alpha}$. Proposition \ref{prop_nrd} shows that $|\cT_1| \le CR_2^{4\alpha+\frac{\alpha^2}{\beta}}$. It follows that
\begin{equation}
\int_{t\in \cT_1} \int_{r(t)}^{6R_1}(-v(t,x))\rho(t,x)\rd{x}\rd{t} \le CR_2^{4\alpha+\frac{\alpha^2}{\beta}}R_1^\alpha.
\end{equation}
\item $\cT_2$ contains those $t\notin \cT_1$ with $\int_{r(t)}^{6R_1}(-v(t,x))\rho(t,x)\rd{x} \le \frac{4C_2R_1^{-1}}{c_1R_1^{-\alpha}}\int_{r(t)}^{6R_1} \mu(t,x)\rd{x} $.  It follows that
\begin{equation}
\int_{t\in \cT_2} \int_{r(t)}^{6R_1}(-v(t,x))\rho(t,x)\rd{x}\rd{t} \le \frac{4C_2R_1^{-1}}{c_1R_1^{-\alpha}}\int_0^\infty\int_{r(t)}^{6R_1} \mu(t,x)\rd{x} \rd{t}\le CR_2^{4\alpha+\frac{\alpha^2}{\beta}}R_1^{\alpha-1},
\end{equation}
where the second inequality uses Proposition \ref{prop_nrd} and $r(t)\ge 2R_1$.
\item $\cT_3=[0,\infty)\backslash(\cT_1\cup\cT_2)$. It then follows that the quantity inside the bracket in \eqref{R1E2} is less than $-\frac{c_1}{4}\int_{r(t)}^{6R_1}(-v(t,x))\rho(t,x)\rd{x}$, and then
\begin{equation}\begin{split}
\frac{\rd}{\rd{t}}E[\rho(t,\cdot)] 
\le & -\frac{c_1^2}{16}R_1^{-2\alpha}\int_{r}^{6R_1}(-v(t,x))\rho(t,x)\rd{x}, \\
\end{split}\end{equation}
for $t\in \cT_3$, using since $|v(t,x)|\le 1$ and $-v(t,x)|_{[r(t),6R_1]}\ge 0$. Therefore
\begin{equation}
\int_{t\in \cT_3} \int_{r(t)}^{6R_1}(-v(x))\rho(t,x)\rd{x}\rd{t} \le CR_1^{2\alpha}(E(0)-E_\infty) = CR_1^{2\alpha}.
\end{equation}
\end{itemize}
Adding the three parts gives the final conclusion, by noticing that $R_2\sim R_1^{\alpha/\beta}$, and $-v(x) \ge 1/4$ for $x\in [5R_1,6R_1]$.

\end{proof}

\section{Quantitative energy dissipation rate estimate}

In this section we prove Theorem \ref{thm_main2}. We first prove Proposition \ref{prop_RCSS}, and then introduce the generalized $h(s)$-linear curve for possibly non-radially-decreasing distributions, which leads to the proof of Theorem \ref{thm_main2}.

\subsection{Energy dissipation from non-radially-decreasing parts}

\begin{proof}[Proof of Proposition \ref{prop_RCSS}]

In view of Corollary \ref{cor_RCSS}, it suffices to relate the RHS of \eqref{cor_RCSS_1} with $E(t)-E[\rho^\#(t,\cdot)]$. It is clear that $\cS[\rho(t,\cdot)] = \cS[\rho^\#(t,\cdot)]$, and we will analyze $\cI[\rho]-\cI[\rho^\#]$ (where the dependence on $t$ is omitted).

In correspondence to the decomposition \eqref{decomp}, we write 
\begin{equation}
\rho^\# = \rho^* + \mu_+^\# + \mu_-^\#,
\end{equation}
where $\mu_+^\#= (\rho^\#-\rho^*)\chi_{[0,\infty)}$ are some horizontal translation on each layer of $\mu_+$, having the same total mass, and similar for $\mu_-^\#$. Therefore
\begin{equation}\label{decomp1}\begin{split}
\cI[\rho]-\cI[\rho^\#] = (\cI[\rho^*,\mu_+] - \cI[\rho^*,\mu_+^\#]) + (\cI[\mu_+]-\cI[\mu_+^\#]) + (\cI[\mu_+,\mu_-] - \cI[\mu_+^\#,\mu_-^\#]).
\end{split}\end{equation}

We estimate these terms when $\mu_-=[c_--r_-,c_-+r_-],\,\mu_+=[c_+-r_+,c_++r_+]$, and the general case follows by linear superposition. In this case $\mu_+^\#=[c_+^\#-r_+,c_+^\#+r_+]$ for some $0<c_+^\#<c_+$.
\begin{equation}\begin{split}
\cI[\rho^*,\mu_+] - \cI[\rho^*,\mu_+^\#] = & \int \int_{[-r_0(h),r_0(h)]} \int_{[-r_+,r_+]}(W(c_++x-y)-W(c_+^\#+x-y))\rd{x}\rd{y}\rd{h} \\
= & \int \int_{[-r_0(h),r_0(h)]} \int_{[-r_+,r_+]}\int_{c_+^\#+x}^{c_++x}W'(z-y)\rd{z}\rd{x}\rd{y}\rd{h} \\
= & \int  \int_{[-r_+,r_+]}\int_{c_+^\#+x}^{c_++x}\int_{[-r_0(h),r_0(h)]} W'(z-y)\rd{y}\rd{z}\rd{x}\rd{h} \\
\le & \|W'\|_{L^\infty}\int  \int_{[-r_+,r_+]}\int_{c_+^\#+x}^{c_++x}\min\{r_0(h),z\}\rd{z}\rd{x}\rd{h} \\
= & \int_{[-r_+,r_+]}\int_{c_+^\#+x}^{c_++x}\int \min\{r_0(h),z\}\rd{h}\rd{z}\rd{x} \\
= &  \|W'\|_{L^\infty} \int_{[-r_+,r_+]}\int_{c_+^\#+x}^{c_++x}\int_0^z \rho^*(y)\rd{y}\rd{z}\rd{x} \\
\le &  \|W'\|_{L^\infty}\cdot \|\rho\|_{L^\infty} \int_{[-r_+,r_+]}\int_{c_+^\#+x}^{c_++x}z\rd{z}\rd{x} \\
\le & \|W'\|_{L^\infty}\cdot \|\rho\|_{L^\infty}r_+(c_++r_+)^2 \\
\le & \|W'\|_{L^\infty}\cdot \|\rho\|_{L^\infty}\int x^2 \mu_+(x)\rd{x},
\end{split}\end{equation}
where the first inequality uses the symmetry of $W$, and the second inequality uses the uniform $L^\infty$ boundedness of $\rho$ (and thus $\rho^*$). 

The other two terms in \eqref{decomp1} can be treated similarly, by replacing $\rho^*$ with an interval in $\mu_+$ or $\mu_-$ and do proper translation. Therefore we conclude
\begin{equation}\label{Ftau3_1}
\cI[\rho]-\cI[\rho^\#] \le C\int_0^\infty x^2 \mu(x)\rd{x},
\end{equation}
and the proof is finished.
\end{proof}

\subsection{$h(s)$-representation for general distributions}

In this subsection we introduce $h(s)$-representation for density distributions which are not necessarily radially-decreasing. 
We first introduce the $h(s)$-representation of $\rho(x)$, which can be viewed as a re-parametrization of the $h$-representation we introduced before. We start by defining the $h(s)$ function for $\rho(x)$: $h[\rho](s)$ is defined by 
\begin{equation}
\int \min\{\rho(x),h(s)\} \rd{x} = s,\quad 0\le s \le 1,
\end{equation}
which coincide with the definition in \cite{DYY} if $\rho$ is radially decreasing. Then we define\footnote{Here we are abusing the notation $\cC$: it refers to the definition in section \ref{sec_css} when the argument is the letter $h$, and the one here when the argument is the letter $s$.}
\begin{equation}
\cC[\rho](s) := \{x: \rho(x) > h(s)\},
\end{equation}
and it is clear that $\cC[\rho]$ satisfies the admissible relation:
\begin{equation}\label{consis}
s_1\le s_2 \text{ implies } \cC(s_2)\subset\cC(s_1),\quad |\cC(s)| = \frac{1}{h'(s)}.
\end{equation}
Furthermore, $\rho(x)$ can be recovered by
\begin{equation}
\rho(x) = h(\sup\{s: x\in\cC(s)\}),
\end{equation}
as long as \eqref{consis} holds. Notice that in the case of radially-decreasing distributions, since $\cC(s)$ is a symmetric interval, one has\footnote{The $'$ notation on an $h(s)$-representation always refers to the partial derivative with respect to $s$.}
\begin{equation}\label{ccrd}
\cC[\rho](s) = [-\frac{1}{2h'(s)},\frac{1}{2h'(s)}].
\end{equation}

The following lemma gives the expression of internal/interaction energy in terms of the $h(s)$-representation.
\begin{lemma}\label{lem_phihs}
For $\rho(x)$ with $h(s)$-representation $h(s)$,
\begin{equation}
\int \Phi(\rho(x))\rd{x} = \int \Phi'(h(s))\rd{s},
\end{equation}
and
\begin{equation}\label{intC}
\int\int W(x-y)\rho(x)\rho(y)\rd{y}\rd{x} = \int\int \left(\int_{\cC(s_1)}\int_{\cC(s_2)}W(x-y)\rd{y}\rd{x}\right)h'(s_1)h'(s_2)\rd{s_1}\rd{s_2}.
\end{equation}
\end{lemma}
The proof of this lemma is similar to that of Lemma \ref{lem_Phi} and is omitted here.

\subsection{$h(s)$-linear curve for radially-decreasing distributions}

In this subsection we review the $h(s)$-linear curve defined in \cite{DYY} for radially-decreasing distributions and refine their energy dissipation rate estimate, which leads to the proof of Proposition \ref{prop_RD}.

For radially-decreasing distributions $\rho_0(x)$ and $\rho_1(x)$ (for which $\cC[\rho](s)$ is determined by $h[\rho](s)$ via \eqref{ccrd}), the $h(s)$-linear curve $\rho_t$ defined in \cite{DYY} is given by its $h(s)$-representation
\begin{equation}
h_t(s) = (1-t)h_0(s) + t h_1(s),\quad 0\le t \le 1,
\end{equation}
where $h_0=h[\rho_0],\,h_1=h[\rho_1]$. It is shown in Theorem 2.6 of \cite{DYY} that the map $t\mapsto E[\rho_t]$ is strictly convex for $m\ge 2$, which is the key of proving the uniqueness of steady state. Now we improve this result and give a quantitative version:
\begin{lemma}\label{lem_RD1}
Let $\rho_0,\rho_1$ be radially-decreasing distributions supported on $[-R,R]$. Then $t\mapsto E[\rho_t]$ satisfies the strict convexity\footnote{Since $h'_t(s)$ is a linear function in $t$, $\partial_t h'_t(s)=h'_1(s)-h'_0(s)$ is constant in $t$, and we omitted the subscript $t$ in \eqref{lem_RD1_1} without ambiguity.}
\begin{equation}\label{lem_RD1_1}
\frac{\rd^2}{\rd{t}^2}E[\rho_t] \ge \frac{1}{32}\cdot\frac{\lambda(2R)}{R}\int_0^1 \frac{(\partial_t h'(s))^2}{(h'_t(s))^4} \rd{s}.
\end{equation}
\end{lemma}

\begin{proof}
The convexity $\frac{\rd^2}{\rd{t}^2}\cS[\rho_t]\ge 0$ for $m\ge 2$ is proved in Proposition 2.3 of \cite{DYY}. Therefore we only need to deal with the convexity of the interaction energy.

Define
\begin{equation}
W_a(x) = \left\{\begin{split}
& 0, \quad 0\le |x| < a \\
& 1,\quad |x|\ge a
\end{split}\right.,
\end{equation}
for $a>0$, and then one can write the interaction energy into a convex combination
\begin{equation}
\cI[\rho_t] = \frac{1}{2}\int_0^{2R} \int_0^1\int_0^1\cI_a(t;s_1,s_2) \rd{s_1}\rd{s_2} W'(a)\rd{a},
\end{equation}
where
\begin{equation}\label{Ia}
 \cI_a(t;s_1,s_2):= \int_{\cC(s_1)}\int_{\cC(s_2)} W_a(x-y)\rd{y}\rd{x}h'_t(s_1)h'_t(s_2).
\end{equation}
We will give a lower bound of $\partial_{tt}\cI_a(t;s_1,s_2)$ for fixed $a,s_1,s_2,t$.

It is shown in the proof of Theorem 2.6 of \cite{DYY} that
\begin{equation}
\partial_{tt}I_a(t;s_1,s_2) \ge 0,\quad \forall a,s_1,s_2,t,
\end{equation}
and furthermore, in the case 
\begin{equation}\label{cond_a}
|\frac{1}{2f_1}-\frac{1}{2f_2}|<a,\quad \frac{1}{2f_1} + \frac{1}{2f_2}>a,\quad f_i(t) = h'_t(s_i),\,i=1,2,
\end{equation}
one has
\begin{equation}\label{dttI}
\partial_{tt}I_a(t;s_1,s_2) = \frac{f_2(\partial_t f_1)^2}{2f_1^3} + \frac{f_1(\partial_t f_2)^2}{2f_2^3} + \Big(2a^2-\frac{1}{2f_1^2}-\frac{1}{2f_2^2}\Big)(\partial_t f_1)(\partial_t f_2).
\end{equation}
It is clear that the condition \eqref{cond_a} guarantees that the RHS of \eqref{dttI} is a positive-definite quadratic form in $\partial_t f_1,\, \partial_t f_2$.

We estimate the lower bound of \eqref{dttI} by
\begin{equation}\label{dttI1}\begin{split}
\partial_{tt}I_a(t;s_1,s_2) = & \frac{1-|\kappa|}{2}\Big(\frac{f_2(\partial_t f_1)^2}{f_1^3} + \frac{f_1(\partial_t f_2)^2}{f_2^3}\Big) + \frac{f_1f_2}{2}|\kappa|\Big(\frac{\partial_t f_1}{f_1^2} + \sgn(\kappa)\frac{\partial_t f_2}{f_2^2}\Big)^2 \\
\ge & \frac{1-|\kappa|}{2}\Big(\frac{f_2(\partial_t f_1)^2}{f_1^3} + \frac{f_1(\partial_t f_2)^2}{f_2^3}\Big),
\end{split}\end{equation}
where
\begin{equation}
\kappa = \kappa_a = f_1f_2\Big(2a^2-\frac{1}{2f_1^2}-\frac{1}{2f_2^2}\Big),
\end{equation}
and notice that \eqref{cond_a} implies $|\kappa|<1$:
\rev{\begin{equation}\begin{split}
1+\kappa = & 2f_1f_2\Big(\frac{1}{2f_1f_2} + a^2-\frac{1}{4f_1^2}-\frac{1}{4f_2^2}\Big) = 2f_1f_2\Big(a^2-(\frac{1}{2f_1}-\frac{1}{2f_2})^2\Big) \\
= & 2f_1f_2\Big(a+(\frac{1}{2f_1}-\frac{1}{2f_2})\Big)\Big(a-(\frac{1}{2f_1}-\frac{1}{2f_2})\Big) > 0, \\
1-\kappa = & 2f_1f_2\Big(\frac{1}{2f_1f_2} - a^2+\frac{1}{4f_1^2}+\frac{1}{4f_2^2}\Big) = -2f_1f_2\Big(a^2-(\frac{1}{2f_1}+\frac{1}{2f_2})^2\Big)\\
= & -2f_1f_2\Big(a+(\frac{1}{2f_1}+\frac{1}{2f_2})\Big)\Big(a-(\frac{1}{2f_1}+\frac{1}{2f_2})\Big) > 0. \\
\end{split}\end{equation}}

Integrating \eqref{dttI1} in $s_1,s_2,a$ gives
\begin{equation}\begin{split}
& \frac{\rd^2}{\rd{t}^2}\cI[\rho_t] \\
\ge & \frac{1}{2}\int_0^1\int_0^1\Big(\frac{f_2(\partial_t f_1)^2}{f_1^3} + \frac{f_1(\partial_t f_2)^2}{f_2^3}\Big)     \int_{a \text{ satisfying } \eqref{cond_a}}  \frac{1-|\kappa_a|}{2} W'(a)\rd{a} \rd{s_1}\rd{s_2} \\
\ge & \frac{1}{2}\lambda(2R)\int_0^1\int_0^1\Big(\frac{f_2(\partial_t f_1)^2}{f_1^3} + \frac{f_1(\partial_t f_2)^2}{f_2^3}\Big)     \int_{a \text{ satisfying } \eqref{cond_a}}  \frac{1-|\kappa_a|}{2} \rd{a} \rd{s_1}\rd{s_2} \\
= & \lambda(2R)\iint_{0\le s_1\le s_2\le 1}\Big(\frac{f_2(\partial_t f_1)^2}{f_1^3} + \frac{f_1(\partial_t f_2)^2}{f_2^3}\Big)     \int_{a \text{ satisfying } \eqref{cond_a}}  \frac{1-|\kappa_a|}{2} \rd{a} \rd{s_1}\rd{s_2} .\\
\end{split}\end{equation}

Now we estimate the last inner integral in $a$, for given $s_1,s_2$. In fact, the condition $s_1\le s_2$ for the last out integral implies $f_1 \le f_2$ since $h'_t(\cdot)$ is non-decreasing. Therefore by requiring $|\kappa_a| < \frac{1}{2}$ we get
\begin{equation}\label{cond_a2}
a^2 < \frac{1}{4f_1^2}+\frac{1}{4f_2^2} + \frac{1}{4f_1f_2},\quad a^2 > \frac{1}{4f_1^2}+\frac{1}{4f_2^2} - \frac{1}{4f_1f_2}.
\end{equation}
Notice that $\frac{1}{4f_1}+\frac{1}{4f_2^2} - \frac{1}{4f_1f_2} \le \frac{1}{4f_1^2}$, and thus the range
\begin{equation}
\frac{1}{2f_1} < a < \frac{1}{2f_1} + \frac{1}{4f_2},
\end{equation}
satisfies \eqref{cond_a2}. Thus we get
\begin{equation}
\int_{a \text{ satisfying }\eqref{cond_a}}(1-|\kappa_a|) \rd{a} \ge \frac{1}{2}\cdot \frac{1}{4f_2} = \frac{1}{8f_2},
\end{equation}
and we conclude
\begin{equation}\begin{split}
\frac{\rd^2}{\rd{t}^2}\cI[\rho_t] \ge & \frac{\lambda(2R)}{16} \iint_{0\le s_1\le s_2\le 1}\Big(\frac{(\partial_t f_1)^2}{f_1^3} + \frac{f_1(\partial_t f_2)^2}{f_2^2}\Big) \rd{s_1}\rd{s_2} \\
= & \frac{\lambda(2R)}{16} \int_0^1 ((1-s)h'_t(s)+h_t(s))\frac{(\partial_t h'(s))^2}{(h'_t(s))^4} \rd{s}.
\end{split}\end{equation}

If for some $s$ one has $h_t(s) \le \frac{1}{2R}$, then since  by the assumption on the support, $\frac{1}{h'_t(0)} \le R$, one has $h'_t(0) \ge \frac{1}{R}$ which implies $s \le \frac{1}{2}$ by the increasing property of $h'_t(\cdot)$. Then it follows that $(1-s)h'_t(s) \ge \frac{1}{2}h'_t(s)  \ge \frac{1}{2R}$. Therefore we get the conclusion.

\end{proof}

Next we estimate the cost of the $h(s)$-linear curve for radially-decreasing distributions, which is a quantitative version of Proposition 2.2 of \cite{DYY}:
\begin{lemma}\label{lem_RDv}
Let $\rho_0,\rho_1$ be radially-decreasing distributions supported on $[-R,R]$. Then the $h(s)$-linear curve $\rho_t$ satisfies
\begin{equation}
\partial_t \rho_t + \partial_x(\rho_t v_t) = 0,
\end{equation}
with
\begin{equation}
v_t(x) = \frac{x(h_1(s_{x,t})-h_0(s_{x,t}))}{\rho_t(x)},\quad x>0,
\end{equation}
where $s_{x,t}$ is defined implicitly by $\frac{1}{2h'_t(s_{x,t})}=x$. It satisfies the estimate 
\begin{equation}\label{lem_RDv_1}
\int |v_t(x)|^2 \rho_t(x)\rd{x} \le \frac{1}{6}R\|\rho_1\|_{L^\infty}\cdot\int_0^{1} \frac{(\partial_t h'(s))^2}{h'_t(s)^3h'_1(s)}\rd{s}.
\end{equation}
\end{lemma}

\begin{proof}
The first claim was proved in equation (4.8) of \cite{DYY}. To prove the estimate \eqref{lem_RDv_1},
\begin{equation}\begin{split}
\int_0^\infty |v_t(x)|^2 \rho_t(x)\rd{x} = & \int_0^R  \frac{x^2}{h_t(s_{x,t })^2} ((h_0-h_1)(s_{x,t }))^2 \rho_t(x)\rd{x} \\
= & \int_0^R  \frac{x^2}{h_t (s_{x,t })} \Big(\int_0^{s_{x,t }} (h_0'-h_1')(s)\rd{s}\Big)^2 \rd{x} \\
= & \int_0^R  x^2\frac{1}{h_t (s_{x,t })} \Big(\int_0^{s_{x,t }} \frac{(h_0'-h_1')(s)}{\sqrt{h'_1(s)}}\sqrt{h'_1(s)}\rd{s}\Big)^2 \rd{x} \\
\le & \int_0^R x^2 \frac{1}{h_t (s_{x,t })}\Big(\int_0^{s_{x,t }}h'_1(s)\rd{s}\Big)\Big(\int_0^{s_{x,t }} \frac{(h_0'-h_1')(s)^2}{h'_1(s)}\rd{s}\Big)  \rd{x} \\
= & \int_0^R x^2 \frac{h_1(s_{x,t })}{h_t (s_{x,t })}\int_0^{s_{x,t }} \frac{(h_0'-h_1')(s)^2}{h'_1(s)}\rd{s}  \rd{x} \\
\le & 2R\|\rho_1\|_{L^\infty}\cdot\int_0^R x^2 \int_0^{s_{x,t }} \frac{(h_0'-h_1')(s)^2}{h'_1(s)}\rd{s}  \rd{x} \\
= & 2R\|\rho_1\|_{L^\infty}\cdot\int_0^{1} \frac{(h_0'-h_1')(s)^2}{h'_1(s)}\int_0^{1/2h'_t(s)} x^2\rd{x}\rd{s}  \\
= & \frac{1}{12}R\|\rho_1\|_{L^\infty}\cdot\int_0^{1} \frac{(h_0'-h_1')(s)^2}{h'_t(s)^3h'_1(s)}\rd{s}  \\
= & \frac{1}{12}R\|\rho_1\|_{L^\infty}\cdot\int_0^{1} \frac{(\partial_t h'(s))^2}{h'_t(s)^3h'_1(s)}\rd{s},  \\
\end{split}\end{equation}
where we used $\frac{h_1(s)}{h_t (s)}\le 2R\|\rho_1\|_{L^\infty}$: to see this, first notice $h_t(s) \ge s/(2R)$ since $\rho_t$ is supported on $[-R,R]$. Then notice that $h_1(s)$ is a convex function with $h_1(0) = 0$ and $h_1(1) = \|\rho_1\|_{L^\infty}$. Thus $h_1(s)\le s\|\rho_1\|_{L^\infty} $.  
\end{proof}

We combine the above two lemmas and give the proof of Proposition \ref{prop_RD}:
\begin{proof}[Proof of Proposition \ref{prop_RD}]
It suffices to prove the case $t_0=0$. Take the $h(s)$-linear curve with $\rho_0=\rho_{\ini}$ and $\rho_1=\rho_\infty$. Lemma \ref{lem_RD1} gives
\begin{equation}
f(t) \ge c\frac{\lambda(R)}{R}\int_0^1 \frac{(h'_0(s)-h'_1(s))^2}{(h'_t(s))^4} \rd{s},\quad f(t):=\frac{\rd^2}{\rd{t}^2}E[\rho_t].
\end{equation}
Since $\rho_1$ is the unique energy minimizer, one has
\begin{equation}
E[\rho_t]\ge E[\rho_1] = E[\rho_t],
\end{equation}
and therefore
\begin{equation}
F(1) \le 0,\quad F(t) := \frac{\rd}{\rd{t}} E[\rho_t].
\end{equation}
Notice that $F'(t)=f(t)$. Therefore
\begin{equation}\label{Ftau}
F(t) = F(1) - \int_t^1 f(\tilde{t})\rd{\tilde{t}}.
\end{equation}
Therefore
\begin{equation}\label{Ftau0}\begin{split}
E[\rho_0]-E[\rho_1] = & -\int_0^1 F(t)\rd{t} = -F(1) + \int_0^1 \int_t^1 f(\tilde{t})\rd{\tilde{t}}\rd{t}\\
= & -F(1) + \int_0^1 t f(t)\rd{t} \\
\le & -F(1) + \int_0^1 f(t)\rd{t} \\
= & -F(0). \\
\end{split}\end{equation}

The integral on RHS of \eqref{lem_RD1_1}, integrated in $t$, can be calculated as
\begin{equation}\begin{split}
& \int_0^1 \int_0^1 \frac{(\partial_t h'(s))^2}{(h'(s))^4} \rd{s}\rd{t} \\
= &  \int_0^1 \int_0^1 \frac{(\partial_t h'(s))^2}{(h'_0(s) + t\partial_t h'(s))^4} \rd{t} \rd{s}\\
= &  \frac{1}{3}\int_0^1  (\partial_t h'(s))(\frac{1}{h'_0(s)^3}-\frac{1}{h'_1(s)^3})  \rd{s}\\
= &  \frac{1}{3}\int_0^1  (\partial_t h'(s))^2\frac{h_0'(s)^2+h_0'(s)h_1'(s)+h_1'(s)^2}{h'_0(s)^3h'_1(s)^3}  \rd{s}\\
\ge &  \frac{1}{3}\int_0^1  (\partial_t h'(s))^2\frac{1}{h'_0(s)^3h'_1(s)}  \rd{s}.\\
\end{split}\end{equation}
Thus \eqref{Ftau} gives
\begin{equation}\label{Ftau1}
F(0) \le -c\frac{\lambda(R)}{R}\int_0^1  (\partial_t h'(s))^2\frac{1}{h'_0(s)^3h'_1(s)}  \rd{s}.
\end{equation}

\rev{Therefore, applying Lemma \ref{lem_basic} with Lemma \ref{lem_RDv} (with $t=0$) gives
\begin{equation}\begin{split}
\frac{\rd}{\rd{t}}\Big|_{t=0}E[\rho(t,\cdot)] \le & -\frac{|F(0)|^2}{CR\int_0^1  (\partial_t h'(s))^2\frac{1}{h'_0(s)^3h'_1(s)}  \rd{s}} \\
\le & - c\frac{\lambda(R)}{R^2}(-F(0))\le - c\frac{\lambda(R)}{R^2}(E[\rho_0]-E[\rho_1]).
\end{split}\end{equation}}
\end{proof}

\subsection{Generalized $h(s)$-linear curve}\label{sec_ghlc}

In this subsection we generalize the $h(s)$-linear curve we discussed above to general distributions.

We consider a density distribution $\rho_0(x)$ which is not necessarily radially-decreasing. Denote $\cC[\rho_0](s) = \bigcup_{j}I_{j},\,I_{j} = [c_j-r_j,c_j+r_j]$ as a disjoint union of intervals, which follows the same symmetry notations as we did in section \ref{sec_css}. We further assume the total number of $I_j$ is finite, and $c_1<c_2<\cdots$ (the general case can be treated via limit procedure which we will omit). For its Steiner symmetrization $\rho_0^\#(x)$, it is clear that $|\cC[\rho_0](s)|=|\cC[\rho_0^\#](s)|$, which allows us to decompose $\cC[\rho_0^\#](s)=\bigcup_{j}I_{j}^\#$ with $I_{j}^\# = [c_j^\#-r_j,c_j^\#+r_j]$ and $c_1^\#<c_2^\#<\cdots$.

With another density distribution $\rho_1(x)$ which is radially-decreasing, we first notice that the $h(s)$-linear curve $\rho^\#_t$ from $\rho_0^\#$ to $\rho_1$ defined in the previous subsection can be written as
\begin{equation}
\cC[\rho^\#_t](s) =  [-\frac{1}{2h'_t(s)},\frac{1}{2h'_t(s)}] = \bigcup_j I_{j,t}^\# ,\quad h_t(s) = (1-t)h_0(s) + th_1(s),
\end{equation}
where the union is disjoint with
\begin{equation}
I_{j,t}^\# = [c_{j,t}^\#-r_{j,t},c_{j,t}^\#+r_{j,t}],\quad r_{j,t} = r_j\cdot \frac{h'_0(s)}{h'_t(s)},\quad c_{j,t}^\# = c_j\cdot \frac{h'_0(s)}{h'_t(s)}
\quad j\ge 0.
\end{equation}
In other words, we decompose $\cC[\rho^\#_0](s)$ into small intervals and change each interval to keep the size proportion between them.

Fix a large parameter $M>1$. We define the generalized $h(s)$-linear curve $\rho_t$ for small $t>0$ by
\begin{equation}\begin{split}
& h_t(s) = (1-t)h_0(s) + t h_1(s), \\
& \forall j\ge 0,\quad r_{j,t} = r_j\cdot \frac{h'_0(s)}{h'_t(s)}, \\
& \forall j\ge 1,\quad c_{j,t} = \left\{\begin{split}
& c_j - M(c_j^\#-c_{j,t}^\#) - (M-1)(r_j-r_{j,t}),\quad \partial_t h'(s) > 0 \\
& c_j+(r_j-r_{j,t}),\quad \partial_t h'(s)\le 0 
\end{split}\right.,
\end{split}\end{equation}
where we denote $\cC[\rho_t](s) = \bigcup_{j}I_{j,t},\,I_{j,t} = [c_{j,t}-r_{j,t},c_{j,t}+r_{j,t}]$, and the cases $j<0$ are defined by symmetry. See Figure \ref{fig5} as illustration.

\begin{figure}
\begin{center}
  \includegraphics[width=.45\linewidth]{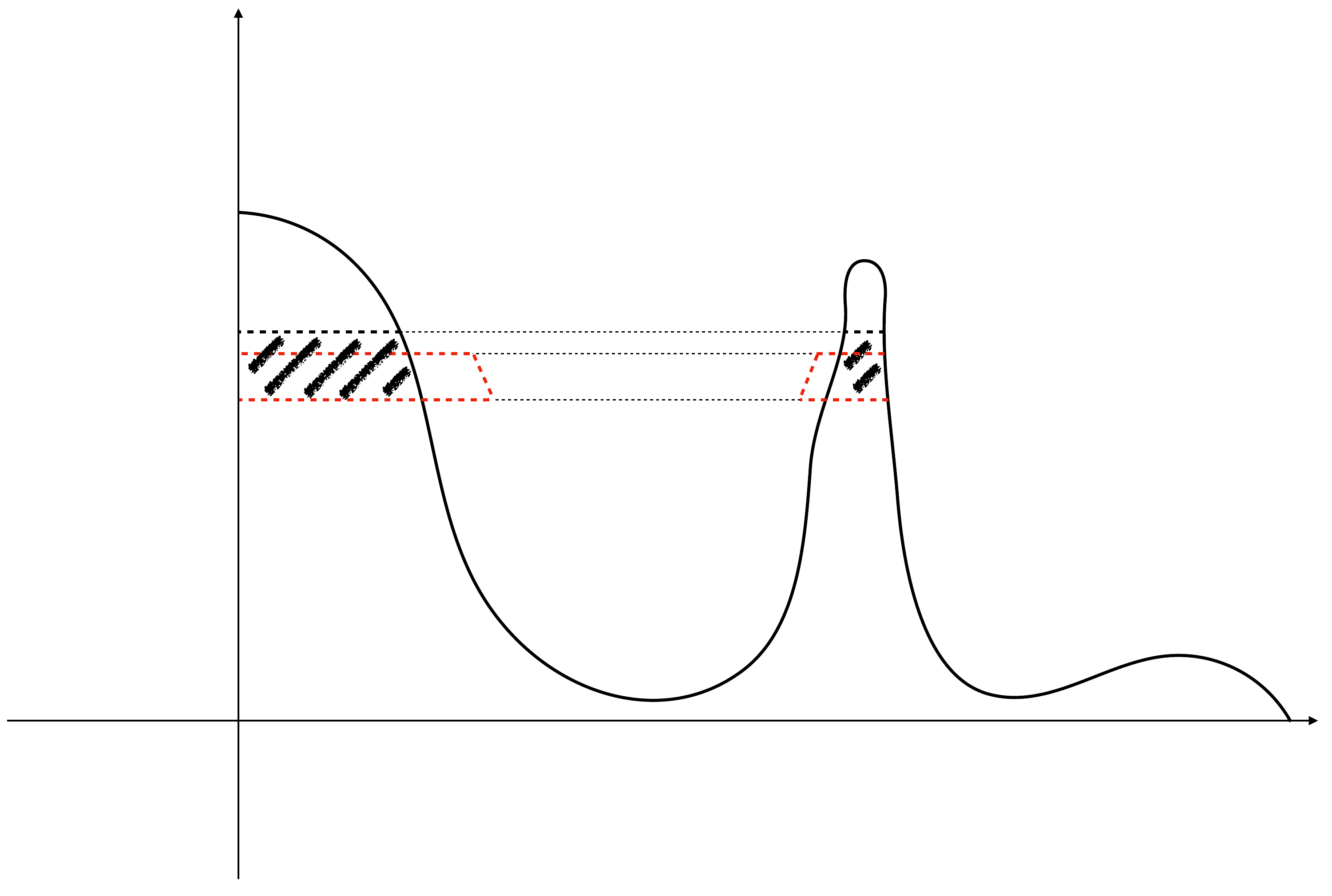}  
  \includegraphics[width=.45\linewidth]{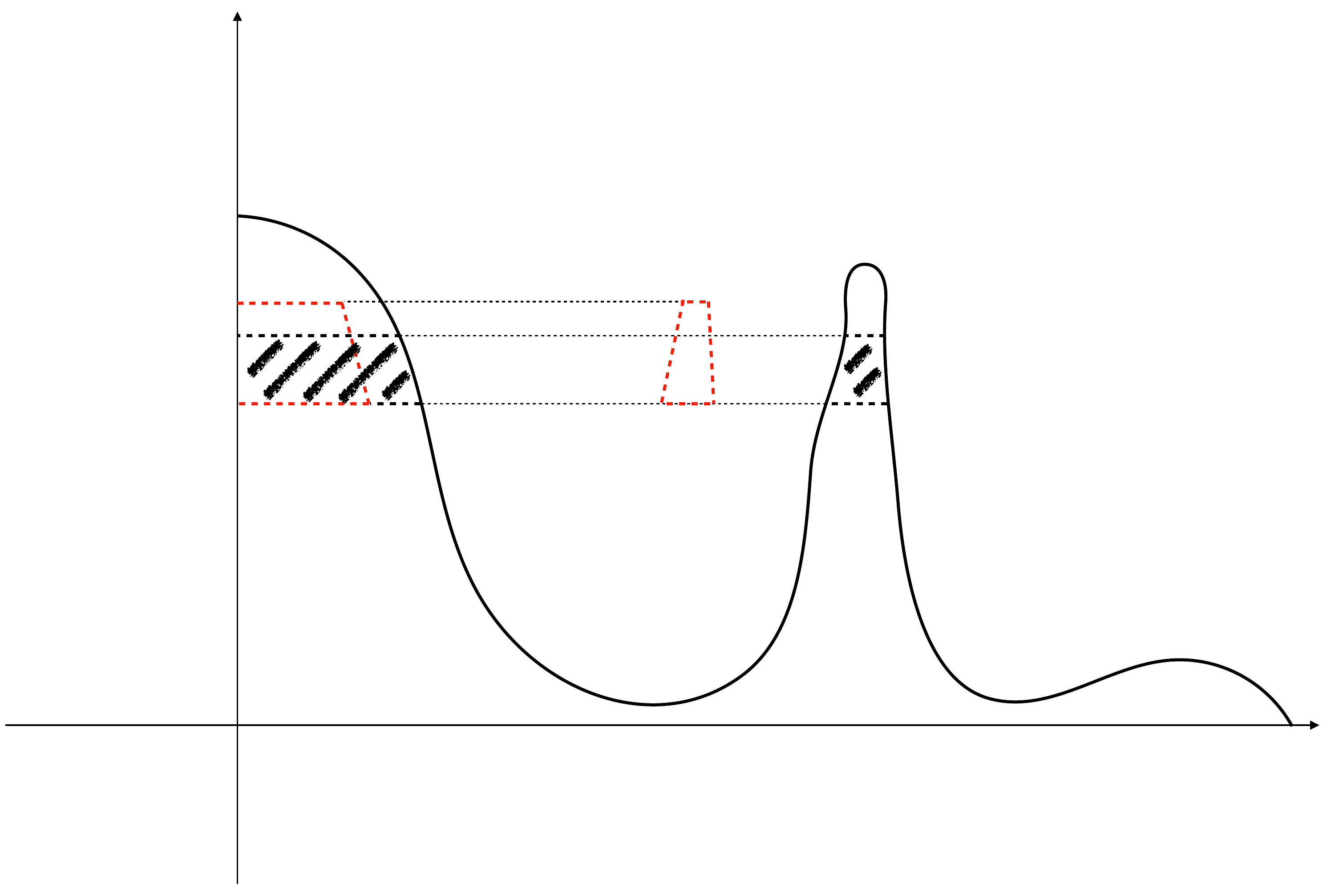}  
  \caption{Illustration of the generalized $h(s)$-linear curve: the shaded region is transported and becomes the region bounded by the red dashed curve. Left: case $\partial_t h'(s) \le 0$ (expansion); Right: case $\partial_t h'(s) > 0$ (compression).}
\label{fig5}
\end{center}
\end{figure}

We give some explanation about the definition of the generalized $h(s)$-linear curve. 
\begin{enumerate}
\item This curve imitates the original $h(s)$-linear curve for radially-decreasing distributions, in the sense that $h_t(s)$ (for those $t$ such that it is well-defined) coincides with the  $h(s)$-linear curve between $\rho_0^\#$ and $\rho_1$ we defined previously. We make the size of every interval $I_j(s)$ changing as proportion.
\item For those $s$ with $\partial_t h'(s)>0$, the total size of $\cC[\rho_t](s) = 1/h_t'(s)$ is decreasing. In this case the curve from $\rho_0^\#$ is moving towards the center. To imitate this, we define the movement of $I_j$ so that the right endpoint $c_{j,t}+r_{j,t}$ is moving towards the center at a speed which is $M$ times as fast as $c_{j,t}^\#+r_{j,t}^\#$. This implies
\begin{equation}\label{dhspos}
\text{Case $\partial_t h'(s)>0$:}\quad |\partial_t (c_{j,t}+r_{j,t}x) |\ge M |\partial_t (c_{j,t}^\#+r_{j,t}x) |,\quad \forall -1\le x \le 1,
\end{equation}
where both quantities inside the absolute values are negative
\item For those $s$ with $\partial_t h'(s)\le 0$, the total size of $\cC[\rho_t](s) = 1/h_t'(s)$ is increasing. In this case the curve from $\rho_0^\#$ is moving away from the center. In this case, we define the movement of $I_j$ so that the right endpoint of $I_{j,t}$ stays the same (at $c_j+r_j$).
\end{enumerate}

Now we analyze the energy change for the generalized $h(s)$-linear curve. For the purpose of proving Theorem \ref{thm_main2}, here we should think of $\rho(x)$ with $\int_0^\infty x^2\mu(x)\rd{x}$ being small, since we already have enough energy dissipation from Proposition \ref{prop_RCSS} if this quantity is large.
\begin{lemma}\label{lem_RDper}
If $\rho(x)$ is supported on $[-R,R],\,R\ge 1$ and has $h(s)$-linear curve $\rho_t(x)$ to $\rho_\infty(x)$ for $0<t <t_1,\,t_1>0$, and one takes $M=\frac{\|W'\|_{L^\infty}}{\lambda(2R)}$, then
\begin{equation}
\frac{\rd}{\rd{t}}\Big|_{t=0} E[\rho_t] \le \frac{\rd}{\rd{t}}\Big|_{t=0} E[\rho^\#_t] + CMR^{2/3}\int_0^\infty x^2\mu(x)\rd{x},
\end{equation}
where $\rho^\#_t$ denotes the $h(s)$-linear curve from $\rho^\#$ to $\rho_\infty$.
\end{lemma}

\begin{proof}
Throughout this proof, the subscript 0 on $h$ will be omitted.

We first notice that the internal energy parts $\cS$ of $\rho_t$ and $\rho^\#_t$ are the same, by Lemma \ref{lem_phihs}, because they share the same $h_t(s)$. So are their $t$ derivatives.

For the interaction energy part, we will use the representation formula \eqref{intC}. We denote $\cC_t(s)=\cC[\rho_t](s)$ and $\cC^\#_t(s)=\cC[\rho^\#_t](s)$, and notice that $\cC^\#_t(s)=[-\frac{1}{2h_t'(s)},\frac{1}{2h_t'(s)}]$ is the unique interval in centered at 0 with length equal to $|\cC_t(s)|$. Also notice that both $\cC_t(s)$ and $\cC^\#_t(s)$ contain the middle interval $I_{0,t}(s)$.
\begin{equation}\label{linper}\begin{split}
\cI[\rho_t] = & \frac{1}{2}\int\int W(x-y)\rho_t(x)\rho_t(y)\rd{y}\rd{x}\\
= & \frac{1}{2}\int\int \left(\int_{\cC_t(s_1)}\int_{\cC_t(s_2)}W(x-y)\rd{y}\rd{x}\right)h_t'(s_1)h_t'(s_2)\rd{s_1}\rd{s_2} \\
= & \frac{1}{2}\int\int \left(\int_{\cC^\#_t(s_1)}\int_{\cC^\#_t(s_2)}W(x-y)\rd{y}\rd{x}\right)h_t'(s_1)h_t'(s_2)\rd{s_1}\rd{s_2} \\
& + \int\int \left(\int  ({\bf 1}_{\cC_t(s_1)\backslash I_{0,t}(s_1)}(x) - {\bf 1}_{\cC^\#_t(s_1)\backslash I_{0,t}(s_1)}(x))\int_{\cC^\#_t(s_2)}W(x-y)\rd{y}\rd{x}\right)\\
& \cdot h_t'(s_1)h_t'(s_2)\rd{s_1}\rd{s_2} \\
& + \frac{1}{2}\int\int \left(\int_{\cC_t(s_1)\backslash I_{0,t}(s_1)}\int_{\cC_t(s_2)\backslash I_{0,t}(s_2)}W(x-y)\rd{y}\rd{x}\right)h_t'(s_1)h_t'(s_2)\rd{s_1}\rd{s_2} \\
& - \int\int \left(\int_{\cC_t(s_1)\backslash I_{0,t}(s_1)}\int_{\cC^\#_t(s_2)\backslash I_{0,t}(s_2)}W(x-y)\rd{y}\rd{x}\right)h_t'(s_1)h_t'(s_2)\rd{s_1}\rd{s_2} \\
& + \frac{1}{2}\int\int \left(\int_{\cC^\#_t(s_1)\backslash I_{0,t}(s_1)}\int_{\cC^\#_t(s_2)\backslash I_{0,t}(s_2)}W(x-y)\rd{y}\rd{x}\right)h_t'(s_1)h_t'(s_2)\rd{s_1}\rd{s_2} ,\\
\end{split}\end{equation}
where the first term on the RHS is $\cI[\rho^\#_t]$. The second term is the change of interaction energy against $\rho^\#_t$ when one changes from $\cC_t(s_1)$ to $\cC^\#_t(s_1)$, and can be viewed as linear in $\mu(x)$. The other three terms can be viewed as quadratic terms in $\mu(x)$.

{\bf STEP 1}: Linear terms.

We first estimate the time derivative of the linear terms (the second term on the RHS of \eqref{linper}). 
\begin{equation}\label{keylin}\begin{split}
&  \frac{\rd}{\rd{t}}\Big|_{t=0}\left(\int  ({\bf 1}_{\cC_t(s_1)\backslash I_0(s_1)}(x) - {\bf 1}_{\cC^\#_t(s_1)\backslash I_0(s_1)}(x))\int_{\cC^\#_t(s_2)}W(x-y)\rd{y}\rd{x}h_t'(s_1)h_t'(s_2)\right) \\
= &  2\frac{\rd}{\rd{t}}\Big|_{t=0}\left(\int  \sum_{j\ge 1}({\bf 1}_{I_j}(x) - {\bf 1}_{I_j^\#}(x))\int_{\cC^\#_t(s_2)}W(x-y)\rd{y}\rd{x}h_t'(s_1)h_t'(s_2)\right) \\
= & \frac{\rd}{\rd{t}}\Big|_{t=0}\Big(\int_{-1}^1\int_{-1}^1 \sum_{j\ge 1}\Big[W((c_{j,t}+r_{j,t}x)-\frac{1}{2h'_t(s_2)}y)\\
& -W((c_{j,t}^\#+r_{j,t}x)-\frac{1}{2h'_t(s_2)}y)\Big]\rd{y}\rd{x}\Big)r_jh'(s_1) \\
= & \int_{-1}^1\int_{-1}^1 \Big[W'((c_j+r_jx)-\frac{1}{2h'(s_2)}y)\partial_t|_{t=0}(c_{j,t} + r_{j,t} x) \\
& -W'((c_j^\#+r_jx)-\frac{1}{2h'(s_2)}y)\partial_t|_{t=0}(c^\#_{j,t} + r_{j,t} x) \Big]\rd{y}\rd{x}\cdot r_j h'(s_1) \\
& + \int_{-1}^1\int_{-1}^1 \Big[W'((c_j+r_jx)-\frac{1}{2h'(s_2)}y)-W'((c_j^\#+r_jx)-\frac{1}{2h'(s_2)}y)\Big] \\
& \cdot\frac{\partial_th'(s_2)}{2h'(s_2)^2}y\rd{y}\rd{x}\cdot r_j h'(s_1), \\
\end{split}\end{equation}
where in the first equality we use the symmetry between $j>0$ and $j<0$; in the second equality we use change of variables and the fact $h'_t(s_1)r_{j,t} = r_j h'(s_1)$ (which means all intervals $I_j$ change size in proportion). 

{\bf STEP 1.1}: We show that the first term (which comes from the movement of $I_j$ and $I_j^\#$) of the RHS of \eqref{keylin} is negative. We separate into the cases of $\partial_t h'(s_1)>0$ and $\partial_t h'(s_1)\le 0$:
\begin{itemize}
\item If $\partial_t h'(s_1)>0$ (contraction). In this case we know that $\partial_t|_{t=0}(c_{j,t}^\# + r_{j,t} x)$ and $\partial_t|_{t=0}(c_{j,t} + r_{j,t} x)$ are negative and satisfy \eqref{dhspos}. 

We take a fixed $j\ge 1$. We estimate the following positive quantity
\begin{equation}\begin{split}
&\int_{-1}^1 W'((c_j+r_jx)-\frac{1}{2h'(s_2)}y)\rd{y} \le 2\min\{c_j+r_jx, \frac{1}{2h'(s_2)}\}\cdot \|W'\|_{L^\infty},
\end{split}\end{equation}
by using $W'(-x)=-W'(x)$. Similarly
\begin{equation}\begin{split}
&\int_{-1}^1 W'((c_j^\#+r_jx)-\frac{1}{2h'(s_2)}y)\rd{y} \ge 2\min\{c_j^\#+r_jx, \frac{1}{2h'(s_2)}\}\cdot \lambda(2R).
\end{split}\end{equation}
Therefore by taking $M = \frac{\|W'\|_{L^\infty}}{\lambda(2R)}$ and using $c_j^\# \le c_j$, we obtain
\begin{equation}\begin{split}
& \int_{-1}^1\int_{-1}^1 \Big[W'((c_j+r_jx)-\frac{1}{2h'(s_2)}y)\partial_t|_{t=0}(c_{j,t} + r_{j,t} x) \\
& -W'((c_j^\#+r_jx)-\frac{1}{2h'(s_2)}y)\partial_t|_{t=0}(c^\#_{j,t} + r_{j,t} x) \Big]\rd{y}\rd{x}\cdot r_j h'(s_1) \le 0 .\\
\end{split}\end{equation}

\item If $\partial_t h'(s_1)\le 0$ (expansion). In this case we have $\partial_t|_{t=0}(c_{j,t}^\# + r_{j,t} x)\ge 0$, $\partial_t|_{t=0}(c_{j,t} + r_{j,t} x)\le 0$. Then we obtain
\begin{equation}\begin{split}
 & \int_{-1}^1\int_{-1}^1 \Big[W'((c_j+r_jx)-\frac{1}{h'(s_2)}y)\partial_t|_{t=0}(c_{j,t} + r_{j,t} x) \\
& -W'((c_j^\#+r_jx)-\frac{1}{h'(s_2)}y)\partial_t|_{t=0}(c^\#_{j,t} + r_{j,t} x) \Big]\rd{y}\rd{x} \le 0.
\end{split}\end{equation}

\end{itemize}

{\bf STEP 1.2}: We estimate the second term (which comes from the movement $\cC^\#(s_2)$) on the RHS of \eqref{keylin}.

For a fixed $x$, if $c_j+r_jx \le \frac{1}{2h'(s_2)}$, then we first symmetrize in $y$ for the part with $c_j$ and obtain
\begin{equation}\begin{split}
& \int_{-1}^1 \Big[W'((c_j+r_jx)-\frac{1}{2h'(s_2)}y)-W'(-\frac{1}{2h'(s_2)}y)\Big] y\rd{y} \\
= & \int_0^1 \Big[W'((c_j+r_jx)-\frac{1}{2h'(s_2)}y)- W'((c_j+r_jx)+\frac{1}{2h'(s_2)}y) - 2W'(-\frac{1}{2h'(s_2)}y)\Big] y\rd{y} \\
= & \int_0^1 \Big[W'((c_j+r_jx)-\frac{1}{2h'(s_2)}y)+ W'(-(c_j+r_jx)-\frac{1}{2h'(s_2)}y) - 2W'(-\frac{1}{2h'(s_2)}y)\Big] y\rd{y} ,\\
\end{split}\end{equation}
where we inserted a term $W'(-\frac{1}{2h'(s_2)}y)$ which does not depend on $c_j$. Notice that the last quantity in the bracket is a centered difference (with all arguments in $W'$ being negative): it satisfies
\begin{equation}\begin{split}
& \Big|W'((c_j+r_jx)-\frac{1}{2h'(s_2)}y)+ W'(-(c_j+r_jx)-\frac{1}{2h'(s_2)}y) - 2W'(-\frac{1}{2h'(s_2)}y)\Big| \\
\le & \|W'''\|_{L^\infty} (c_j+r_jx)^2,
\end{split}\end{equation}
by using Taylor expansion of $W'$ at $-\frac{1}{2h'(s_2)}y$ with the assumption {\bf (A3)}. By similar trick for the term with $c_j^\#$, noticing that the inserted terms are cancelled and $c_j^\# \le c_j$, we obtain
\begin{equation}\begin{split}
\left|\int_{-1}^1 (W'((c_j+r_jx)-\frac{1}{2h'(s_2)}y)-W'((c_j^\#+r_jx)-\frac{1}{2h'(s_2)}y)) y\rd{y}\right| \le \|W'''\|_{L^\infty} (c_j+r_jx)^2. \\
\end{split}\end{equation}

If $c_j+r_jx > \frac{1}{2h'(s_2)}$, then we estimate by 
\begin{equation}\begin{split}
&\left|\int_{-1}^1 (W'((c_j+r_jx)-\frac{1}{2h'(s_2)}y)-W'((c_j^\#+r_jx)-\frac{1}{2h'(s_2)}y))y\rd{y}\right| \le 2\|W'\|_{L^\infty}.\\
\end{split}\end{equation}

Combining the above two cases, we obtain
\begin{equation}\begin{split}
& \int_{-1}^1\int_{-1}^1 (W'((c_j+r_jx)-\frac{1}{2h'(s_2)}y)-W'((c_j^\#+r_jx)-\frac{1}{2h'(s_2)}y)) \cdot\frac{\partial_th'(s_2)}{h'(s_2)^2}y\rd{y}\rd{x} \\
\le & \Big(\|W'''\|_{L^\infty} (c_j+r_jx)^2 + 2\|W'\|_{L^\infty}\int_{c_j+r_jx > \frac{1}{2h'(s_2)}}\rd{x}\Big) \cdot\frac{\partial_th'(s_2)}{h'(s_2)^2}.
\end{split}\end{equation}

{\bf STEP 1.3}: Finalize the linear terms. Combining the result in STEP 1.1 and STEP 1.2, summing over $j$ and integrating in $s_1,s_2$, we obtain
\begin{equation}\begin{split}
& \frac{\rd}{\rd{t}}\Big|_{t=0}\int\int \int  ({\bf 1}_{\cC_t(s_1)\backslash I_0(s_1)}(x) - {\bf 1}_{\cC^\#_t(s_1)\backslash I_0(s_1)}(x))\int_{\cC^\#_t(s_2)}W(x-y)\rd{y}\rd{x}\\
& \cdot h_t'(s_1)h_t'(s_2)\rd{s_1}\rd{s_2} \\
\le & C\int\int \int_{\cC(s_1)\backslash I_0(s_1)}x^2\rd{x} h'(s_1)  \rd{s_1}\rd{s_2} \\
& + C\int\int \int_{x\in \cC(s_1)\backslash I_0(s_1): x>\frac{1}{2h'(s_2)}}\rd{x} \cdot\frac{\partial_t h'(s_2)}{h'(s_2)^2}h'(s_1)  \rd{s_1}\rd{s_2}. \\
\end{split}\end{equation}
Now we estimate the above two terms separately. We first recognize that the first term is exactly $C\int x^2\mu(x)\rd{x}$. To estimate the second term,
\begin{equation}\begin{split}
& \int\int \int_{x\in \cC(s_1)\backslash I_0(s_1): x>\frac{1}{2h'(s_2)}}\rd{x} \cdot\frac{\partial_t h'(s_2)}{h'(s_2)^2}h'(s_1)  \rd{s_1}\rd{s_2} \\
= & \int\int_{x\in \cC(s_1)\backslash I_0(s_1)} \int_{s_2:x>\frac{1}{2h'(s_2)}}\frac{\partial_t h'(s_2)}{h'(s_2)^2}\rd{s_2} \rd{x} h'(s_1)  \rd{s_1}.\\
\end{split}\end{equation}
To estimate the inner integral, 
\begin{equation}
\int_{s_2:x>\frac{1}{2h'(s_2)}}\frac{ \partial_t h'(s_2)}{h'(s_2)^2}\rd{s_2} \le \int_{s_2:x>\frac{1}{2h'(s_2)}}(h'(s_2)+h'_1(s_2))\rd{s_2}\cdot x^2  \le (\|\rho_0\|_{L^\infty}+\|\rho_\infty\|_{L^\infty})x^2,
\end{equation}
and it follows that
\begin{equation}\begin{split}
& \int\int \int_{x\in \cC(s_1)\backslash I_0(s_1): x>\frac{1}{2h'(s_2)}}\rd{x} \cdot\frac{\partial_t h'(s_2)}{h'(s_2)^2}h'(s_1)  \rd{s_1}\rd{s_2} \le  C\int x^2\mu(x)\rd{x}.
\end{split}\end{equation}


{\bf STEP 2}: Quadratic terms.

Finally we estimate the quadratic terms (the last three terms of the RHS of \eqref{linper}). We will estimate the first of them (with $\cC$ interaction with $\cC$) and the other two terms with $\cC^\#$ can be estimated similarly.

 For fixed  $s_1,s_2$ and intervals $I_j=[c_j-r_j,c_j+r_j]\subset \cC(s_1)\backslash I_0(s_1)$ and $I_k=[c_k-r_k,c_k+r_k]\subset \cC(s_2)\backslash I_0(s_2)$. 
\begin{equation}\begin{split}
& \frac{\rd}{\rd{t}}\Big|_{t=0}\left(\int_{I_j}\int_{I_k}W(x-y)\rd{y}\rd{x}h_t'(s_1)h_t'(s_2)\right) \\
= & \frac{\rd}{\rd{t}}\Big|_{t=0}\left(\int_{-1}^1\int_{-1}^1 W((c_j+r_jx)-(c_k+r_ky))\rd{y}\rd{x}\right)r_jh'(s_1)r_kh'(s_2) \\
= & \int_{-1}^1\int_{-1}^1 W'((c_j+r_jx)-(c_k+r_ky)) \partial_t|_{t=0}((c_j+r_jx)-(c_k+r_ky))\rd{y}\rd{x}r_jh'(s_1)r_kh'(s_2). \\
\end{split}\end{equation}
By construction of $\rho_t$, we have
\begin{equation}\label{Mbound}
|\partial_t|_{t=0}((c_j+r_jx)| \le M\frac{|\partial_t h'(s_1)|}{h'(s_1)^2},\quad |\partial_t|_{t=0}((c_k+r_ky)| \le M\frac{|\partial_t h'(s_2)|}{h'(s_2)^2}.
\end{equation}
In fact, for $\partial_t h'(s)>0$, this follows from \eqref{dhspos}. For $\partial_t h'(s)\le 0$, this follows from $|\partial_t|_{t=0}(c_{j,t}+r_{j,t}\tilde{x})| \le 2|\partial_t|_{t=0}r_{j,t}| \le \frac{|\partial_t h'(s)|}{2h'(s)^2}$ since $c_{j,t}+r_{j,t}$ does not move.  Therefore
\begin{equation}\begin{split}
& \Big|\frac{\rd}{\rd{t}}\Big|_{t=0}\Big(\int_{I_j}\int_{I_k}W(x-y)\rd{y}\rd{x}h_t'(s_1)h_t'(s_2)\Big)\Big| \\
\le & \|W'\|_{L^\infty}\cdot M\Big(\frac{|\partial_t h'(s_1)|}{h'(s_1)^2} + \frac{|\partial_t h'(s_2)|}{h'(s_2)^2}\Big)r_jh'(s_1)r_kh'(s_2).
\end{split}\end{equation}

The two summand above are symmetric with respect to $s_1$ and $s_2$, and thus we only need to estimate one of them. 
\begin{equation}\label{mu20}\begin{split}
& \int\int \sum_{j,k}\frac{|\partial_t h'(s_1)|}{h'(s_1)^2} r_jh'(s_1)r_kh'(s_2) \rd{s_1}\rd{s_2} \\
= & \int \sum_{j} r_j\frac{|\partial_t h'(s_1)|}{h'(s_1)^2}h'(s_1) \rd{s_1} \cdot \int \sum_k r_kh'(s_2)\rd{s_2} \\
= & \int |\cC(s_1)\backslash I_0(s_1)|\frac{|\partial_t h'(s_1)|}{h'(s_1)^2}h'(s_1) \rd{s_1} \cdot \int |\cC(s_2)\backslash I_0(s_2)|h'(s_2)\rd{s_2} \\
= & \int |\cC(s_1)\backslash I_0(s_1)|\frac{|\partial_t h'(s_1)|}{h'(s_1)^2}h'(s_1) \rd{s_1} \cdot \int \mu(x)\rd{x}. \\
\end{split}\end{equation}

To estimate the first integral in the last term, we use $|\partial_t h'(s)| = |h'(s)-h'_1(s)| \le h'(s)+h'_1(s)$ to decompose it into two parts. The first part is
\begin{equation}\label{mu21}\begin{split}
& \int |\cC(s)\backslash I_0(s)|\frac{h'(s)}{h'(s)^2}h'(s) \rd{s} \le 2 \int x\mu(x)\rd{x}. \\
\end{split}\end{equation}
To estimate the term with $h'_1$, we use the fact that $\partial_x\rho_\infty(0)=0,\,|\partial_{xx}\rho_\infty(0) | < \infty$ from Lemma \ref{lem_reg2}. This implies 
\begin{equation}\label{hs13}
h'_1(s) \le C(1-s)^{-1/3},\quad\forall 0\le s \le 1,
\end{equation}
 by using a Taylor expansion of $\rho_\infty$ at 0 to handle small $s$, and enlarging $C$ to handle large $s$ if necessary (see Figure \ref{fig6}). Fix $\delta>0$ to be chosen, and $\epsilon=0.01$.
\begin{equation}\begin{split}
& \int_{1-\delta}^1 |\cC(s)\backslash I_0(s)|\frac{h'_1(s)}{h'(s)^2}h'(s) \rd{s} \\
\le & C\int_{1-\delta}^1 |\cC(s)\backslash I_0(s)|\frac{1}{h'(s)}(1-s)^{-1/3} \rd{s} \\
\le & C \left(\int_{1-\delta}^1|\cC(s)\backslash I_0(s)|\frac{1}{h'(s)^2}h'(s) \rd{s}\right)^{\frac{2-\epsilon}{3}}\left(\int_{1-\delta}^1|\cC(s)\backslash I_0(s)|\frac{1}{h'(s)}(1-s)^{-1/(1+\epsilon)} \rd{s}\right)^{\frac{1+\epsilon}{3}} \\
\le & CR^{2(1+\epsilon)/3}\left(\int x^2\mu(x)\rd{x} \right)^{\frac{2-\epsilon}{3}}\delta^{\epsilon/3},
\end{split}\end{equation}
using $|\cC(s)\backslash I_0(s)| \le R,\,\frac{1}{h'(s)}\le R$.
\begin{equation}\begin{split}
& \int_0^{1-\delta} |\cC(s)\backslash I_0(s)|\frac{1}{h'(s)}(1-s)^{-1/3} \rd{s} \\
\le & C \left(\int_0^{1-\delta}|\cC(s)\backslash I_0(s)|\frac{1}{h'(s)^2}h'(s) \rd{s}\right)^{\frac{2+\epsilon}{3}}\left(\int_0^{1-\delta}|\cC(s)\backslash I_0(s)|\frac{1}{h'(s)}(1-s)^{-1/(1-\epsilon)} \rd{s}\right)^{\frac{1-\epsilon}{3}} \\
\le & CR^{2(1-\epsilon)/3}\left(\int x^2\mu(x)\rd{x} \right)^{\frac{2+\epsilon}{3}}\delta^{-\epsilon/3},
\end{split}\end{equation}
Choosing $\delta = R^{-2}\int x^2\mu(x)\rd{x}$ gives
\begin{equation}\label{mu22}
\int |\cC(s)\backslash I_0(s)|\frac{h'(s)}{h'(s)^2}h'(s) \rd{s} \le CR^{2/3}\left(\int x^2\mu(x)\rd{x} \right)^{2/3}.
\end{equation}

\begin{figure}
\begin{center}
  \includegraphics[width=.8\linewidth]{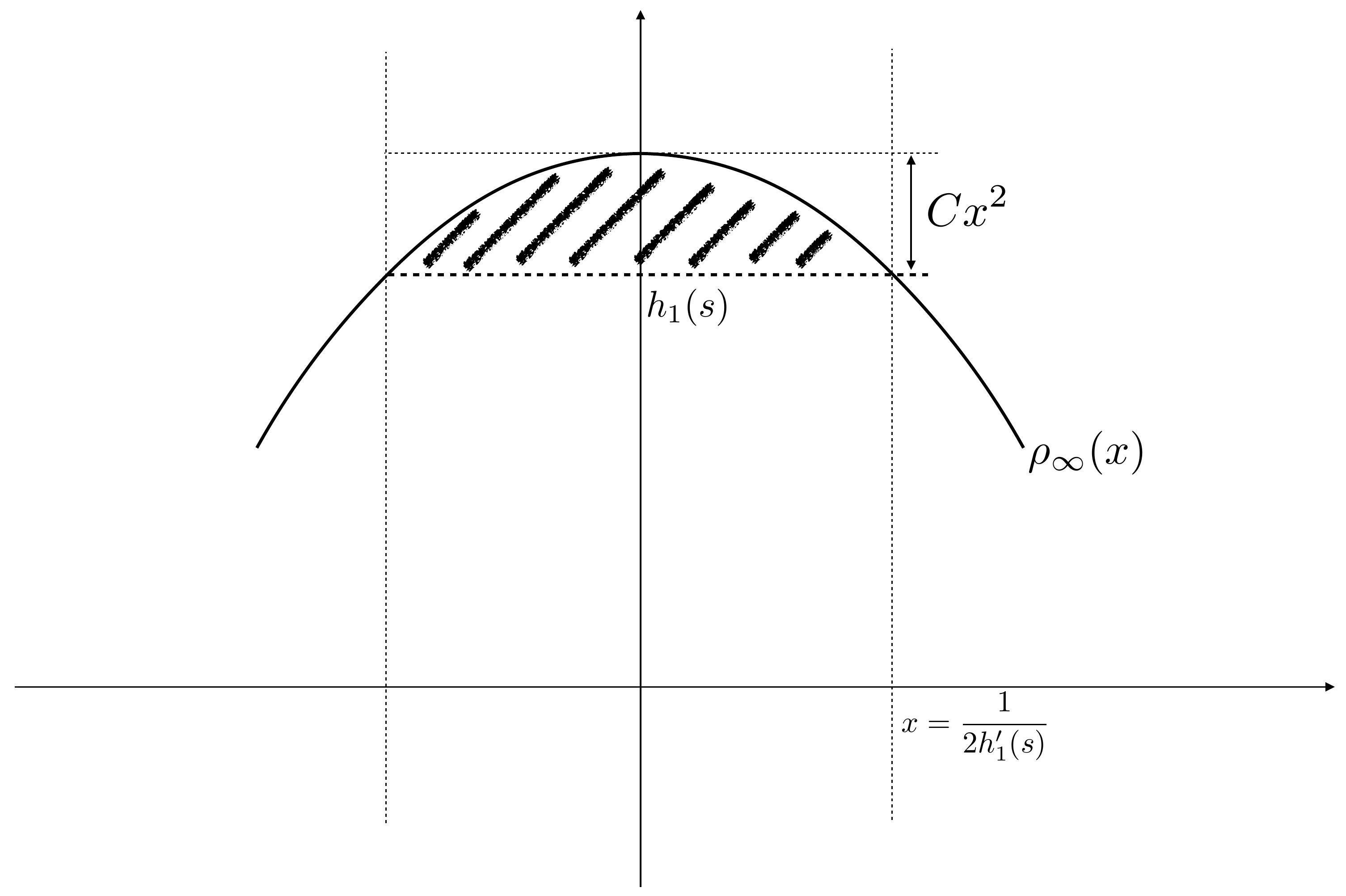}  
  \caption{Explanation of \eqref{hs13}: the shaded area is $1-s$, which is of the same size as a rectangle with side lengths $2x$ and $Cx^2$, with $x=\frac{1}{2h'_1(s)}$.}
\label{fig6}
\end{center}
\end{figure}

Using \eqref{mu21} and \eqref{mu22} in \eqref{mu20}, combined with Lemma \ref{lem_mucomp} stated below, we obtain
\begin{equation}\begin{split}
& \frac{\rd}{\rd{t}}\Big|_{t=0}\left(\int_{I_j}\int_{I_k}W(x-y)\rd{y}\rd{x}h_t'(s_1)h_t'(s_2)\right) \le  CM(1+R^{2/3})\cdot \int x^2\mu(x)\rd{x},
\end{split}\end{equation}
and the proof is finished.

\end{proof}

\begin{lemma}\label{lem_mucomp}
Let $\mu(x)\in L^\infty(0,\infty)$ be non-negative. Then
\begin{equation}\label{lem_mucomp_1}
\int_0^\infty x^2\mu(x)\rd{x} \ge \frac{1}{3\|\mu\|_{L^\infty}^2}\Big(\int_0^\infty \mu(x)\rd{x}\Big)^{3},
\end{equation}
and
\begin{equation}\label{lem_mucomp_2}
\int_0^\infty x^2\mu(x)\rd{x} \ge \frac{2^{3/2}}{3\|\mu\|_{L^\infty}^{1/2}}\Big(\int_0^\infty x\mu(x)\rd{x}\Big)^{3/2}.
\end{equation}
\end{lemma}
\begin{proof}
We first prove \eqref{lem_mucomp_1}. We write
\begin{equation}
\int_0^\infty x^2\mu(x)\rd{x} = \iint_A x^2\rd{x}\rd{y},\quad A = \{(x,y): x\in (0,\infty), 0\le y \le \mu(x)\}.
\end{equation}
We view $|A|=\int_0^\infty \mu(x)\rd{x}$ as a constraint, and minimize $\iint_A x^2\rd{x}\rd{y}$. By the $L^\infty$ bound on $\mu$, we have
\begin{equation}
A\subset (0,\infty)\times [0,\|\mu\|_{L^\infty}],
\end{equation}
and $\iint_A x^2\rd{x}\rd{y}$ is clearly minimized when $A$ contains those points with the weight $x^2$ as small as possible, i.e., 
\begin{equation}
A=\Big(0,\frac{\int_0^\infty \mu(x)\rd{x}}{\|\mu\|_{L^\infty}}\Big)\times [0,\|\mu\|_{L^\infty}].
\end{equation}
In this case one can verify that  \eqref{lem_mucomp_1} achieves the equality.

We then prove \eqref{lem_mucomp_2}. We write
\begin{equation}
\int_0^\infty x^2\mu(x)\rd{x} = \iint_B x\rd{x}\rd{y},\quad B = \{(x,y): x\in (0,\infty), 0\le y \le x\mu(x)\}.
\end{equation}
We view $|B|=\int_0^\infty x\mu(x)\rd{x}$ as a constraint, and minimize $\iint_B x\rd{x}\rd{y}$. By the $L^\infty$ bound on $\mu$, we have
\begin{equation}
B\subset \{(x,y): x\in (0,\infty), 0\le y \le x\|\mu\|_{L^\infty}\},
\end{equation}
and $\iint_B x^2\rd{x}\rd{y}$ is clearly minimized when $B$ contains those points with the weight $x$ as small as possible, i.e., 
\begin{equation}
B= \{(x,y): x\in (0,b), 0\le y \le x\|\mu\|_{L^\infty}\},
\end{equation}
where $b$ is determined by the constraint $|B|=\int_0^\infty x\mu(x)\rd{x}$:
\begin{equation}
b = \Big(\frac{2\int_0^\infty x\mu(x)\rd{x}}{\|\mu\|_{L^\infty}}\Big)^{1/2}.
\end{equation}
In this case one can verify that  \eqref{lem_mucomp_2} achieves the equality.

\end{proof}

The following lemma gives the cost of the generalized $h(s)$-linear curve (for simplicity, we only treat the case $t=0$, and the subscript 0 on $h$ is omitted):
\begin{lemma}\label{lem_RDW}
The generalized $h(s)$-linear curve $\rho_t$ from $\rho_0$ to $\rho_\infty$ satisfies (at $t=0$)
\begin{equation}\label{lem_RDW_0}
\partial_t\rho_t + \partial_x(\rho_t v) = 0,\quad v(x) = \frac{1}{\rho(x)} \int_{0}^{s_x} w(x,s)h'(s)\rd{s},\,x>0,\quad v(-x) = -v(x),
\end{equation}
where $w(x,s)$ is given by
\begin{equation}
w(x,s) = \left\{\begin{split}
& -\frac{\partial_t h'(s)}{2h'(s)^2}\cdot(x\cdot 2h'(s)),\quad h(s) \le \rho^*(x) \\
& \partial_t|_{t=0}(c_{j,t}+r_{j,t}\tilde{x}),\quad h(s) > \rho^*(x) \\
\end{split}\right.,
\end{equation}
where $[c_j-r_j,c_j+r_j]$ is the interval in $\cC(s)$ containing $x=c_{j}+r_{j}\tilde{x},\,-1\le \tilde{x}\le 1$. We have the estimate
\begin{equation}\label{lem_RDW_1}
\int |v(x)|^2 \rho_0(x)\rd{x} \le CRM^2\int_0^{1} \frac{(\partial_t h'(s))^2}{h'(s)^3h'_1(s)}\rd{s} .\end{equation}
\end{lemma}

\begin{proof}
The velocity field $w(x,s)$ describes the curve $\rho_t$ at level $s$ with location $s$, and we omit the proof of \eqref{lem_RDW_0} which is similar to Lemma \ref{lem_CSScost}.

Now we prove \eqref{lem_RDW_1}. Let $s_x$ be defined by $h[\rho](s_x) = \rho(x)$, and similarly define $s_x^*$ and $s_x^\#$. For the contribution from the radially-decreasing part,
\begin{equation}\begin{split}
\Big(\int_0^{s^*_x} w(x,s)h'(s)\rd{s}\Big)^2 = & x^2 \Big(\int_0^{s^*_x} \partial_th'(s)\rd{s}\Big)^2 \le x^2 \Big(\int_0^{s^*_x} \frac{(\partial_th'(s))^2}{h_1'(s)}\rd{s}\Big)\Big(\int_0^{s^*_x} h_1'(s)\rd{s}\Big)\\
 = & x^2 \Big(\int_0^{s^*_x} \frac{(\partial_th'(s))^2}{h_1'(s)}\rd{s}\Big) h_1(s^*_x).
\end{split}\end{equation}

For the contribution from the non-radially-decreasing part, we first notice from \eqref{Mbound} that 
\begin{equation}
|w(x,s)| \le M\frac{|\partial_t h'(s)|}{2h'(s)^2}.
\end{equation}
Therefore
\begin{equation}\begin{split}
\Big(\int_{s^*_x}^{s_x} w(x,s)h'(s)\rd{s}\Big)^2 \le & M^2\Big(\int_{s^*_x}^{s_x} \frac{|\partial_t h'(s)|}{h'(s)}\rd{s}\Big)^2 \\
\le & M^2\Big(\int_{s^*_x}^{s_x} \frac{(\partial_t h'(s))^2}{h'(s)^2h_1'(s)}\rd{s}\Big)\Big(\int_{s^*_x}^{s_x}h_1'(s)\rd{s}\Big)  \\
= & M^2\Big(\int_{s^*_x}^{s_x} \frac{(\partial_t h'(s))^2}{h'(s)^2h_1'(s)}\rd{s}\Big)(h_1(s_x)-h_1(s^*_x)).  \\
\end{split}\end{equation}

Combining the above two parts,
\begin{equation}\label{vM}\begin{split}
\frac{1}{2}\int |v(x,t)|^2 \rho_t(x)\rd{x} \le & \int_0^R \frac{1}{\rho(x)} x^2 \Big(\int_0^{s^*_x} \frac{(\partial_th'(s))^2}{h_1'(s)}\rd{s}\Big) h_1(s^*_x) \rd{x} \\
& + \int_0^R \frac{1}{\rho(x)} M^2\Big(\int_{s^*_x}^{s_x} \frac{(\partial_t h'(s))^2}{h'(s)^2h_1'(s)}\rd{s}\Big)(h_1(s_x)-h_1(s^*_x)) \rd{x}.
\end{split}\end{equation}

To estimate the first term in \eqref{vM}, we first notice that $s^*_x \le s^\#_x$ and $s^*_x \le s_x$. Therefore
\begin{equation}
h_1(s^*_x) \le h_1(s_x) \le \|\rho_\infty\|_{L^\infty}\cdot s_x \le \|\rho_\infty\|_{L^\infty}\cdot 2Rh(s_x)=\|\rho_\infty\|_{L^\infty}\cdot 2R\rho(x).
\end{equation}
Then
\begin{equation}\begin{split}
& \int_0^R \frac{1}{\rho(x)} x^2 \Big(\int_0^{s^*_x} \frac{(\partial_th'(s))^2}{h_1'(s)}\rd{s}\Big) h_1(s^*_x) \rd{x} \\
\le & CR\int_0^R  x^2 \Big(\int_0^{s^\#_x} \frac{(\partial_th'(s))^2}{h_1'(s)}\rd{s}\Big)\rd{x} \\
= & CR\int_0^{1} \frac{(\partial_t h'(s))^2}{h'(s)^3h'_1(s)}\rd{s},
\end{split}\end{equation}
where the last equality follows as the radially-decreasing case.

To estimate the second term in \eqref{vM},
\begin{equation}\begin{split}
& \int_0^R \frac{1}{\rho(x)} M^2\Big(\int_{s^*_x}^{s_x} \frac{(\partial_t h'(s))^2}{h'(s)^2h_1'(s)}\rd{s}\Big)(h_1(s_x)-h_1(s^*_x)) \rd{x} \\
\le & CRM^2\int_0^R\Big(\int_{s^*_x}^{s_x} \frac{(\partial_t h'(s))^2}{h'(s)^2h_1'(s)}\rd{s}\Big) \rd{x} \\
= & CRM^2\int_0^1 \frac{(\partial_t h'(s))^2}{h'(s)^2h_1'(s)} \int_{x:s^*_x\le s \le s_x} \rd{x} \rd{s}\\
= & CRM^2\int_0^1 \frac{(\partial_t h'(s))^2}{h'(s)^2h_1'(s)} |\cC(s)\backslash I_0(s)| \rd{s}\\
\le & CRM^2\int_0^1 \frac{(\partial_t h'(s))^2}{h'(s)^3h_1'(s)}  \rd{s}.\\
\end{split}\end{equation}

\end{proof}

Now we prove Theorem \ref{thm_main2}:
\begin{proof}[Proof of Theorem \ref{thm_main2}]
It suffices to prove for $t=0$. We consider the generalized $h(s)$-linear curve with $\rho_0=\rho_{\ini}$ and $\rho_1=\rho_\infty$, with $M=\frac{\|W'\|_{L^\infty}}{\lambda(2R)}$ chosen as in Lemma \ref{lem_RDper}. 

Recall the proof of Proposition \ref{prop_RD} (c.f. \eqref{Ftau1}) that 
 the $h(s)$-linear curve $\rho^\#_t$ from $\rho^\#$ to $\rho_\infty$ satisfies
\begin{equation}\label{Ftau2}
\frac{\rd}{\rd{t}}\Big|_{t=0}E[\rho^\#_t] \le -c\frac{\lambda(2R)}{R}\int_0^1  \frac{(\partial_t h'(s))^2}{h'(s)^3h'_1(s)}  \rd{s},
\end{equation}
and there holds the estimate
\begin{equation}\label{Ftau3}
\frac{\rd}{\rd{t}}\Big|_{t=0}E[\rho^\#_t] \le -(E[\rho^\#_0]-E_\infty).
\end{equation}

Then Lemma \ref{lem_RDper}, Lemma \ref{lem_RDW} with Lemma \ref{lem_basic} gives
\begin{equation}\begin{split}
\frac{\rd}{\rd{t}}\Big|_{t=0}E(t) \le & -\frac{\Big[-\frac{\rd}{\rd{t}}\Big|_{t=0}E[\rho^\#_t] - CMR^{2/3}\int_0^\infty x^2\mu(x)\rd{x}\Big]^2}{CRM^2\int_0^1 \frac{(\partial_t h'(s))^2}{h'(s)^3h'_1(s)}\rd{s}}, \\
\end{split}\end{equation}
as long as the quantity in the bracket is positive.

\begin{itemize}
\item \rev{ If $CMR^{2/3}\int_0^\infty x^2\mu(x)\rd{x}\le -\frac{1}{2}\frac{\rd}{\rd{t}}\Big|_{t=0}E[\rho^\#_t]$, then  the bracket is bounded below by $-\frac{1}{2}\frac{\rd}{\rd{t}}\Big|_{t=0}E[\rho^\#_t]$. Then we get
\begin{equation}\begin{split}
\frac{\rd}{\rd{t}}\Big|_{t=0}E(t) \le  & -c\frac{(-\frac{\rd}{\rd{t}}\Big|_{t=0}E[\rho^\#_t] )^2}{CRM^2\int_0^1 \frac{(\partial_t h'(s))^2}{h'(s)^3h'_1(s)}\rd{s}} \le -c\frac{\lambda(2R)^3}{R^2}\Big(-\frac{\rd}{\rd{t}}\Big|_{t=0}E[\rho^\#_t] \Big) \\
 \le & -c\frac{\lambda(2R)^3}{R^2}(E[\rho^\#_0]-E_\infty),\\
\end{split}\end{equation}
where the second inequality uses \eqref{Ftau2} and the last inequality uses \eqref{Ftau3}.
}
\item \rev{Otherwise we have
\begin{equation}
\int_0^\infty x^2\mu(x)\rd{x}> \frac{1}{CMR^{2/3}}\Big(-\frac{\rd}{\rd{t}}\Big|_{t=0}E[\rho^\#_t]\Big) = \frac{c\lambda(2R)}{R^{2/3}}\Big(-\frac{\rd}{\rd{t}}\Big|_{t=0}E[\rho^\#_t]\Big).
\end{equation}
Then Corollary \ref{cor_RCSS} gives
\begin{equation}\begin{split}
\frac{\rd}{\rd{t}}\Big|_{t=0}E(t) \le & -c\frac{\lambda(2R)^2}{R^2}\int_0^\infty x^2\mu(x)\rd{x} \le -c\frac{\lambda(2R)^3}{R^{8/3}}\Big(-\frac{\rd}{\rd{t}}\Big|_{t=0}E[\rho^\#_t]\Big) \\
\le & -c\frac{\lambda(2R)^3}{R^{8/3}}(E[\rho^\#_0]-E_\infty).
\end{split}\end{equation}}
\end{itemize}

Combining the two cases, we obtain
\rev{\begin{equation}
\frac{\rd}{\rd{t}}\Big|_{t=0}E(t)  \le -c\frac{\lambda(2R)^3}{R^{8/3}}(E[\rho^\#_0]-E_\infty).
\end{equation}}
Adding this with \eqref{prop_RCSS_1} gives the conclusion.

\end{proof}

\section{Proof of Theorem \ref{thm_main}}

Finally we give the proof of Theorem \ref{thm_main}. 

\begin{proof}

{\bf STEP 1}: reformulate the tightness.

Fix $R>0$ large enough, and then take $R_2>R$ large, to be chosen. Then Propositions \ref{prop_nrd} and \ref{prop_rd} give
\begin{equation}
\int_0^\infty\int_{R}^{R_2} \rho(t,x)\rd{x}\rd{t} \le CR_2^{\gamma-1},\quad \gamma:=2+\alpha+4\frac{\alpha^2}{\beta}+\frac{\alpha^3}{\beta^2},
\end{equation}
(where we need to apply Proposition \ref{prop_rd} on many subintervals of $[R,R_2]$, at most $R_2/R$ of them). Fix $\epsilon>0$ small enough, then
\begin{equation}
\Big|\{t:\int_{R}^{R_2} \rho(t,x)\rd{x} >\epsilon\}\Big|\le C\frac{R_2^{\gamma-1}}{\epsilon},
\end{equation}
which implies
\begin{equation}\label{set_t}
\Big|\{t:\int_{R}^\infty \rho(t,x)\rd{x} >\epsilon\}\Big|\le \frac{C}{\epsilon^{\gamma}},
\end{equation}
by Theorem \ref{thm_main1} with $R_2 = C/\epsilon$ (to control the integral on $[R_2,\infty)$). In the rest of this proof, we will fix the choice of $R$ and $R_2$ and ignore the dependence of constants on them, but keep track of the dependence on $\epsilon$.

{\bf STEP 2}:  energy decay estimate by curves from perturbed $\rho_{\ini}$.

Fix $t_0\ge 0$ with $\int_{R}^\infty \rho(t_0,x)\rd{x} \le \epsilon/2$. We will analyze the energy dissipation rate at $t_0$. To do this we may assume $t_0=0$ without loss of generality, and then $\rho(t_0,x)=\rho_{\ini}$. 

Consider the following density distribution supported on $[-R,R]$ with total mass 1:
\begin{equation}
\tilde{\rho}_{\ini} =  \frac{1}{1-\delta}\rho_{\ini}\chi_{[-R,R]},\quad \delta:=2\int_{R}^\infty \rho_{\ini}(x)\rd{x} \le \epsilon.
\end{equation}
Let $\tilde{\rho}_t$ be the generalized $h(s)$-linear curve defined in section \ref{sec_ghlc}, from $\tilde{\rho}_{\ini}$ to $\rho_\infty$. It is clear from its definition that $\tilde{\rho}_t$ is also supported on $[-R,R]$, since $R$ is large enough (so that $\supp\rho_\infty\subset [-R,R]$). Define the curve $\rho_t$ by
\begin{equation}
\rho_t = (1-\delta)\tilde{\rho}_t + \rho_{\ini}\chi_{(-\infty,-R]\cup[R,\infty)}.
\end{equation}

Lemma \ref{lem_RDper} gives
\begin{equation}
\frac{\rd}{\rd{t}}\Big|_{t=0} E[\tilde{\rho}_t] \le \frac{\rd}{\rd{t}}\Big|_{t=0} E[\tilde{\rho}^\#_t] + C\int_0^\infty x^2\tilde{\mu}(x)\rd{x},
\end{equation}
where $\tilde{\rho}^\#_t$ denotes the $h(s)$-linear curve from the Steiner symmetrization $\tilde{\rho}^\#$ of $\tilde{\rho}$ to $\rho_\infty$, and the $R$ dependence is ignored as announced before.
Then we compare $ E[\tilde{\rho}_t]$ and $ E[\rho_t]$:
\begin{equation}\label{finalErho1}
E[\tilde{\rho}_t] = \cS[\tilde{\rho}_t] + \cI[\tilde{\rho}_t],
\end{equation}
and
\begin{equation}\label{finalErho2}\begin{split}
E[\rho_t] = & (1-\delta)^m\cS[\tilde{\rho}_t] + (1-\delta)^2\cI[\tilde{\rho}_t] + E[\rho_{\ini}\chi_{(-\infty,-R]\cup[R,\infty)}] \\
& + \cI[\rho_{\ini}\chi_{(-\infty,-R]\cup[R,\infty)},(1-\delta)\tilde{\rho}_t] .
\end{split}\end{equation}

Taking $t$-derivatives gives
\begin{equation}\begin{split}
\frac{\rd}{\rd{t}}\Big|_{t=0}E[\rho_t] = & (1-\delta)^2\frac{\rd}{\rd{t}}\Big|_{t=0}E[\tilde{\rho}_t] + ((1-\delta)^m-(1-\delta)^2)\frac{\rd}{\rd{t}}\Big|_{t=0}\cS[\tilde{\rho}_t] \\
& + \frac{\rd}{\rd{t}}\Big|_{t=0}\cI[\rho_{\ini}\chi_{(-\infty,-R]\cup[R,\infty)},(1-\delta)\tilde{\rho}_t] .
\end{split}\end{equation}

We estimate the above RHS term by term. First,
\begin{equation}
\frac{\rd}{\rd{t}}\Big|_{t=0}\cS[\tilde{\rho}_t] = \frac{m}{m-1}\frac{\rd}{\rd{t}}\Big|_{t=0}\int_0^1 h_t(s)^{m-1}\rd{s} = m\int_0^1 h_t(s)^{m-2}(h_1(s)-h_0(s))\rd{s} \le C,
\end{equation}
by using the $h(s)$-linearity of the curve $\tilde{\rho}_t$. Next,
\begin{equation}\begin{split}
& \frac{\rd}{\rd{t}}\Big|_{t=0}\cI[\rho_{\ini}\chi_{(-\infty,-R]\cup[R,\infty)},(1-\delta)\tilde{\rho}_t] \\
= & -\int\int_{|y|\ge R} W(x-y)\partial_x ((1-\delta)\tilde{\rho}_{\ini}(x)v(x)) \rho_{\ini}(y)\rd{y}\rd{x} \\
= & (1-\delta)\int\int_{|y|\ge R} W'(x-y)\tilde{\rho}_{\ini}(x)v(x) \rho_{\ini}(y)\rd{y}\rd{x}, \\
\end{split}\end{equation}
where $v(x)$ is the velocity field of $\tilde{\rho}_t$ at $t=0$, given by Lemma \ref{lem_RDW}. Therefore
\begin{equation}\begin{split}
& \left|\frac{\rd}{\rd{t}}\Big|_{t=0}\cI[\rho_{\ini}\chi_{(-\infty,-R]\cup[R,\infty)},(1-\delta)\tilde{\rho}_t]\right| \\
\le & C\delta\int\tilde{\rho}_{\ini}(x)|v(x)|\rd{x} \\
\le & C\delta\Big(\int\tilde{\rho}_{\ini}(x)|v(x)|^2\rd{x}\Big)^{1/2} \\
\le & C\delta  \Big(\int_0^{1}\frac{(\partial_t h'(s))^2}{h'(s)^3h'_1(s)}\rd{s}\Big)^{1/2} \\
\le & C\delta  \Big(-\frac{\rd}{\rd{t}}\Big|_{t=0} E[\tilde{\rho}^\#_t]\Big)^{1/2},
\end{split}\end{equation}
by Lemma \ref{lem_RDW}, the choice of $M$ as in Lemma \ref{lem_RDper} (where the $R$ dependence is ignored), and \eqref{Ftau2}. Therefore we conclude
\begin{equation}\label{rrt}\begin{split}
\frac{\rd}{\rd{t}}\Big|_{t=0}E[\rho_t] \le & \frac{1}{2}\frac{\rd}{\rd{t}}\Big|_{t=0}E[\tilde{\rho}_t] + C\delta  \Big(-\frac{\rd}{\rd{t}}\Big|_{t=0} E[\tilde{\rho}^\#_t]\Big)^{1/2} \\
\le & \frac{1}{2}\frac{\rd}{\rd{t}}\Big|_{t=0} E[\tilde{\rho}^\#_t] + C\int_0^\infty x^2\tilde{\mu}(x)\rd{x} + C\delta  \Big(-\frac{\rd}{\rd{t}}\Big|_{t=0} E[\tilde{\rho}^\#_t]\Big)^{1/2},
\end{split}\end{equation}
for small $\delta>0$. This implies
\begin{equation}\label{final1}\begin{split}
& \frac{\rd}{\rd{t}}\Big|_{t=0}E[\rho(t,\cdot)]  \\
\le & -c\frac{\Big[-\frac{\rd}{\rd{t}}\Big|_{t=0} E[\tilde{\rho}^\#_t] - C \int_0^\infty x^2\tilde{\mu}(x)\rd{x} - C\delta  \Big(-\frac{\rd}{\rd{t}}\Big|_{t=0} E[\tilde{\rho}^\#_t]\Big)^{1/2}\Big]^2}{\int \frac{(\partial_t h'(s))^2}{h'(s)^3h'_1(s)}\rd{s}},
\end{split}\end{equation}
by Lemma \ref{lem_basic}, as long as the quantity in the bracket is positive.

A similar procedure applied to the RCSS curve from $\tilde{\rho}_{\ini}$ gives
\begin{equation}\label{final2}\begin{split}
& \frac{\rd}{\rd{t}}\Big|_{t=0}E[\rho(t,\cdot)]  
\le  -c\frac{\Big[-\int_0^\infty x^2\tilde{\mu}(x)\rd{x} - C\delta (\int_0^\infty x^2\tilde{\mu}(x)\rd{x})^{1/2}\Big]^2}{\int_0^\infty x^2\tilde{\mu}(x)\rd{x}},
\end{split}\end{equation}
as long as the quantity in the bracket is positive.

{\bf STEP 3}: make the energy decay estimate effective.

Using \eqref{final1} and \eqref{final2}, we separate into several cases according to the sizes of $\int_0^\infty x^2\tilde{\mu}(x)\rd{x}$ and $-\frac{\rd}{\rd{t}}\Big|_{t=0} E[\tilde{\rho}^\#_t]$:
\begin{itemize}

\item If $\int_0^\infty x^2\tilde{\mu}(x)\rd{x} \ge C\delta^2$, then the first term in the bracket in \eqref{final2} can absorb the second term, and gives
\begin{equation}\label{finalE1}
\quad \frac{\rd}{\rd{t}}\Big|_{t=0}E(t)  \le -c\int_0^\infty x^2\tilde{\mu}(x)\rd{x}.
\end{equation}
\begin{itemize}
\item If $-\frac{\rd}{\rd{t}}\Big|_{t=0} E[\tilde{\rho}^\#_t] \ge \max\{C\int_0^\infty x^2\tilde{\mu}(x)\rd{x}, C\delta^2\}$, then the first term in the bracket in \eqref{final1} can absorb the other two terms, and we get
\begin{equation}\label{finalE11}
\textbf{Case 1:}\quad \frac{\rd}{\rd{t}}\Big|_{t=0}E(t)  \le -c(E[\tilde{\rho}^\#_{\ini}]-E_\infty),
\end{equation}
as in the proof of Theorem \ref{thm_main2}.

\item If $-\frac{\rd}{\rd{t}}\Big|_{t=0} E[\tilde{\rho}^\#_t] < C\int_0^\infty x^2\tilde{\mu}(x)\rd{x}$, then \eqref{finalE1} gives
\begin{equation}\label{finalE12}
\textbf{Case 2:}\quad \frac{\rd}{\rd{t}}\Big|_{t=0}E(t)  \le -c\int_0^\infty x^2\tilde{\mu}(x)\rd{x} - c(-\frac{\rd}{\rd{t}}\Big|_{t=0} E[\tilde{\rho}^\#_t] ).
\end{equation}

\item If $-\frac{\rd}{\rd{t}}\Big|_{t=0} E[\tilde{\rho}^\#_t] <C\delta^2$, then \eqref{finalE1} gives
\begin{equation}\label{finalE13}
\textbf{Case 3:}\quad \frac{\rd}{\rd{t}}\Big|_{t=0}E(t)  \le -c\left(\int_0^\infty x^2\tilde{\mu}(x)\rd{x} + (-\frac{\rd}{\rd{t}}\Big|_{t=0} E[\tilde{\rho}^\#_t])\right) + C\delta^2.
\end{equation}

\end{itemize} 

\item If otherwise:
\begin{equation}\label{finalE20}
\int_0^\infty x^2\tilde{\mu}(x)\rd{x} < C\delta^2.
\end{equation}
\begin{itemize}
\item If $-\frac{\rd}{\rd{t}}\Big|_{t=0} E[\tilde{\rho}^\#_t] \ge \max\{C\int_0^\infty x^2\tilde{\mu}(x)\rd{x}, C\delta^2\} = C\delta^2$, then from \eqref{final1} and \eqref{finalE20} we get 
\begin{equation}\begin{split}\label{finalE21}
\textbf{Case 4:}\quad \frac{\rd}{\rd{t}}\Big|_{t=0}E(t)  \le & -c(-\frac{\rd}{\rd{t}}\Big|_{t=0} E[\tilde{\rho}^\#_t] ) \\
 \le & -c\left((-\frac{\rd}{\rd{t}}\Big|_{t=0} E[\tilde{\rho}^\#_t] ) - \int_0^\infty x^2\tilde{\mu}(x)\rd{x}\right) + C\delta^2 ,\\
\end{split}\end{equation}
as in the proof of Theorem \ref{thm_main2}.

%
\item If $-\frac{\rd}{\rd{t}}\Big|_{t=0} E[\tilde{\rho}^\#_t] <C\delta^2$, then adding with \eqref{finalE20} gives
\begin{equation}
\textbf{Case 5:}\quad -\frac{\rd}{\rd{t}}\Big|_{t=0} E[\tilde{\rho}^\#_t] + \int_0^\infty x^2\tilde{\mu}(x)\rd{x} \le C\delta^2.
\end{equation}

\end{itemize}

\end{itemize}

Combining with the facts
\begin{equation}
-\frac{\rd}{\rd{t}}\Big|_{t=0}E[\tilde{\rho}^\#_t] \ge c(E[\tilde{\rho}^\#_{\ini}]-E_\infty),
\end{equation}
from \eqref{Ftau3}, and
\begin{equation}
\int_0^\infty x^2\tilde{\mu}(x)\rd{x} \ge c(E[\tilde{\rho}_{\ini}]-E[\tilde{\rho}^\#_{\ini}]),
\end{equation}
from \eqref{Ftau3_1}, we see that the cases above reduce to 
\begin{equation}\label{finalcase1}
\frac{\rd}{\rd{t}}\Big|_{t=0}E(t)  \le -c(E[\tilde{\rho}_{\ini}]-E_\infty) + C\delta^2.
\end{equation}

Recall that $\delta\le\epsilon$. Let $T$ be the first time such that $E[\tilde{\rho}(t,\cdot)]-E_\infty \le C\epsilon$. Then for every time spot $t$ with $\int_{R}^\infty \rho(t,x)\rd{x} \le \epsilon$ and $0\le t \le T$,
\begin{equation}
\frac{\rd}{\rd{t}}E(t)  \le -c\epsilon + C\epsilon^2 \le -c\epsilon(1-C\epsilon ) \le -c\epsilon,
\end{equation}
for $\epsilon$ small enough. This gives an upper bound of the amount of such time spots:
\begin{equation}
T \le \frac{C}{\epsilon} + \frac{C}{\epsilon^{\gamma}}
\le \frac{C}{\epsilon^{\gamma}},
\end{equation}
where we also used \eqref{set_t}. 

Therefore there exists $t\in [0,\frac{C}{\epsilon^{\gamma}}]$ such that $\int_{R}^\infty \rho(t,x)\rd{x} \le \epsilon$ and $E[\tilde{\rho}(t,\cdot)]-E_\infty \le C\epsilon$.

{\bf STEP 4}: connect with $E(t)-E_\infty$.

Finally we estimate the difference between $E[\tilde{\rho}(t,\cdot)]$ and $E[\rho(t,\cdot)]$ (for $t$ with $\int_{R}^\infty \rho(t,x)\rd{x} \le \epsilon$) by using \eqref{finalErho1} and \eqref{finalErho2}:
\begin{equation}\begin{split}
E[\rho(t,\cdot)] - E[\tilde{\rho}(t,\cdot)] \le & E[\rho(t,\cdot)\chi_{(-\infty,-R]\cup[R,\infty)}]  + \cI[\rho(t,\cdot)\chi_{(-\infty,-R]\cup[R,\infty)},(1-\delta)\tilde{\rho}(t,\cdot)] \le C\epsilon, \\
\end{split}\end{equation}
by using $W(x) \le C|x|$ (from the $L^\infty$ bound of $W'$) and the $L^\infty$ bound of $\rho$. 

Therefore, combined with STEP 3, we get the following conclusion: there exists $t\in [0,\frac{C}{\epsilon^{\gamma}}]$ such that
\begin{equation}\begin{split}
E(t) - E_\infty  \le C\epsilon. \\
\end{split}\end{equation}
This finishes the proof since $\epsilon>0$ is arbitrary and $E(t)$ is non-increasing.

\end{proof}

\section{Conclusion}

In this paper we proved the equilibration of the 1D aggregation-diffusion equation \eqref{eq0} under certain assumptions. The first part of the proof gives a uniform bound on the first moment of the solution $\rho(t,\cdot)$, using various curves $\rho_t$ of density distributions and their energy decay rate estimates. The second part gives a quantitative energy dissipation rate estimate, using a combination of the methods in \cite{CHVY} and \cite{DYY}, with certain improvement, together with a perturbative argument as connection. This is the first time one could handle a a general class of weakly confining potentials $W(r)$, for which the tightness does not follow from $E(t)\le E(0)$ directly.

There are several related directions one could try in the future:
\begin{itemize}
\item Prove the equilibration for multi-dimensional aggregation diffusion equations for radially-symmetric solutions, with general weakly confining potentials. This might be accessible for $d=2$, where one could use an analog to \eqref{mphi} with $\Delta\phi = \chi_{5R_1\le |\bx| \le 6R_1}$ and try to prove that the time integral of $\rho$ in $\{5R_1\le |\bx| \le 6R_1\}$ is finite. Such $\phi$ grows like $\ln r$ as $r\rightarrow\infty$, providing tightness (although weaker than the 1D case). 

\item In the 1D case, remove the assumptions which lead to Lemma \ref{lem_ct}, i.e., symmetry and the critical threshold {\bf (A6)}. In this case one has to find new mechanism to rule out the situation where \emph{all} mass are escaping to either $\infty$ or $-\infty$.

\item Use the energy dissipation rate estimate to prove the equilibration for the second-order counterpart of \eqref{eq0}, i.e., isentropic Euler equations with pressure $p(\rho)=\rho^m$, with certain velocity damping mechanism. One might need to combined the current methods with hypocoercivity arguments, as was done in~\cite{ST1,ST2}.

\end{itemize}

\section*{Acknowledgement}

The author would like to thank Yao Yao for inspiring discussions which initiated this research project. The author would also like to thank Jacob Bedrossian, Jos\'e Carrillo, Jingwei Hu, Eitan Tadmor, Changhui Tan, Xukai Yan and Ruixiang Zhang for helpful discussions.

\begin{appendix}
\setcounter{section}{0}
\setcounter{theorem}{0}
\setcounter{equation}{0}
\renewcommand{\thetheorem}{\thesection.\arabic{theorem}}
\section*{Appendix: subsequential convergence of $\rho(t,\cdot)$ to $\rho_\infty$ }
\renewcommand{\thesection}{A} 
     \renewcommand{\thelemma}{\Alph{section}\arabic{lemma}}
     
\begin{proposition}
For any $\epsilon_0>0$, we have
\begin{equation}\label{rhoconv0}
\left|\Big\{t: \|\rho(t,\cdot)-\rho_\infty\|_{H^{-1}(\mathbb{R})} > \epsilon_0\Big\}\right| \le \frac{C}{\epsilon_0^{4\gamma}}.
\end{equation}
In particular, there exists a subsequence $\{\rho(t_n,\cdot)\}$ with $\lim_{n\rightarrow\infty}t_n=\infty$  converges to $\rho_\infty$ in $H^{-1}$ as $n\rightarrow\infty$, with
\begin{equation}\label{rhoconv1}
\|\rho(t_n,\cdot)-\rho_\infty\|_{H^{-1}(\mathbb{R})} \le C(1+t_n)^{-\frac{1}{4\gamma}}.
\end{equation}
\end{proposition}

\begin{proof}

Taking the same $R$ as in the proof of Theorem \ref{thm_main}, we have \eqref{set_t} for any $\epsilon>0$. Take $\epsilon=\epsilon_0^4$. Then, combined with Theorem \ref{thm_main},
\begin{equation}
\Big|\Big\{t:\int_{R}^\infty \rho(t,x)\rd{x} >\epsilon \text{ or }E[\rho]-E_\infty > \epsilon \Big\}\Big|\le \frac{C}{\epsilon^{\gamma}}.
\end{equation}
Since $R$ is fixed, we will omit any $R$-dependence in the rest of this proof. Consider a fixed $t$ with 
\begin{equation}\label{condrhoeps}
\int_R^\infty \rho(t,x)\rd{x}\le \epsilon \text{ and }E[\rho]-E_\infty \le \epsilon,
\end{equation}
and we aim to show that 
\begin{equation}
\|\rho(t,\cdot)-\rho_\infty\|_{H^{-1}(\mathbb{R})} \le C\epsilon_0 = C\epsilon^{1/4},
\end{equation}
and the conclusion follows, since those $t$ not satisfying \eqref{condrhoeps} has total measure bounded by $\frac{C}{\epsilon^{\gamma}}=\frac{C}{\epsilon_0^{4\gamma}}$.  We will omit the $t$-dependence of $\rho$ in the rest of this proof.

For any test function $\psi\in H^1(\mathbb{R})$ with $\|\psi\|_{H^1(\mathbb{R})}=1$, we need to show that 
\begin{equation}\label{rhopsi}
\Big|\int \rho(x)\psi(x)\rd{x} - \int \rho_\infty(x)\psi(x)\rd{x}\Big| \le C\epsilon^{1/4}.
\end{equation}
First notice that $\|\psi\|_{L^\infty(\mathbb{R})}\le C\|\psi\|_{H^1(\mathbb{R})}=C$.

Denote
\begin{equation}
\delta := \int_{|x|\ge R}\rho(x)\rd{x} \le 2\epsilon,\quad \forall t\ge 0,
\end{equation}
and define
\begin{equation}
\tilde{\rho} =  \frac{1}{1-\delta}\rho\chi_{[-R,R]},
\end{equation}
which has total mass 1, and denote $\tilde{\rho}^\sharp$ as the Steiner symmetrization of $\tilde{\rho}$.

We first claim that 
\begin{equation}
E[\tilde{\rho}] \le (1+C\epsilon)E[\rho].
\end{equation}
In fact, since $W(x)$ is bounded below by $W(0)=0$,
\begin{equation}
\cI[\tilde{\rho}] = (\frac{1}{1-\delta})^2 \cI[\rho\chi_{[-R,R]}] \le (\frac{1}{1-\delta})^2 \cI[\rho],\quad \cS[\tilde{\rho}] \le (\frac{1}{1-\delta})^m \cS[\rho],
\end{equation}
and the claim follows since $\epsilon$ is small (so is $\delta$). Since $E[\rho]$ is bounded above by $E[\rho_{\ini}]$, we obtain
\begin{equation}\label{Etrho}
E[\tilde{\rho}] \le E[\rho] + C\epsilon.
\end{equation}

Then we decompose
\begin{equation}\begin{split}
\Big| & \int \rho(x)\psi(x)\rd{x}-\int \rho_\infty(x)\psi(x)\rd{x}\Big| \le \Big|\int \rho(x)\psi(x)\rd{x}-\int \tilde{\rho}(x)\psi(x)\rd{x}\Big| \\ & + \Big|\int \tilde{\rho}(x)\psi(x)\rd{x}-\int \tilde{\rho}^\sharp(x)\psi(x)\rd{x}\Big| + \Big|\int \tilde{\rho}^\sharp(x)\psi(x)\rd{x}-\int \rho_\infty(x)\psi(x)\rd{x}\Big| \\
= & I_1+I_2+I_3,
\end{split}\end{equation}
and it is clear that $I_1\le C\epsilon$ since $\|\psi\|_{L^\infty(\mathbb{R})}\le C$.

To estimate $I_2$, define $\tilde{\rho}_t$ as the RCSS curve starting from $\rho$. If this curve ends at $\tilde{\rho}^\sharp$ within finite time, then let $T$ be the final time of this curve, i.e., $\tilde{\rho}_T = \tilde{\rho}^\sharp$. Denote $\mu_t$ as the non-radially-decreasing part of $\tilde{\rho}_t$, and $v_t$ as the velocity field of this curve. Then
\begin{equation}\begin{split}
I_2 = & \Big|\int_0^T \int \partial_t \tilde{\rho}_t \psi(x)\rd{x}\rd{t}\Big| \\
= & \Big|\int_0^T \int \partial_x(\tilde{\rho}_t v_t) \psi(x)\rd{x}\rd{t}\Big| \\
= & \Big|\int_0^T \int \tilde{\rho}_t v_t \partial_x\psi(x)\rd{x}\rd{t}\Big| \\
\le & \int_0^T \|\partial_x \psi\|_{L^2(\tilde{\rho}_t\rd{x})}  \|v_t\|_{L^2(\tilde{\rho}_t\rd{x})} \rd{t}.
\end{split}\end{equation}
Notice that 
\begin{equation}
\|\partial_x \psi\|_{L^2(\tilde{\rho}_t\rd{x})} \le \|\tilde{\rho}_t\|_{L^\infty}\|\partial_x \psi\|_{L^2} \le C,
\end{equation}
and $\|v_t\|_{L^2(\rho_t\rd{x})}$ is the cost of the curve, bounded above by $\Big(\int x^2\mu_t(x)\rd{x}\Big)^{1/2}$ by \eqref{RCSScost}. Therefore
\begin{equation}
I_2 \le C\int_0^T\Big(\int x^2\mu_t(x)\rd{x}\Big)^{1/2}\rd{t}.
\end{equation}
Lemma \ref{lem_RCSS}, applied to every $\tilde{\rho}_t$, shows 
\begin{equation}
E[\tilde{\rho}]-E[\tilde{\rho}^\sharp] \ge c \int_0^T\int x^2\mu_t(x)\rd{x}\rd{t},
\end{equation}
(where $R$-dependence is omitted). By the definition of the RCSS curve, the center $c_j(t)$ is bounded by $c_j(t)\le Ce^{-t}$. Therefore we have $\int x^2\mu_t(x)\rd{x} \le Ce^{-t}$, which implies
\begin{equation}
\int_0^T\Big(\int x^2\mu_t(x)\rd{x}\Big)^{1/2}\rd{t} \le C\int_0^T\Big(\int x^2\mu_t(x)\rd{x}\Big)^{\frac{1}{4}}e^{-\frac{t}{4}}\rd{t} \le C\Big(\int_0^T\int x^2\mu_t(x)\rd{x}\rd{t}\Big)^{1/4}.
\end{equation}

 Therefore, combined with \eqref{Etrho}, we obtain
\begin{equation}
I_2 \le C(E[\tilde{\rho}]-E[\tilde{\rho}^\sharp])^{1/4} \le C(E[\rho]-E_\infty+\epsilon)^{1/4}.
\end{equation}

In the case when the curve $\tilde{\rho}_t$ lasts for infinite time, the same argument gives
\begin{equation}
\Big|\int \tilde{\rho}(x)\psi(x)\rd{x}-\int \tilde{\rho}_T(x)\psi(x)\rd{x}\Big| \le C(E[\rho]-E_\infty+\epsilon)^{1/4},
\end{equation}
for any $T>0$. By the definition of the RCSS curve, the velocity field $v_t$ satisfies $\|v_t\|_{L^\infty} \le CRe^{-t}$. Therefore $\tilde{\rho}_T$ converges to $\tilde{\rho}^\sharp$ as $T\rightarrow\infty$ in the sense of Wasserstein-$\infty$. Therefore 
\begin{equation}
\lim_{T\rightarrow\infty}\int \tilde{\rho}_T(x)\psi(x)\rd{x} = \int \tilde{\rho}^\sharp(x)\psi(x)\rd{x},
\end{equation}
since $\psi$ is continuous on $[-R,R]$ and every $\tilde{\rho}_T$ is supported on $[-R,R]$. Therefore we get the same estimate for $I_2$ as before.

To estimate $I_3$, let $\tilde{\rho}^\sharp_t,\,t\in [0,1]$ be the $h(s)$-linear curve connecting $\tilde{\rho}^\sharp$ and $\rho_\infty$, and denote $\tilde{v}_t$ as its transport velocity field. By Lemma \ref{lem_RDv}, we obtain (recall that $\partial_t h'(s) = h'_1(s)-h'_0(s)$)
\begin{equation}\label{I3long}\begin{split}
I_3 = & \Big|\int_0^1 \int \partial_t \tilde{\rho}^\sharp_t \psi(x)\rd{x}\rd{t}\Big| \\
= & \Big|\int_0^1 \int \tilde{\rho}^\sharp_t \tilde{v}_t \partial_x\psi(x)\rd{x}\rd{t}\Big| \\
\le & \int_0^1 \|\partial_x \psi\|_{L^2(\tilde{\rho}^\sharp_t\rd{x})}  \|\tilde{v}_t\|_{L^2(\tilde{\rho}^\sharp_t\rd{x})} \rd{t} \\
\le & C\int_0^1\Big(\int_0^{1} \frac{(\partial_t h'(s))^2}{h'_t(s)^3h'_1(s)}\rd{s}\Big)^{1/2}\rd{t} \\
\le & C\Big(\int_0^1\int_0^{1} \frac{(\partial_t h'(s))^2}{h'_t(s)^3h'_1(s)}\rd{s}\rd{t}\Big)^{1/2} \\
= & C\Big(\int_0^1\int_0^{1} \frac{(\partial_t h'(s))^2}{(h'_0(s)+t\partial_t h'(s))^3h'_1(s)}\rd{t}\rd{s}\Big)^{1/2}\\ 
= & C\Big(\int_0^1 \frac{\partial_t h'(s)}{h'_1(s)}(\frac{1}{h'_0(s)^2}-\frac{1}{h'_1(s)^2})\rd{s}\Big)^{1/2} \\
= & C\Big(\int_0^1 \frac{(\partial_t h'(s))^2}{h'_0(s)^2h'_1(s)^3}(h'_0(s)+h'_1(s))\rd{s}\Big)^{1/2}, \\
\end{split}\end{equation}
similar as in the estimate of $I_2$. Then, using Lemma \ref{lem_RD1} and \eqref{Ftau}, we compute 
\begin{equation}\begin{split}
E[\tilde{\rho}^\sharp]-E_\infty \ge & \int_0^1 t\int_0^1 \frac{(\partial_t h'(s))^2}{(h'_t(s))^4} \rd{s} \rd{t} \\
= & \int_0^1\int_0^1 \frac{(\partial_t h'(s))^2}{(h'_0(s)+t\partial_t h'(s))^4}t\rd{t}  \rd{s} \\
= & \int_0^1\int_0^1 \frac{\partial_t h'(s)}{(h'_0(s)+t\partial_t h'(s))^3}\rd{t}  \rd{s} - \int_0^1\int_0^1 \frac{\partial_t h'(s)h'_0(s)}{(h'_0(s)+t\partial_t h'(s))^4}\rd{t}  \rd{s} \\
= & \frac{1}{2}\int_0^1 (\frac{1}{h'_0(s)^2}-\frac{1}{h'_1(s)^2}) \rd{s} -  \frac{1}{3}\int_0^1 h'_0(s)(\frac{1}{h'_0(s)^3}-\frac{1}{h'_1(s)^3}) \rd{s} \\
= & \frac{1}{2}\int_0^1 \frac{\partial_t h(s)}{h'_0(s)^2h'_1(s)^2}(h'_0(s)+h'_1(s)) \rd{s} \\
& -  \frac{1}{3}\int_0^1 \frac{\partial_t h(s)}{h'_0(s)^2h'_1(s)^3}(h'_0(s)^2+h'_0(s)h'_1(s)+h'_1(s)^2) \rd{s} \\
= &  \frac{1}{6}\int_0^1 \frac{\partial_t h(s)}{h'_0(s)^2h'_1(s)^3}(-2h'_0(s)^2+h'_0(s)h'_1(s)+h'_1(s)^2) \rd{s} \\
= &  \frac{1}{6}\int_0^1 \frac{(\partial_t h(s))^2}{h'_0(s)^2h'_1(s)^3}(h'_1(s)+2h'_0(s)) \rd{s}. \\
\end{split}\end{equation}
Compared with \eqref{I3long}, we get
\begin{equation}
I_3 \le C(E[\tilde{\rho}^\sharp]-E_\infty)^{1/2} \le C(E[\rho]-E_\infty+\epsilon)^{1/2}.
\end{equation}

Combined the estimates for $I_1,I_2,I_3$, we obtain
\begin{equation}\begin{split}
\Big|  \int \rho(x)\psi(x)\rd{x}-\int \rho_\infty(x)\psi(x)\rd{x}\Big|  \le C(E[\rho]-E_\infty+\epsilon)^{1/4}.
\end{split}\end{equation}
Noticing that $E[\rho]-E_\infty\le \epsilon$ by assumption, we obtain \eqref{rhopsi}.

Finally, to see \eqref{rhoconv1}, we notice that \eqref{rhoconv0} with $\epsilon_0=2^{-n}$ implies that there exists $t_n$ with $2^{4\gamma n} \le t_n \le C2^{4\gamma n}$ such that $\|\rho(t,\cdot)-\rho_\infty\|_{H^{-1}(\mathbb{R})}\le 2^{-n} \le C(1+t_n)^{-\frac{1}{4\gamma}}$. For this sequence of $t_n$ (which satisfies $\lim_{n\rightarrow\infty}t_n=\infty$), we obtain \eqref{rhoconv1}.

\end{proof}

\end{appendix}

\bibliographystyle{plain}
\bibliography{radial_convergence_bib}

\begin{thebibliography}{10}

\bibitem{AGS}
L.~Ambrosio, N.~Gigli, and G.~Savar\'e.
\newblock {\em Gradient flows: in metric spaces and in the space of probability
  measures}.
\newblock Springer Science \& Business Media, 2008.

\bibitem{BCH}
R.~Bailo, J.~A. Carrillo, and J.~Hu.
\newblock Fully discrete positivity-preserving and energy-dissipating schemes
  for aggregation-diffusion equations with a gradient flow structure.
\newblock {\em preprint, arXiv:1811.11502}.

\bibitem{Bed}
J.~Bedrossian.
\newblock Global minimizers for free energies of subcritical aggregation
  equations with degenerate diffusion.
\newblock {\em Appl. Math. Letters}, 24(11):1927--1932, 2011.

\bibitem{BB}
J.-D. Benamou and Y.~Brenier.
\newblock A computational fluid mechanics solution to the {M}onge-{K}antorovich
  mass transfer problem.
\newblock {\em Numerische Mathematik}, 84(3):375--393, 2000.

\bibitem{BCC}
A.~Blanchet, V.~Calvez, and J.~A. Carrillo.
\newblock Convergence of the mass-transport steepest descent scheme for the
  subcritical {P}atlak-{K}eller-{S}egel model.
\newblock {\em SIAM J. Numer. Anal.}, 46:691--721, 2008.

\bibitem{BCL}
A.~Blanchet, J.~A. Carrillo, and P.~Lauren\c{c}ot.
\newblock Critical mass for a {P}atlak-{K}eller-{S}egel model with degenerate
  diffusion in higher dimensions.
\newblock {\em Calc. Var. Partial Differential Equations}, 35:133--168, 2009.

\bibitem{BDP}
A.~Blanchet, J.~Dolbeault, and B.~Perthame.
\newblock Two-dimensional {K}eller-{S}egel model: optimal critical mass and
  qualitative properties of the solutions.
\newblock {\em Electronic J. Differ. Equat.}, (44):1--32, 2006.

\bibitem{BCM}
S.~Boi, V.~Capasso, and D.~Morale.
\newblock Modeling the aggragative behavior of ants of the species polyergus
  rufescens.
\newblock {\em Nonlinear Anal. Real World Appl.}, 1(163-176), 2000.

\bibitem{BurCM}
M.~Burger, V.~Capasso, and D.~Morale.
\newblock On an aggregation model with long and short range interactions.
\newblock {\em Nonlinear Anal. Real World Appl.}, 8:939--958, 2007.

\bibitem{CCH2}
V.~Calvez, J.~A. Carrillo, and F.~Hoffmann.
\newblock Uniqueness of stationary states for singular {K}eller-{S}egel type
  models.
\newblock {\em preprint, arXiv:1905.07788}.

\bibitem{CCH}
V.~Calvez, J.~A. Carrillo, and F.~Hoffmann.
\newblock Equilibria of homogeneous functionals in the fair-competition regime.
\newblock {\em Nonlinear Anal.}, 159:85--128, 2017.

\bibitem{CC}
V.~Calvez and L.~Corrias.
\newblock The parabolic-parabolic {K}eller-{S}egel model in $\mathbb{R}^2$.
\newblock {\em Comm. Math. Sci.}, 6(2):417--447, 2008.

\bibitem{CCV}
J.~A. Carrillo, D.~Castorina, and B.~Volzone.
\newblock Ground states for diffusion dominated free energies with logarithmic
  interaction.
\newblock {\em SIAM J. Math. Anal.}, 47(1):1--25, 2015.

\bibitem{CCWW}
J.~A. Carrillo, K.~Craig, L.~Wang, and C.~Wei.
\newblock Primal dual methods for {W}asserstein gradient flows.
\newblock {\em preprint, arXiv:1901.08081}.

\bibitem{CHVY}
J.~A. Carrillo, S.~Hittmeir, B.~Volzone, and Y.~Yao.
\newblock Nonlinear aggregation-diffusion equations: radial symmetry and long
  time asymptotics.
\newblock {\em Invent. Math.}, 218(3):889--977, 2019.

\bibitem{CHMV}
J.~A. Carrillo, F.~Hoffmann, E.~Mainini, and B.~Volzone.
\newblock Ground states in the diffusion-dominated regime.
\newblock {\em Calc. Var. Partial Differential Equations}, 57(5):127, 2018.

\bibitem{CMV1}
J.~A. Carrillo, R.~J. McCann, and C.~Villani.
\newblock Kinetic equilibration rates for granular media and related equations:
  entropy dissipation and mass transportation estimates.
\newblock {\em Rev. Mathm\'atica Iberoamericana}, 19:1--48, 2003.

\bibitem{CMV2}
J.~A. Carrillo, R.~J. McCann, and C.~Villani.
\newblock Contractions in the 2-{W}asserstein length space and thermaliztion of
  granular media.
\newblock {\em Arch. Ration. Mech. Anal.}, 179:217--263, 2006.

\bibitem{CHKK}
Y.~S. Chung, S.~Hwang, K.~Kang, and J.~Kim.
\newblock {H}\"older continuity of {K}eller-{S}egel equation of porous medium
  type coupled to fluid equations.
\newblock {\em J. Differ. Equat.}, 263:2157--2212, 2017.

\bibitem{CS1}
F.~Cucker and S.~Smale.
\newblock Emergent behavior in flocks.
\newblock {\em IEEE Trans. Automat. Control}, 52(5):852--862, 2007.

\bibitem{CS2}
F.~Cucker and S.~Smale.
\newblock On the mathematics of emergence.
\newblock {\em Jpn. J. Math.}, 2(1):197--227, 2007.

\bibitem{DYY}
M.~G. Delgadino, X.~Yan, and Y.~Yao.
\newblock Uniqueness and non-uniqueness of steady states of
  aggregation-diffusion equations.
\newblock {\em preprint, arXiv:1908.09782}, 2019.

\bibitem{HT}
S.-Y. Ha and E.~Tadmor.
\newblock From particle to kinetic and hydrodynamic descriptions of flocking.
\newblock {\em Kinetic and Related Models}, 1(3):415--435, 2008.

\bibitem{Hor}
D.~Horstmann.
\newblock From 1970 until present: the {K}eller-{S}egel model in chemotaxis and
  its consequences.
\newblock {\em preprint}, 2003.

\bibitem{HZ}
S.~Hwang and Y.~P. Zhang.
\newblock Continuity results for degenerate diffusion equations with
  $l_t^pl_x^q$ drifts.
\newblock {\em preprint, arXiv:1906.04961}.

\bibitem{JL}
W.~J\"ager and S.~Luckhaus.
\newblock On explosions of solutions to a system of partial differential
  equations modelling chemotaxis.
\newblock {\em Trans. Amer. Math. Soc.}, 329:819--824, 1992.

\bibitem{KS}
E.~F. Keller and L.~A. Segel.
\newblock Initiation of slide mold aggregation viewed as an instability.
\newblock {\em J. Theor. Biol.}, 26:399--415, 1970.

\bibitem{KY}
I.~Kim and Y.~Yao.
\newblock The {P}atlak--{K}eller--{S}egel model and its variations: properties
  of solutions via maximum principle.
\newblock {\em SIAM J. Math. Anal.}, 44(2):568--602, 2012.

\bibitem{KZ}
I.~Kim and Y.~P. Zhang.
\newblock Regularity properties of degenerate diffusion equations with drifts.
\newblock {\em SIAM J. Math. Anal.}, 50(4):4371--4406, 2018.

\bibitem{LY}
E.~H. Lieb and H.~T. Yau.
\newblock The {C}handrasekhar theory of stellar collapse as the limit of
  quantum mechanics.
\newblock {\em Comm. Math. Phys.}, 112(1):147--174, 1987.

\bibitem{MEK1}
A.~Mogilner and L.~Edelstein-Keshet.
\newblock A non-local model for a swarm.
\newblock {\em J. Math. Biol.}, 38:534--570, 1999.

\bibitem{MEK2}
A.~Mogilner, L.~Edelstein-Keshet, L.~Bent, and A.~Spiros.
\newblock Mutual interactions, potentials, and individual distance in a social
  aggregation.
\newblock {\em J. Math. Biol.}, 47:353--389, 2003.

\bibitem{MCO}
D.~Morale, V.~Capasso, and K.~Oelschl\"ager.
\newblock An interacting particle system modelling aggregation behavior: from
  individuals to populations.
\newblock {\em J. Math. Biol.}, 50:49--66, 2005.

\bibitem{MT}
S.~Motsch and E.~Tadmor.
\newblock Heterophilious dynamics enhances consensus.
\newblock {\em SIAM Review}, 56(4):577--621, 2014.

\bibitem{Oel}
K.~Oelschl\"ager.
\newblock Large systems of interacting particles and the porous medium
  equation.
\newblock {\em J. Diff. Equat.}, 88(2):294--346, 1990.

\bibitem{Pat}
C.~S. Patlak.
\newblock Random walk with persistence and external bias.
\newblock {\em Bull. Math. Biophys.}, 15:311--338, 1953.

\bibitem{ST2}
R.~Shu and E.~Tadmor.
\newblock Anticipation breeds alignment.
\newblock {\em preprint, arXiv:1905.00633}.

\bibitem{ST1}
R.~Shu and E.~Tadmor.
\newblock Flocking hydrodynamics with external potentials.
\newblock {\em preprint, arXiv:1901.07099}.

\bibitem{Str}
G.~Str\"ohmer.
\newblock Stationary states and moving planes.
\newblock {\em Parabolic and {N}avier-{S}tokes equations}, 81(Part 2):501--513,
  2008.

\bibitem{Sug}
Y.~Sugiyama.
\newblock The global existence and asymptotic behavior of solutions to
  degenerate quasi-linear parabolic systems of chemotaxis.
\newblock {\em Differ. Integral Equat.}, 20:133--180, 2007.

\bibitem{TBL}
C.~M. Topaz, A.~L. Bertozzi, and M.~A. Lewis.
\newblock A nonlocal continuum model for biological aggregation.
\newblock {\em Bull. Math. Biol.}, 68:1601--1623, 2006.

\end{thebibliography}

\end{document}